\documentclass[11pt,a4paper,twoside]{book}

\usepackage[T1]{fontenc}
\usepackage[latin1]{inputenc}
\usepackage[ngerman,english]{babel}
\usepackage{lmodern}

\usepackage{emptypage}

\usepackage[hmargin=2.5cm,vmargin=2.5cm]{geometry} 

\usepackage{amsmath,amssymb,amsthm,mathrsfs,mathtools,bbm}
\usepackage{nicefrac}

\usepackage{enumitem}
\setlist[enumerate]{label=(\roman*)}

\usepackage{graphicx}
\graphicspath{{Images/}}
\usepackage{subfigure}

\usepackage{array}
\usepackage{booktabs}
\usepackage{longtable}

\usepackage{algorithm}
\usepackage{algpseudocode}
\algrenewcommand\algorithmicrequire{\textbf{Input:}}
\algrenewcommand\algorithmicensure{\textbf{Output:}}

\usepackage{csquotes}
\usepackage[backend=biber, style=ieee, sorting=nyt, block=space, dashed=false]{biblatex}
\usepackage{etoolbox}
\preto\section{%
  \ifnum\value{section}=0\addtocontents{toc}{\vskip11pt}\fi
}

\makeatletter
\pretocmd{\chapter}{\addtocontents{toc}{\protect\addvspace{15\p@}}}{}{}
\pretocmd{\section}{\addtocontents{toc}{\protect\addvspace{2.5\p@}}}{}{}
\makeatother

\usepackage[linktocpage=true,pdfpagemode=UseNone]{hyperref}

\hypersetup{pdftitle={Computer Tomography},pdfauthor={Matthias Beckmann}}

\newcounter{mycounter}[chapter]
\numberwithin{mycounter}{chapter}
\numberwithin{equation}{chapter}

\theoremstyle{plain}

\newtheorem{lemma}[mycounter]{Lemma}
\newtheorem{theorem}[mycounter]{Theorem}
\newtheorem{proposition}[mycounter]{Proposition}
\newtheorem{corollary}[mycounter]{Corollary}
\newtheorem{definition}[mycounter]{Definition}
\newtheorem{remark}[mycounter]{Remark}

\newtheorem{example}[mycounter]{Example}

\newtheorem{problem}[mycounter]{Problem}
\newtheorem{observation}[mycounter]{Observation}

\def\N			{\mathbb N}
\def\Z			{\mathbb Z}
\def\Q			{\mathbb Q}
\def\R			{\mathbb R}
\def\C			{\mathbb C}
\def\Sphere		{\mathbb S}

\def\A			{\mathcal A}
\def\E			{\mathcal E}
\def\Q			{\mathcal Q}
\def\T			{\mathcal T}

\def\Radon		{\mathcal R}
\def\Back		{\mathcal B}
\def\Fourier	{\mathcal F}
\def\Int		{\mathcal I}
\def\Fan		{\mathcal D}
\def\Xray		{\mathcal P}
\def\Xback		{\mathcal K}

\def\Ell		{\mathcal L}

\def\Cont		{\mathscr C}

\def\Test		{\mathscr D}

\def\L			{\mathrm L}
\def\H			{\mathrm H}
\def\Schwartz	{\mathcal S}

\def\X			{\mathcal X}
\def\Y			{\mathcal Y}

\def\D			{\mathrm D}
\def\d			{\mathrm d}
\def\e			{\mathrm e}
\def\i			{\mathrm i}

\def\bfn		{\mathbf n}
\def\bfx		{\mathbf x}
\def\bfy		{\mathbf y}

\def\loc		{\mathrm{loc}}

\def\FBP		{\mathrm{FBP}}

\DeclareMathOperator{\sinc}{sinc}
\DeclareMathOperator{\atan}{atan}
\DeclareMathOperator{\supp}{supp}
\DeclareMathOperator{\sgn}{sgn}
\DeclareMathOperator{\diam}{diam}
\DeclareMathOperator{\im}{im}
\DeclareMathOperator{\Id}{Id}
\DeclareMathOperator{\rank}{rank}
\DeclareMathOperator{\Length}{Length}

\DeclareMathOperator*{\argmin}{arg\,min}

\def\seq		{\subseteq}

\newcommand{\HRule}{\rule{.9\linewidth}{.6pt}}

\makeatletter
\providecommand{\bigsqcap}{\mathop{\mathpalette\@updown\bigsqcup}}
\newcommand*{\@updown}[2]{\rotatebox[origin=c]{180}{$\m@th#1#2$}}
\makeatother

\newcommand*{\rect}[1]{\mathchoice{\textstyle\bigsqcap_{#1}\displaystyle}{\textstyle\bigsqcap_{#1}}{\scriptstyle\bigsqcap_{#1}}{\scriptscriptstyle\bigsqcap_{#1}}}

\newcommand{\tendsto}{\longrightarrow}

\begin{document}
%
\pagestyle{myheadings}
\renewcommand{\chaptermark}[1]{\markboth{\thechapter\ #1}{\thechapter\ #1}}
\renewcommand{\sectionmark}[1]{\markright{\thesection\ #1}}

\makeatletter
\let\ps@plain\ps@empty
\makeatother

\pagenumbering{Roman}


\begin{titlepage}
\begin{center}

\vspace*{.1\textheight}
{\scshape\LARGE Lecture notes\par}

\vspace{1.5cm}

\HRule \\[0.5cm]
{\huge \bfseries Computer Tomography \par}\vspace{0.5cm}
\HRule \\[1.5cm]

\vfill

{\Large Matthias Beckmann} \\[0.5cm]
{\large \href{mailto:info@mbeckmann.de}{info@mbeckmann.de} \\ \url{https://www.mbeckmann.de/}}

\vspace*{.05\textheight}
\end{center}
\end{titlepage}

\newpage
\thispagestyle{empty}
\setcounter{page}{2}

\phantom{.}

\vfill

\begin{center}
--- December 04, 2023 ---
\end{center}

\frontmatter

\chapter{Preface}
\markboth{Preface}{Preface}

The development of {\em computerized tomography} (CT) has revolutionized the field of diagnostic radiology and CT is by now one of the standard modalities in medical imaging.
Its goal consists in imaging the interior structure of a scanned object by measuring and processing the attenuation of X-rays along a large number of lines through the cross-sections of the object to be examined.
In this process, a fundamental feature of CT is the mathematical reconstruction of an image by the application of a suitable and sophisticated algorithm.

The impact of CT in diagnostic medicine has been revolutionary, since it has provided a non-invasive imaging modality and has enabled doctors to view internal organs with mould-breaking precision and safety for the patient.
Since the invention of the first CT scanner in the 1970s the number of CT scans for diagnostic purpose has been growing extensively.
In addition, there are numerous non-medical imaging applications which are also based on the methods of computerized tomography.
One example is non-destructive testing (NDT) in materials science, where we want to evaluate the properties of a material without causing damage.
Another application is electron microscopy, which is a typical example for an incomplete data problem, because only observations in a limited angular range are available.

Mathematically, an X-ray scan provides the line integral values of the object's attenuation function along lines in the plane.
Hence, the CT reconstruction problem requires the recovery of a bivariate function $f: \R^2 \to \R$ from the knowledge of its line integrals
\begin{equation*}
\Radon f(t,\theta) = \int_{\{x \cos(\theta) + y \sin(\theta) = t\}} f(x,y) \: \d (x,y)
\quad \mbox{ for } (t,\theta) \in \R \times [0,\pi).
\end{equation*}
The purely mathematical problem of reconstructing a function from its line integral values was first studied and analytically solved by the Austrian mathematician J. Radon in 1917 in his pioneering paper ``Über die Bestimmung von Funktionen durch ihre Integralwerte längs gewisser Mannigfaltigkeiten'', cf.~\cite{Radon1917}.
In that work, Radon derived an explicit inversion formula for the linear integral transform
\begin{equation*}
\Radon: f \longmapsto \Radon f
\end{equation*}
under the assumption that the data $\Radon f(t,\theta)$ is complete, i.e., available for all possible values $(t,\theta) \in \R \times [0,\pi)$.
In his honour, the operator $\Radon$ is now known as the {\em Radon transform} and the corresponding integral values are called {\em Radon data}.

Historically, the foundation of CT was laid in 1895 by the German physicist W. Röntgen, who discovered a new kind of radiation, which he called {\em X-radiation} to emphasize its unknown type.
Immediately after the discovery, X-rays have been used to image the interior of the human body.
In 1901 his achievements earned Röntgen the first Nobel Prize in Physics.
The two pioneering scientists who were primarily responsible for the development of CT in the 1960s and 1970s were A. Cormack and G. Hounsfield.
With their work, the hitherto purely mathematical problem of reconstructing a bivariate function from the knowledge of its Radon transform has finally become relevant for practical applications.
In~\cite{Cormack1963, Cormack1964}, Cormack developed mathematical algorithms to create an image from X-ray scans.
At about the same time, but working completely independently of Cormack, Hounsfield designed the first operational CT scanner as well as the first commercially available model, see~\cite{Hounsfield1973}.
In 1979 the Nobel Prize for Medicine and Physiology was jointly awarded to Cormack and Hounsfield for their fundamental achievements.

\bigbreak

This lecture gives an introduction to the mathematics of computer(ized) tomography (CT).
We will discuss the principle of X-ray tomography, the mathematical model for the measurement process, the underlying mathematical reconstruction problem and some classical reconstruction techniques.
Since tomography is a typical example of an inverse problem, we will also discuss the ill-posedness of the CT reconstruction problem.
The following topics will be addressed:
\begin{itemize}
\item Imaging principle of X-ray tomography
\item Radon transform and its properties
\item Filtered back projection
\item Ill-posedness of CT reconstruction problem
\item Reconstruction techniques
\end{itemize}
The lecture is mainly based on the following standard references (in alphabetical order):
\begin{itemize}
\item[\cite{Buzug2010}] T.~Buzug: {\it Computed Tomography}
\item[\cite{Epstein2008}] C.~Epstein: {\it Introduction to the Mathematics of Medical Imaging}
\item[\cite{Feeman2015}] T.~Feeman:  {\it The Mathematics of Medical Imaging}
\item[\cite{Iske2018}] A.~Iske: {\it Approximation Theory and Algorithms for Data Analysis}
\item[\cite{Kak2001}] A.~Kak, M.~Slaney: {\it Principles of Computerized Tomographic Imaging}
\item[\cite{Natterer2001}] F.~Natterer: {\it The Mathematics of Computerized Tomography}
\item[\cite{Natterer2001a}] F.~Natterer, F.~Wübbeling: {\it Mathematical Methods in Image Reconstruction}
\end{itemize}

\tableofcontents
\markboth{Contents}{Contents}

\mainmatter

\chapter{Introduction}

The term {\em computerized tomography} (CT) refers to the reconstruction of a bivariate function from its line integral values.
The expression 'tomography' is derived from the Ancient Greek words $\tau \acute{o} \mu o \varsigma$, which means 'slice', and $\gamma \rho \acute{\alpha} \varphi \omega$ meaning 'to write'.

One of the most prominent examples of X-ray tomography is still transmission CT in medical imaging and non-destructive testing.
Here, the aim consists in recovering the interior of an unknown two-dimensional object or cross-section of a three-dimensional object from measurements of one-dimensional X-ray projections.
In order to explore the two-dimensional structure of the object, the X-ray projections are taken from different views.
To this end, a source-detector pair is rotated around the object, see Figure~\ref{fig:CT_principle}, where the source emits X-ray beams of a given initial intensity and the detector measures the intensity of the beams after passing the object.

The imaging principle of X-ray tomography is now based on the fact that the X-ray beams are attenuated when passing matter.
The attenuation of the X-rays depends on the inner structure of the scanned medium and, thus, carries information about the interior of the unknown object.
When a single X-ray beam of known intensity travels along a straight line from source to detector, a fraction of the X-ray photons present in the beam is absorbed by the material and the remaining portion passes through.
The intensity of the beam, as it emerges from the medium, is measured at the detector by counting the arriving photons and the difference of the initial and the final intensities describes the ability of the material to absorb X-ray photons.
These measurements can be transformed into line integral values of the object's attenuation function to be recovered.

\begin{figure}[b]
\centering
\includegraphics[height=0.375\textwidth]{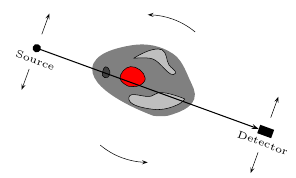}
\caption{Imaging principle of X-ray tomography. To explore the inner structure of an object, a source-detector pair is rotated around the object collecting X-ray projections at different views.}
\label{fig:CT_principle}
\end{figure}

In this chapter, we briefly explain the physical fundamentals of X-ray tomography, derive its mathematical model and give a short historical overview of the milestones in CT development.

\section{Physical principles}

X-ray radiation is generated by the deceleration of fast electrons entering a solid metal anode and consists of electromagnetic waves with wavelengths roughly between $10^{-8} \, \text{m}$ and $10^{-13} \, \text{m}$.
The electrons are emitted from a cathode filament, which is heated by applying the heating voltage $U_h$ to overcome the binding energy of the electrons to the metal of the filament.
After this {\em thermionic emission} the electrons are accelerated in the electric field between the cathode and the anode, which are contained in a vacuum tube called X-ray tube.
The electron velocity~$\nu$ depends on the acceleration voltage $U_a$ via the conservation law of energy
\begin{equation*}
e \, U_a = \frac{1}{2} \, m_e \, \nu^2,
\end{equation*}
where $e = 1.602 \cdot 10^{-19} \, \text{C}$ denotes the charge and $m_e = 9.109 \cdot 10^{-31} \, \text{kg}$ the mass of electrons.
In medical diagnostics $U_a$ is usually chosen between 25kV and 150kV, whereas in material testing it can reach up to 500kV.
A schematic drawing of an X-ray source can be seen in Figure~\ref{fig:X_ray_tube}.

\begin{figure}[b]
\centering
\includegraphics[width=0.925\textwidth]{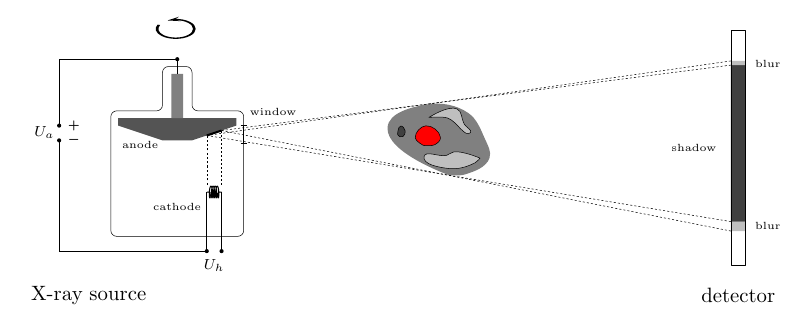}
\caption{Schematic drawing of an X-ray tube and the projection of an object on the detector.}
\label{fig:X_ray_tube}
\end{figure}

When the electrons reach the surface of the anode, they are abruptly stopped and electromagnetic waves are radiated, which leave the tube through a window.
Usually, several photons emerge through the deceleration of one single electron.
It can happen, however, that the entire energy of an electron is transferred into a single photon.
This determines the maximum energy of the X-ray radiation via
\begin{equation*}
E_{\text{max}} = h \,  \nu = e \, U_a
\end{equation*}
and the minimal wavelength of the X-ray spectrum via
\begin{equation*}
\lambda_{\text{min}} = \frac{h \, c}{e \, U_a},
\end{equation*}
where $h$ is Planck's constant and $c$ is the speed of light.
Thus, the range of wavelengths of the generated X-ray spectrum depends on the anode voltage $U_a$.
In contrast to this, the intensity of the X-ray spectrum or the number of photons is solely controlled by the anode current $I_a$.

The efficiency of the conversion from kinetic energy into X-ray radiation depends on the anode material, usually tungsten, and the acceleration voltage.
Most of the kinetic energy, approximately 99\%, is transferred into thermal energy so that the anode heats up and suffers from a serious heat problem.
Thus, rotating anode disks are used to distribute the thermal load over the entire anode.
The heat capacity $E_h$ of an X-ray tube depends on $U_a$, $I_a$ and time $t$ via
\begin{equation*}
E_h = U_a \cdot I_a \cdot t
\end{equation*}
and one has to balance the acceleration voltage, anode current and exposure time appropriately.

\bigbreak

The anode surface is angulated with respect to the electron beam and the target area of the electron beam on the anode is called {\em X-ray focus}.
The size of the electron beam can be controlled by manipulating the trajectories of the accelerated electrons with a focusing device called Wehnelt cylinder.
For a fixed electron focus the size of the X-ray focus can be controlled by varying the anode angle $\varphi$.
The larger $\varphi$, the larger the X-ray focus on which the thermal load is distributed, but the larger also the effective target area seen by the detector, see Figure~\ref{fig:X_ray_focus}.
This projection of the X-ray focus on the detector is called {\em optical focus} and its size determines the quality of the resulting image.
Ideally, the X-rays are created from a point source and an increase in source size results in penumbra regions on the detector and, thus, in a blurred image, as illustrated in Figure~\ref{fig:X_ray_tube}.
However, a very small anode angle is generally not desired because the probability of a backscattering of the electrons increases as the angle between the anode surface normal and the anode rotation axis decreases.

\begin{figure}[t]
\centering
\subfigure[$\varphi = 18.43^\circ$]{\includegraphics[width=0.45\textwidth]{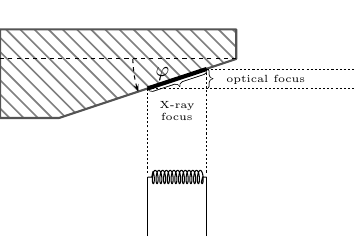}}
\hfil
\subfigure[$\varphi = 26.57^\circ$]{\includegraphics[width=0.45\textwidth]{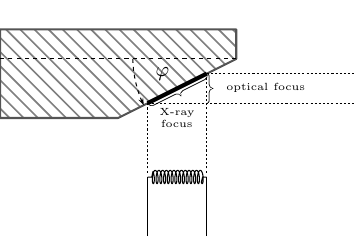}}
\caption{X-ray focus and optical focus depending on the anode angle $\varphi$.}
\label{fig:X_ray_focus}
\end{figure}

\smallskip

As noted before, the imaging principle of computerized tomography is based on the fact that X-ray beams are attenuated when passing through matter.
The attenuation of X-rays, however, is a complicated function of the wavelength.
In general, low-energy X-rays, i.e., radiation with a large wavelength, are more attenuated than high-energy X-rays.
This is the origin of what is called {\em beam hardening}, which produces artefacts in the reconstructed image, because it is standard to consider the X-rays to be monochromatic, i.e., radiation of only one wavelength, within the mathematical reconstruction process.
In practice, the beam-hardening artefacts can be reduced by pre-hardening of the radiation with a metal filter in the X-ray tube.

The X-ray attenuation is due to absorption and scattering of the X-ray photons.
Generally, photon interaction with matter can result in a change in incoming photon energy, photon number or travelling direction.
The most important mechanisms of photon-matter-interaction are\\[-2.5ex]
\begin{tabular}{p{0.325\linewidth}p{0.3\linewidth}p{0.275\linewidth}}
\begin{itemize}
\item photoelectric absorption,
\end{itemize}
&
\begin{itemize}
\item Compton scattering,
\end{itemize}
&
\begin{itemize}
\item pair production.
\end{itemize}
\end{tabular}\\[-1.5ex]
As a consequence of such kind of interactions a photon that interacts with matter is completely removed from the incident beam, in other words an X-ray beam that crosses a medium is not degraded in energy but only attenuated in intensity.
More precisely, the intensity $I(x)$ of a beam travelling through an homogeneous medium with thickness $x$ decreases exponential via
\begin{equation*}
I(x) = I_0 \, \e^{-\mu x},
\end{equation*}
where $I_0$ is the initial intensity of the beam and $\mu$ is the attenuation coefficient of the medium.

In greater detail, photoelectric absorption can happen if the energy of the photon is just larger than the binding energy of the atomic electrons.
In this case, the entire energy of the photon is absorbed by an electron and the photon vanishes.
This process results in a ionization followed by a subsequent ejection of the electron from the atom, where the energy of the liberated electron is the difference between the photon energy and the binding energy of the electron.

\bigbreak

Compton scattering can happen if the energy of the X-ray photon is high with respect to the binding energy of the orbital electrons.
In this case, this latter energy can be ignored and the electrons can be treated as essentially free.
If a photon collides with a quasi-free electron, the photon passes a part of its energy to the electron and the electron is scattered away.
Thus, the scattered photon has a lower energy when it continues to travel through the matter and may be further attenuated until it is completely absorbed if the material thickness is sufficiently large.

For very high photon energy the pair production effect starts to be relevant.
In this process the photon interacts with an orbital electron or atomic nucleus so that the photon disappears and its entire energy is absorbed to produce a positron-electron pair.

\smallskip

After passing the matter the remaining photons in the X-ray beam are detected via their interaction with the detector material.
A commonly used detector architecture consists of a scintillation layer followed by a photon detector.
The short-wave X-ray radiation is converted into long-wave light inside the scintillation medium, which is subsequently detected by a photodiode.

\medskip

In conclusion, the attenuation of X-rays through matter is well understood and the gray values of CT images are a direct physical representation of the material properties.
High values of the attenuation coefficient, describing the extent to which the X-ray intensity is reduced, are due to a high density or a high atomic number of the material the X-rays travelled through.
Traditionally, the quantitative {\em Hounsfield unit} (HU) scale is used to measure X-ray attenuation, which compares the attenuation coefficient of the medium passed through with that of water.

\section{Mathematical model}

Let $f:\R^2 \to \R$ denote the spatially varying {\em attenuation function} of the scanned object, which describes the proportion of X-ray photons being absorbed by the materials travelled through and whose support is assumed to lie in a convex set $\Omega \seq \R^2$, see Figure~\ref{fig:X_ray_beam}.
Thus, $f$ is a characteristic quantity of the scanned body and, mathematically, the goal of computerized tomography is to recover the function~$f$ from the given X-ray scans.

\begin{figure}[b]
\centering
\includegraphics[height=0.375\textwidth]{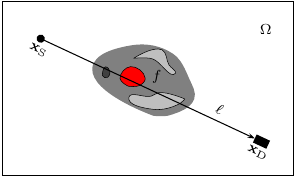}
\caption{A single X-ray beam travels along straight line $\ell$ from source $\bfx_{\text{S}}$ to detector $\bfx_{\text{D}}$.}
\label{fig:X_ray_beam}
\end{figure}

In the following we assume that the X-ray beam is monochromatic, i.e., each X-ray photon has the same energy and the beam propagates with a constant wavelength.
Furthermore, we assume that the beam has zero width and is neither refracted nor diffracted.
Hence, the X-ray beam travels along a straight line $\ell \subset \R^2$ through the object under investigation.
If $I(\bfx)$ denotes the intensity of the X-ray beam at position $\bfx \in \Omega$, the intensity loss $\Delta I(\bfx)$ in a small segment of $\ell$ of length $\Delta \bfx$ is approximately given by
\begin{equation*}
\Delta I(\bfx) \approx -f(\bfx) \cdot I(\bfx) \cdot \Delta \bfx.
\end{equation*}
Taking the limit $\Delta \bfx \to 0$ leads us to the ordinary differential equation
\begin{equation}
\frac{\d}{\d \bfx}I(\bfx) = -f(\bfx) \cdot I(\bfx),
\label{eq:Beer_law}
\end{equation}
which is known as {\em Beer's law}.
If $I_{\text{S}}$ denotes the initial intensity of the X-ray at the source $\bfx_{\text{S}}$ and $I_{\text{D}}$ its final intensity at the detector $\bfx_{\text{D}}$, integrating~\eqref{eq:Beer_law} from source to detector yields
\begin{equation*}
\int_{I_{\text{S}}}^{I_{\text{D}}} \frac{1}{I} \: \d I = -\int_\ell f(\bfx) \: \d \bfx,
\end{equation*}
which implies that
\begin{equation}
\log\biggl(\frac{I_{\text{S}}}{I_{\text{D}}}\biggr) = \int_\ell f(\bfx) \: \d \bfx.
\label{eq:reconstruction_problem}
\end{equation}
Since the values $I_{\text{S}}$ and $I_{\text{D}}$ are measured during the scanning process, the X-ray data provides us with the line integral values of the attenuation function $f$ along the straight line $\ell \subset \R^2$.
Recovering the absorption coefficient $f$ or, equivalently, reconstructing an object from X-ray projections therefore reduces to solving the integral equation~\eqref{eq:reconstruction_problem}.

\smallskip

Consequently, the basic CT reconstruction problem can be formulated as follows.

\begin{problem}[Basic reconstruction problem]
Reconstruct a bivariate function $f \equiv f(x,y)$ on its domain $\Omega \seq \R^2$ from given line integral values
\begin{equation*}
\int_{\ell} f(x,y) \: \d (x,y)
\end{equation*}
along all straight lines $\ell \subset \R^2$ passing through $\Omega$.
\end{problem}

We wish to remark that in reality the attenuation function $f$ not only depends on the spatial variable $\bfx$ but also on the energy $E$ of the X-rays.
Assuming $I_{\text{S}}(E)$ to be the energy spectrum of the X-ray source, equation~\eqref{eq:reconstruction_problem} has to be replaced by
\begin{equation}
I_{\text{D}} = \int_0^{E_{\text{max}}} I_{\text{S}}(E) \, \e^{-\int_\ell f(\bfx,E) \: \d \bfx} \: \d E.
\label{eq:reconstruction_problem_extended}
\end{equation}
Using~\eqref{eq:reconstruction_problem} instead of~\eqref{eq:reconstruction_problem_extended} results in beam-hardening artefacts in the reconstructed image, as explained before.
In practice, however, only equation~\eqref{eq:reconstruction_problem} is used.
Therefore, this is the basic mathematical model for the CT imaging process that will be used throughout the lecture.

\section{Historical milestones}

Historically, the foundation of computerized tomography (CT) was laid in 1895 by the German physicist Wilhelm Röntgen, who discovered X-rays and their capability of penetrating matter.

The purely mathematical problem of reconstructing a function from its line integral values was analytically solved by the Austrian mathematician Johann Radon in 1917 in his pioneering paper ``Über die Bestimmung von Funktionen durch ihre Integralwerte längs gewisser Mannigfaltigkeiten''.
However, due to the complexity and depth of his mathematical publication the consequences of his ground-breaking results were revealed only very late in the mid-20th century.
Moreover, the paper was published in German, which hindered a wide distribution of the work.

The two pioneering scientists who were primarily responsible for the development of CT in the 1960s and 1970s were Allan Cormack and Godfrey Hounsfield.
Cormack developed mathematical algorithms to create an image from X-ray scans and Hounsfield designed the first operational CT scanner as well as the first commercially available model.
In 1979 the Nobel Prize for Medicine was jointly awarded to Cormack and Hounsfield for their fundamental achievements.
In Table~\ref{tab:CT_milestones}, some of the historical milestones and development steps of CT are summarized.

\begin{figure}[t]
\centering
\subfigure[Pencil beam geometry]{\label{fig:pencil_beam}\includegraphics[width=0.25\textwidth]{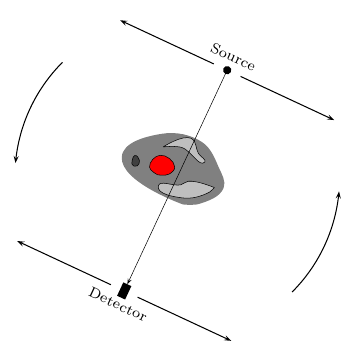}}
\hfil
\subfigure[Parallel beam geometry]{\label{fig:parallel_beam}\includegraphics[width=0.25\textwidth]{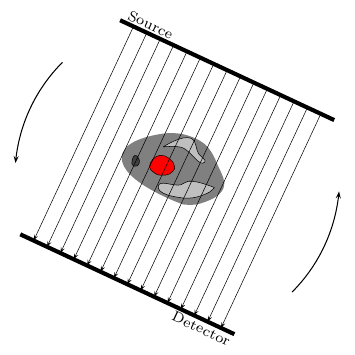}}
\hfil
\subfigure[Fan beam geometry]{\label{fig:fan_beam}\includegraphics[width=0.25\textwidth]{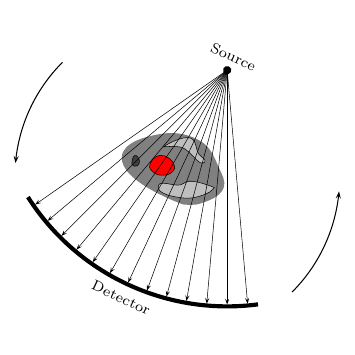}}
\caption{Different scanning geometries.}
\end{figure}

\bigbreak

Retrospectively, four distinct generations of classical CT scanners have emerged, which are still developed further.
Their classification relates to both the way that X-ray tube and detector are constructed and the source-detector pair moves around the object under investigation.

The first generation involves an X-ray tube that emits a single X-ray beam.
A single detector is situated opposite of the X-ray source and this configuration is moved linearly as well as rotated through different angles, cf.\ Figure~\ref{fig:pencil_beam}.
This scanning geometry is called {\em pencil beam geometry}.
The time-consuming linear movement of the source-detector pair can be avoided by using an X-ray source emitting parallel X-ray beams and a detector array, see Figure~\ref{fig:parallel_beam}.
In this {\em parallel beam geometry}, the source-detector pair only needs to be rotated around the object.

The second generation involves an X-ray source that emits a narrow fan of X-ray beams and a short detector array consisting of multiple elements.
However, since the aperture angle of the fan beam is small, the source-detector pair still needs to be translated linearly before it is rotated.
Despite the need for linear displacement, the acquisition time was reduced as the detector array could measure several intensities simultaneous.

The third generation has a substantially larger angle of the X-ray fan and a longer detector array such that the entire measuring field can be X-rayed simultaneously for one single projection angle, see Figure~\ref{fig:fan_beam}.
In this way, the need for linear displacement of the source-detector system is removed and the acquisition time is drastically reduced.
This scanning geometry is referred to as {\em fan beam geometry}.

The fourth generation does not differ from the third generation with respect to the X-ray tube.
A fan beam source rotates around the object without linear displacement.
The difference is a closed stationary detector ring and the source path can be either inside or outside the ring.

\begin{table}[b]
\caption{Summary of historical CT milestones.}
\centering
\begin{tabular}{lp{0.899\linewidth}}
Year & Milestone \\[0.25ex]
\hline \\[-1.5ex]
1895 & Röntgen discovers a new kind of radiation, which he named X-rays. \\[0.5ex]
1901 & Röntgen receives the first Nobel Prize for Physics. \\[0.5ex]
1917 & Radon publishes his epochal work on the solution to the problem of reconstructing a function from its line integral values. \\[0.5ex]
1963 & Cormack contributes the first mathematical algorithms for tomographic reconstruction from X-ray scans. \\[0.5ex]
1969 & Hounsfield shows proof of the principle with the first CT scanner based on a radioactive source at the EMI research laboratories. \\[0.5ex]
1972 & Hounsfield and Ambrose publish the first clinical scans with an EMI head scanner. \\[0.5ex]
1975 & Hounsfield and Ambrose set-up the first whole body scanner with a fan beam system. \\[0.5ex]
1979 & Cormack and Hounsfield receive the Nobel Prize for Medicine.
\end{tabular}
\label{tab:CT_milestones}
\end{table}

\chapter{The Radon transform}

In this chapter we introduce the Radon transform~$\Radon$ used in the mathematical model for the measurement process in computerized tomography and study some of its fundamental properties.

\section{Lines in the plane}

We have seen that the CT scanner provides line integral values of the function to be reconstructed along straight lines in the plane.
In order to derive a reconstruction theory, we now introduce a suitable parametrization of these lines, as illustrated in Figure~\ref{fig:straight_line}.

\begin{definition}[Straight line in the plane]
For any pair $(t,\theta) \in \R^2$ of parameters, we define $\ell_{t,\theta} \subset \R^2$ to be the unique {\em straight line} that passes through the point $\bfx_{t,\theta} = (t \cos(\theta),t \sin(\theta)) \in \R^2$ and is perpendicular to the unit vector $\bfn_\theta = (\cos(\theta),\sin(\theta)) \in \R^2$.
\end{definition}

\begin{figure}[b]
\centering
\includegraphics[height=5.75cm,keepaspectratio]{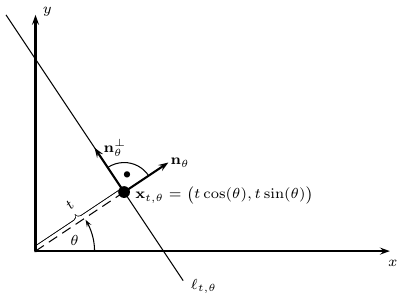}
\caption{Representation of the straight line $\ell_{t,\theta} \subset \R^2$ with parameters $(t,\theta) \in \R \times [0,\pi)$.}
\label{fig:straight_line}
\end{figure}

We remark that every straight line in the plane can be characterized as $\ell_{t,\theta}$ for some real numbers $t, \theta \in \R$.
Furthermore, for any pair of parameters $(t,\theta) \in \R^2$ holds that
\begin{equation}
\ell_{t,\theta+2\pi} = \ell_{t,\theta}
\quad \mbox{ and } \quad
\ell_{t,\theta+\pi} = \ell_{-t,\theta}.
\label{eq:lines_periodicity}
\end{equation}
Thus, each line in the plane has infinitely many different representations of the form $\ell_{t,\theta}$.
To enforce uniqueness, we restrict the parameter pair $(t,\theta)$ to $t \in \R$ and $\theta \in [0,\pi)$ so that the set
\begin{equation*}
\bigl\{ \ell_{t,\theta} \bigm| t \in \R \mbox{, } \theta \in [0,\pi) \bigr\}
\end{equation*}
contains exactly all straight lines in the plane.
For fixed $t \in \R$ and $\theta \in [0,\pi)$ we now parametrize the line $\ell_{t,\theta}$ as follows.

\begin{remark}[Parametrization of the straight line $\ell_{t,\theta}$]
For any fixed angle $\theta \in [0,\pi)$, we have
\begin{equation*}
\bfn_\theta = (\cos(\theta),\sin(\theta)) \perp (-\sin(\theta),\cos(\theta)) = \bfn^{\perp}_\theta.
\end{equation*}
Therefore, every point $(x,y)$ on the line $\ell_{t,\theta}$ is of the form
\begin{equation}
(x,y) = t \cdot \textbf{n}_\theta + s \cdot \textbf{n}^{\perp}_\theta = (t \cos(\theta) - s \sin(\theta),t \sin(\theta) + s \cos(\theta))
\label{eq:line_parametrization}
\end{equation}
for some $s \in \R$ and, consequently,
\begin{equation*}
\ell_{t,\theta} = \bigl\{ (t \cos(\theta) - s \sin(\theta),t \sin(\theta) + s \cos(\theta)) \bigm| s \in \R \bigr\}.
\end{equation*}
\end{remark}

We close this paragraph on lines in the plane by noting that for fixed angle $\theta \in [0,\pi)$ there is a unique straight line $\ell_{t,\theta}$ passing through a given point $(x,y) \in \R^2$.

\begin{remark}
\label{rem:points_on_lines}
For any point $(x,y) \in \R^2$ in the plane and a given angle $\theta \in [0,\pi)$ there exists a unique value for $t \in \R$ such that the straight line $\ell_{t,\theta}$ passes through $(x,y)$.
The unique $t,s \in \R$ satisfying
\begin{equation*}
x = t \cos(\theta) - s \sin(\theta)
\quad \mbox{ and } \quad
y = t \sin(\theta) + s \cos(\theta)
\end{equation*}
are given by
\begin{equation*}
t = x \cos(\theta) + y \sin(\theta)
\quad \mbox{ and } \quad
s = -x \sin(\theta) + y \cos(\theta)
\end{equation*}
and meet the relation
\begin{equation*}
x^2 + y^2 = t^2 + s^2.
\end{equation*}
\end{remark}

With the described parametrization of straight lines in the plane, we finally reformulate the basic CT reconstruction problem as follows.

\begin{problem}[Basic CT reconstruction problem]
On given domain $\Omega \seq \R^2$, reconstruct a bivariate function $f \equiv f(x,y)$ from its line integral values
\begin{equation}
\int_{\ell_{t,\theta}} f(x,y) \: \d (x,y),
\label{eq:line_integral}
\end{equation}
which are assumed to be given for all parameters $(t,\theta) \in \R \times [0,\pi)$.
\end{problem}

\section{Definition and basic properties}

We have seen that the scanning process in CT provides us with the line integrals of a bivariate function $f$ along straight lines in the plane, from which we have to reconstruct $f$.
For $f \in \L^1(\R^2)$ and any parameter pair $(t,\theta) \in \R^2$ the line integral in~\eqref{eq:line_integral} can be rewritten as
\begin{equation*}
\int_{\ell_{t,\theta}} f(x,y) \: \d (x,y) = \int_{\R} f(t \cos(\theta) - s \sin(\theta),t \sin(\theta) + s \cos(\theta)) \: \d s,
\end{equation*}
where we used the parametrization~\eqref{eq:line_parametrization} with the arclength element
\begin{equation*}
\|(\dot{x}(s),\dot{y}(s))\|_{\R^2} \: \d s = \sqrt{(-\sin(\theta))^2 + (\cos(\theta))^2} \: \d s = \d s.
\end{equation*}
The integral transform which maps a function $f:\R^2 \to \R$ into the set of its line integrals was firstly investigated by the Austrian mathematician Johann Radon in 1917.
In his honour, this integral transform is called {\em Radon transform} and, thus, the CT reconstruction problem simply seeks for the inversion of the Radon transform on $\R^2$.

\begin{definition}[Radon transform]
\label{def:Radon_transform}
Let $f \in \L^1(\R^2)$ be a bivariate function in Cartesian coordinates.
Then, the {\em Radon transform} $\Radon f$ of $f$ at the point $(t,\theta) \in \R^2$ is defined as
\begin{equation}
\Radon f(t,\theta) = \int_{\ell_{t,\theta}} f(x,y) \: \d (x,y) = \int_{\R} f(t \cos(\theta) - s \sin(\theta),t \sin(\theta) + s \cos(\theta)) \: \d s.
\label{eq:Radon_transform}
\end{equation}
The graph of the Radon transform $\Radon f$ in the $(t,\theta)$-plane is called {\em sinogram} of $f$.
\end{definition}

Note that the Radon transform $\Radon f$ of a bivariate function $f$ is $2\pi$-periodic in the angular variable $\theta$ and, due to relation~\eqref{eq:lines_periodicity}, it suffices to consider $\Radon f$ only on the domain $\R \times [0,\pi)$.
Further, an application of Fubini's theorem shows that the Radon transform $\Radon f$ of $f \in \L^1(\R^2)$ is well-defined almost everywhere on $\R \times [0,\pi)$ in the sense that for any angle $\theta \in [0,\pi)$ the integral in~\eqref{eq:Radon_transform} is well-defined for almost all $t \in \R$.

\begin{proposition}[Well-definedness of the Radon transform]
For $f \in \L^1(\R^2)$, the Radon transform $\Radon f$ is well-defined almost everywhere on $\R \times [0,\pi)$.
\end{proposition}

\begin{proof}
Let $f \in \L^1(\R^2)$. We define the auxiliary function $H: \R^2 \times [0,\pi) \to \R^2$ as
\begin{equation*}
H(s,t,\theta) = (t \cos(\theta) - s \sin(\theta),t \sin(\theta) + s \cos(\theta))
\quad \mbox{ for } (s,t,\theta) \in \R^2 \times [0,\pi).
\end{equation*}
Then, $H$ is continuous on $\R^2 \times [0,\pi)$ and we have
\begin{equation*}
\Radon f(t,\theta) = \int_\R f(H(s,t,\theta)) \: \d s
\quad \forall \, (t,\theta) \in \R \times [0,\pi).
\end{equation*}
Note that, for fixed angle $\theta \in [0,\pi)$, the mapping
\begin{equation*}
(s,t) \longmapsto (t \cos(\theta) - s \sin(\theta),t \sin(\theta) + s \cos(\theta))
\end{equation*}
is a rotation in $\R^2$ and, thus, measure preserving.
In particular, the mapping
\begin{equation*}
(s,t) \longmapsto f(H(s,t,\theta))
\end{equation*}
is in $\L^1(\R^2)$ for all $\theta \in [0,\pi)$, since $f \in \L^1(\R^2)$ by assumption.
Therefore, Fubini's theorem shows that the partial mapping
\begin{equation*}
s \longmapsto f(H(s,t,\theta))
\end{equation*}
is integrable on $\R$ for almost all $t \in \R$ and all $\theta \in [0,\pi)$.
This implies that
\begin{equation*}
(t,\theta) \longmapsto \int_\R f(H(s,t,\theta)) \: \d s = \Radon f(t,\theta)
\end{equation*}
is well-defined for almost all $t \in \R$ and all $\theta \in [0,\pi)$.
\end{proof}

\begin{remark}
The Radon transform $\Radon f$ of a function $f$ is defined everywhere on $\R \times [0,\pi)$ if the integral of $f$ along the line $\ell_{t,\theta}$ exists for all pairs $(t,\theta) \in \R \times [0,\pi)$.
To ensure this, it suffices to require that $f$ is continuous on $\R^2$ and has compact support.
\end{remark}

With Definition~\ref{def:Radon_transform} the classical CT reconstruction problem can be reformulated as follows.

\begin{problem}[Reconstruction problem]
\label{prob:CT_reconstruction_problem_Radon}
For given domain $\Omega \seq \R^2$, reconstruct a bivariate function $f \in \L^1(\Omega)$ from its Radon data
\begin{equation*}
\bigl\{ \Radon f(t,\theta) \bigm| t \in \R,\, \theta \in [0,\pi) \bigr\}.
\end{equation*}
\end{problem}

Therefore, the CT reconstruction problem seeks for the inversion of the Radon transform $\Radon$.
However, before explaining the inversion of $\Radon$, we collect some of its fundamental properties.

\begin{observation}
The Radon transform $\Radon$ maps a bivariate function $f \equiv f(x,y)$ in Cartesian coordinates onto a bivariate function $\Radon f \equiv \Radon f(t,\theta)$ in polar coordinates.
\end{observation}

The first two important properties of the Radon transform are its linearity and evenness.

\begin{proposition}[Linearity of the Radon transform]
The Radon transform $\Radon$ is a positive linear integral operator, i.e., for all $\alpha,\beta \in \R$ and $f,g \in \L^1(\R^2)$ we have
\begin{equation*}
\Radon (\alpha f + \beta g) = \alpha \, \Radon f + \beta \, \Radon g
\end{equation*}
and
\begin{equation*}
f \geq 0
\quad \implies \quad
\Radon f \geq 0.
\end{equation*}
Furthermore, the Radon transform $\Radon f$ of $f \in \L^1(\R^2)$ satisfies the evenness condition
\begin{equation*}
\Radon f(-t,\theta+\pi) = \Radon f(t,\theta)
\quad \forall \, (t,\theta) \in \R^2.
\end{equation*}
\end{proposition}

\begin{proof}
The statement follows from the positivity and linearity of the integral and~\eqref{eq:lines_periodicity}.
\end{proof}

The next theorem is concerned with the continuity of the Radon transform as a mapping
\begin{equation*}
\Radon: \L^1(\R^2) \to \L^1(\R \times [0,\pi)).
\end{equation*}

\begin{proposition}[Continuity of the Radon transform]
\label{prop:Radon_L1_norm}
The Radon transform $\Radon$ is a continuous operator from $\L^1(\R^2)$ to $\L^1(\R \times [0,\pi))$.
In particular, for $f \in \L^1(\R^2)$ we have
\begin{equation*}
\|\Radon f\|_{\L^1(\R \times [0,\pi))} \leq \pi \, \|f\|_{\L^1(\R^2)}.
\end{equation*}
\end{proposition}

\begin{proof}
Let $f \in \L^1(\R^2)$ and $\theta \in [0 ,\pi)$ be fixed. By the definition of the Radon transform $\Radon$ follows that
\begin{align*}
\int_\R |\Radon f(t,\theta)| \: \d t & = \int_\R \bigg|\int_{\ell_{t,\theta}} f(x,y) \: \d (x,y)\bigg| \: \d t \\
& \leq \int_\R \int_\R |f(t \cos(\theta) - s \sin(\theta),t \sin(\theta) + s \cos(\theta))| \: \d s \, \d t.
\end{align*}
Applying the transformation
\begin{equation*}
x = t \cos(\theta) - s \sin(\theta)
\quad \mbox{ and } \quad
y = t \sin(\theta) + s \cos(\theta),
\end{equation*}
we get $\d x \, \d y = \d t \, \d s$ and, consequently,
\begin{equation*}
\int_\R |\Radon f(t,\theta)| \: \d t \leq \int_\R \int_\R |f(x,y)| \: \d x \, \d y = \|f\|_{\L^1(\R^2)}.
\end{equation*}
This gives
\begin{equation*}
\|\Radon f\|_{\L^1(\R \times [0,\pi))} = \int_0^\pi \int_\R |\Radon f(t,\theta)| \: \d t \, \d \theta \leq \|f\|_{\L^1(\R^2)} \, \int_0^\pi 1 \: \d \theta = \pi \, \|f\|_{\L^1(\R^2)}.
\end{equation*}
Hence, $\Radon$ is a continuous operator from $\L^1(\R^2)$ to $\L^1(\R \times [0,\pi))$.
\end{proof}

A special situation occurs if the function $f: \R^2 \to \R$ is radially symmetric, i.e., $f$ is invariant under rotations.
This means that there exists a function $f_0: \R \to \R$ with
\begin{equation*}
f(x,y) = f_0(\|(x,y)\|_{\R^2})
\quad \forall \, (x,y) \in \R^2.
\end{equation*}

\begin{proposition}
If the function $f \in \L^1(\R^2)$ is radially symmetric, its Radon transform $\Radon f$ depends only on the modulus $|t|$ of the radial variable $t \in \R$, but not on the angle $\theta \in [0,\pi)$.
\end{proposition}

\begin{proof}
Let $f \in \L^1(\R^2)$ be radially symmetric.
Thus, there exists a function $f_0:\R \to \R$ with
\begin{equation*}
f(x,y) = f_0(x^2+y^2)
\quad \forall \, (x,y) \in \R^2.
\end{equation*}
With this, for all $(t,\theta) \in \R \times [0,\pi)$ follows that
\begin{align*}
\Radon f(t,\theta) & = \int_{\ell_{t,\theta}} f(x,y) \: \d (x,y) = \int_{\R} f(t \cos(\theta) - s \sin(\theta),t \sin(\theta) + s \cos(\theta)) \: \d s \\
& = \int_{\R} f_0((t \cos(\theta) - s \sin(\theta))^2 + (t \sin(\theta) + s \cos(\theta))^2) \: \d s \\
& = \int_{\R} f_0(t^2+s^2) \: \d s.
\end{align*}
This shows that $\Radon f(t,\theta)$ is independent of the angular variable $\theta \in [0,\pi)$ and only depends on the absolute value $|t|$ of the radial variable $t \in \R$.
\end{proof}

Another important property of $\Radon$ is that the compact support of a function $f$ carries over to its Radon transform $\Radon f$.

\begin{proposition}
Let $f \in \L^1(\R^2)$ have compact support, i.e., there exists an $R > 0$ such that
\begin{equation*}
f(x,y) = 0
\quad \forall \, (x,y) \in \R^2: ~ x^2 + y^2 > R^2.
\end{equation*}
Then, $\Radon f$ has compact support as well with
\begin{equation*}
\Radon f(t,\theta) = 0
\quad \forall \, |t| > R, ~ \theta \in [0,\pi).
\end{equation*}
\end{proposition}

\begin{proof}
For each fixed $\theta \in [0,\pi)$ and $t \in \R$ with $|t| > R$, we have
\begin{equation*}
f(t \cos(\theta) - s \sin(\theta),t \sin(\theta) + s \cos(\theta)) = 0
\quad \forall \, s \in \R.
\end{equation*}
This implies that
\begin{equation*}
\Radon f(t,\theta) = 0
\quad \forall \, (t,\theta) \in \R \times [0,\pi): ~ |t| > R
\end{equation*}
and, hence, $\Radon f$ has compact support as well.
\end{proof}

We now come to the first basic analytical example.
To this end, let $\rect{R}: \R \to \R$ denote the characteristic function of the interval $[-R,R]$, i.e.,
\begin{equation*}
\rect{R}(S) = \chi_{[-R,R]}(S)
= \begin{cases}
1 & \text{for } |S| \leq R \\
0 & \text{for } |S| > R.
\end{cases}
\end{equation*}

\begin{example}
\label{ex:Radon_ball}
Consider the characteristic function $\chi_{B_R(0)}$ of the ball $B_R(0) \subset \R^2$ of radius $R > 0$ around $0$, i.e.,
\begin{equation*}
\chi_{B_R(0)}(x,y) = \begin{cases}
1 & \text{for } x^2 + y^2 \leq R^2 \\
0 & \text{for } x^2 + y^2 > R^2.
\end{cases}
\end{equation*}
For all $(t,\theta) \in \R \times [0,\pi)$ we then have
\begin{equation*}
\chi_{B_R(0)}(t \cos(\theta) - s \sin(\theta),t \sin(\theta) + s \cos(\theta)) = \begin{cases}
1 & \text{for } t^2 + s^2 \leq R^2 \\
0 & \text{for } t^2 + s^2 > R^2
\end{cases}
\end{equation*}
and, thus, for the Radon transform of $\chi_{B_R(0)}$ follows that
\begin{align*}
\Radon \chi_{B_R(0)}(t,\theta) & = \int_\R \chi_{B_R(0)}(t \cos(\theta) - s \sin(\theta),t \sin(\theta) + s \cos(\theta)) \: \d s \\
& = \begin{cases}
2 \sqrt{R^2 - t^2} & \text{for } |t| \leq R \\
0 & \text{for } |t| > R
\end{cases} = 2 \sqrt{R^2 - t^2} \, \rect{R}(t).
\end{align*}
\end{example}

To calculate the Radon transform of more involved examples, we need some additional basic properties of the Radon transform, namely its shift, scaling and rotation property.

\bigbreak

We start with studying the effect of shifting the argument in the target function $f$.

\begin{proposition}[Shift property of the Radon transform]
Let $f \equiv f(x,y)$ be a bivariate function with Radon transform $\Radon f \equiv \Radon f(t,\theta)$.
For a given vector $c = (c_x,c_y) \in \R^2$ we define the shifted function $f_c$ via
\begin{equation*}
f_c(x,y) = f(x-c_x,y-c_y)
\quad \mbox{ for } (x,y) \in \R^2.
\end{equation*}
Then, the Radon transform $\Radon f_c$ of $f_c$ is given by
\begin{equation*}
\Radon f_c(t,\theta) = \Radon f(t - c_x \cos(\theta) - c_y \sin(\theta),\theta)
\quad \forall \, (t,\theta) \in \R \times [0,\pi).
\end{equation*}
\end{proposition}

\begin{proof}
For fixed $(t,\theta) \in \R \times [0,\pi)$, the definition of the Radon transform $\Radon$ yields
\begin{align*}
\Radon f_c(t,\theta) & = \int_\R f_c(t \cos(\theta) - s \sin(\theta),t \sin(\theta) + s \cos(\theta)) \: \d s \\
& = \int_\R f(t \cos(\theta) - s \sin(\theta) - c_x,t \sin(\theta) + s \cos(\theta) - c_y) \: \d s,
\end{align*}
where
\begin{align*}
t \cos(\theta) - s \sin(\theta) - c_x & = (t - c_x \cos(\theta)) \cos(\theta) - (s + c_x \sin(\theta)) \sin(\theta) \\
& = (t - c_x \cos(\theta) - c_y \sin(\theta)) \cos(\theta) - (s + c_x \sin(\theta) - c_y \cos(\theta)) \sin(\theta)
\end{align*}
and
\begin{align*}
t \sin(\theta) + s \cos(\theta) - c_y & = (t - c_y \sin(\theta)) \sin(\theta) + (s - c_y \cos(\theta)) \cos(\theta) \\
& = (t - c_x \cos(\theta) - c_y \sin(\theta)) \sin(\theta) + (s + c_x \sin(\theta) - c_y \cos(\theta)) \cos(\theta).
\end{align*}
Consequently, by substituting
\begin{equation*}
\tau = t - c_x \cos(\theta) - c_y \sin(\theta)
\quad \mbox{ and } \quad
\sigma = s + c_x \sin(\theta) - c_y \cos(\theta),
\end{equation*}
we obtain $\d \sigma = \d s$ and can conclude that
\begin{align*}
\Radon f_c(t,\theta) & = \int_\R f(\tau \cos(\theta) - \sigma \sin(\theta),\tau \sin(\theta) + \sigma \cos(\theta)) \: \d \sigma = \Radon f(\tau,\theta) \\
& = \Radon f(t - c_x \cos(\theta) - c_y \sin(\theta),\theta),
\end{align*}
which completes the proof.
\end{proof}

We continue with the effect of scaling the argument in the target function $f$.

\begin{proposition}[Scaling property of the Radon transform]
Let $f \equiv f(x,y)$ be a bivariate function with Radon transform $\Radon f \equiv \Radon f(t,\theta)$.
For given positive constants $a,b > 0$ we define the scaled function $f_{a,b}$ via
\begin{equation*}
f_{a,b}(x,y) = f\Big(\frac{x}{a},\frac{y}{b}\Big)
\quad \mbox{ for } (x,y) \in \R^2.
\end{equation*}
Then, the Radon transform $\Radon f_{a,b}$ of $f_{a,b}$ is given by
\begin{equation*}
\Radon f_{a,b}(t,\theta) = \frac{a b}{c_{a,b}(\theta)} \, \Radon f\bigg(\frac{t}{c_{a,b}(\theta)}, \atan\Big(\frac{b}{a} \tan(\theta)\Big)\bigg)
\quad \forall \, (t,\theta) \in \R \times [0,\pi),
\end{equation*}
where
\begin{equation*}
c_{a,b}(\theta) = \sqrt{a^2 \cos^2(\theta) + b^2 \sin^2(\theta)} > 0
\end{equation*}
and
\begin{equation*}
\atan\Big(\frac{b}{a} \tan(\theta)\Big) = \begin{cases}
\arctan\big(\frac{b}{a} \tan(\theta)\big) & \text{for } \sin(\theta) \cos(\theta) > 0 \\
0 & \text{for } \sin(\theta) = 0 \\
\frac{\pi}{2} & \text{for } \cos(\theta) = 0 \\
\arctan\big(\frac{b}{a} \tan(\theta)\big) + \pi & \text{for } \sin(\theta) \cos(\theta) < 0
\end{cases} \in [0,\pi).
\end{equation*}
\end{proposition}

\begin{proof}
For fixed $(t,\theta) \in \R \times [0,\pi)$, the definition of the Radon transform yields
\begin{align*}
\Radon f_{a,b}(t,\theta) & = \int_\R f_{a,b}(t \cos(\theta) - s \sin(\theta),t \sin(\theta) + s \cos(\theta)) \: \d s \\
& = \int_\R f\Big(\frac{t}{a} \cos(\theta) - \frac{s}{a} \sin(\theta),\frac{t}{b} \sin(\theta) + \frac{s}{b} \cos(\theta)\Big) \: \d s.
\end{align*}
By considering the modified angle
\begin{equation*}
\vartheta = \atan\Big(\frac{b}{a} \tan(\theta)\Big) \in [0,\pi),
\end{equation*}
we have
\begin{equation*}
\cos(\vartheta) = \frac{a \cos(\theta)}{c_{a,b}(\theta)}
\quad \mbox{ and } \quad
\sin(\vartheta) = \frac{b \sin(\theta)}{c_{a,b}(\theta)}.
\end{equation*}
Consequently, using the relation
\begin{equation*}
c_{a,b}^2(\theta) = a^2 \cos^2(\theta) + b^2 \sin^2(\theta)
\end{equation*}
we obtain
\begin{align*}
\frac{t}{a} \cos(\theta) - \frac{s}{a} \sin(\theta) & = \frac{t}{c_{a,b}(\theta)} \, \frac{a \cos(\theta)}{c_{a,b}(\theta)} - \frac{t}{a} \, \frac{a^2 - c_{a,b}^2(\theta)}{c_{a,b}^2(\theta)} \, \cos(\theta) - \frac{s}{a} \, \sin(\theta) \\
& = \frac{t}{c_{a,b}(\theta)} \, \cos(\vartheta) - \frac{c_{a,b}(\theta)}{a b} \, \bigg(s + t \, \frac{a^2 - b^2}{c_{a,b}^2(\theta)} \, \sin(\theta) \cos(\theta)\bigg) \, \sin(\vartheta)
\end{align*}
and
\begin{align*}
\frac{t}{b} \sin(\theta) + \frac{s}{b} \cos(\theta) & = \frac{t}{c_{a,b}(\theta)} \, \frac{b \sin(\theta)}{c_{a,b}(\theta)} - \frac{t}{b} \, \frac{b^2 - c_{a,b}^2(\theta)}{c_{a,b}^2(\theta)} \, \sin(\theta) + \frac{s}{b} \, \cos(\theta) \\
& = \frac{t}{c_{a,b}(\theta)} \, \sin(\vartheta) + \frac{c_{a,b}(\theta)}{a b} \, \bigg(s + t \, \frac{a^2 - b^2}{c_{a,b}^2(\theta)} \, \sin(\theta) \cos(\theta)\bigg) \, \cos(\vartheta).
\end{align*}
Therefore, by substituting
\begin{equation*}
\tau = \frac{t}{c_{a,b}(\theta)}
\quad \mbox{ and } \quad
\sigma = \frac{c_{a,b}(\theta)}{a b} \, \bigg(s + t \, \frac{a^2 - b^2}{c_{a,b}^2(\theta)} \, \sin(\theta) \cos(\theta)\bigg),
\end{equation*}
we have
\begin{equation*}
\d \sigma = \frac{c_{a,b}(\theta)}{a b} \: \d s
\end{equation*}
and can conclude that
\begin{align*}
\Radon f_c(t,\theta) & = \frac{a b}{c_{a,b}(\theta)} \int_\R f(\tau \cos(\vartheta) - \sigma \sin(\vartheta),\tau \sin(\vartheta) + \sigma \cos(\vartheta)) \: \d \sigma = \frac{a b}{c_{a,b}(\theta)} \, \Radon f(\tau,\vartheta) \\
& = \frac{a b}{c_{a,b}(\theta)} \, \Radon f\bigg(\frac{t}{c_{a,b}(\theta)}, \atan\Big(\frac{b}{a} \tan(\theta)\Big)\bigg),
\end{align*}
as stated.
\end{proof}

Finally, we come to the effect of rotating the argument in the target function $f$.

\begin{proposition}[Rotation property of the Radon transform]
Let $f \equiv f(x,y)$ be a bivariate function with Radon transform $\Radon f \equiv \Radon f(t,\theta)$.
For a given rotation angle $\varphi \in [-\pi,\pi)$ we define the rotated function $f_\varphi$ via
\begin{equation*}
f_\varphi(x,y) = f(x \cos(\varphi) + y \sin(\varphi),-x \sin(\varphi) + y \cos(\varphi))
\quad \mbox{ for } (x,y) \in \R^2.
\end{equation*}
Then, the Radon transform $\Radon f_\varphi$ of $f_\varphi$ is given by
\begin{equation*}
\Radon f_\varphi(t,\theta) = \Radon f(t,\theta - \varphi)
\quad \forall \, (t,\theta) \in \R \times [0,\pi).
\end{equation*}
\end{proposition}

\begin{proof}
For fixed $(t,\theta) \in \R \times [0,\pi)$, the definition of the Radon transform $\Radon$ yields
\begin{equation*}
\Radon f_\varphi(t,\theta) = \int_\R f_\varphi(t \cos(\theta) - s \sin(\theta),t \sin(\theta) + s \cos(\theta)) \: \d s = \int_\R f(x(s),y(s)) \: \d s
\end{equation*}
with
\begin{align*}
x(s) & = (t \cos(\theta) - s \sin(\theta)) \cos(\varphi) + (t \sin(\theta) + s \cos(\theta)) \sin(\varphi) \\
& = t (\cos(\theta) \cos(\varphi) + \sin(\theta) \sin(\varphi)) - s (\sin(\theta) \cos(\varphi) - \cos(\theta) \sin(\varphi)) \\
& = t \cos(\theta - \varphi) - s \sin(\theta - \varphi)
\end{align*}
and
\begin{align*}
y(s) & = -(t \cos(\theta) - s \sin(\theta)) \sin(\varphi) + (t \sin(\theta) + s \cos(\theta)) \cos(\varphi) \\
& = t (\sin(\theta) \cos(\varphi) - \cos(\theta) \sin(\varphi)) + s (\sin(\theta) \sin(\varphi) + \cos(\theta) \cos(\varphi)) \\
& = t \sin(\theta - \varphi) + s \cos(\theta - \varphi).
\end{align*}
Consequently, we obtain
\begin{align*}
\Radon f_\varphi(t,\theta) & = \int_\R f(t \cos(\theta - \varphi) - s \sin(\theta - \varphi),t \sin(\theta - \varphi) + s \cos(\theta - \varphi)) \: \d s \\
& = \Radon f(t ,\theta - \varphi)
\end{align*}
and the proof is complete.
\end{proof}

We are now prepared to deal with more evolved examples.
To this end, we first consider the characteristic function of an ellipse with the following parameters:
\begin{tabbing}
\hspace*{2.5cm}\=\hspace*{0.5cm}\=\hspace*{2.5cm}\=\hspace*{0.5cm}\=\hspace*{5cm}\=\hspace*{0.5cm}\= \kill
\> a: \> major axis, \> h: \> x-coordinate of the center, \> $\varphi$: \> rotation angle, \\
\> b: \> minor axis, \> k: \> y-coordinate of the center.
\end{tabbing}

\begin{example}
\label{ex:Radon_ellipse}
Let $f_e$ denote the characteristic function of an ellipse with parameters $a,b > 0$, $h,k \in \R$ and $\varphi \in [-\pi,\pi)$, i.e., for $(x,y) \in \R^2$ we have
\begin{equation*}
f_e(x,y) = \chi_{B_1(0)}\Big(\frac{(x-h)\cos(\varphi) + (y-k)\sin(\varphi)}{a},\frac{-(x-h)\sin(\varphi) + (y-k)\cos(\varphi)}{b}\Big).
\end{equation*}
For the sake of brevity, we define the functions
\begin{equation*}
g(x,y) = \chi_{B_1(0)}\Big(\frac{x}{a},\frac{y}{b}\Big)
\quad \mbox{ for } (x,y) \in \R^2
\end{equation*}
and
\begin{equation*}
g_\varphi(x,y) = g(x \cos(\varphi) + y \sin(\varphi),-x \sin(\varphi) + y \cos(\varphi))
\quad \mbox{ for } (x,y) \in \R^2
\end{equation*}
so that the function $f_e$ can be written as
\begin{equation*}
f_e(x,y) = g_\varphi(x-h,y-k)
\quad \forall \, (x,y) \in \R^2.
\end{equation*}
In Example~\ref{ex:Radon_ball} we have seen that the Radon transform of the characteristic function of the ball $B_R(0)$ is given by
\begin{equation*}
\Radon \chi_{B_R(0)}(t,\theta) = 2 \sqrt{R^2 - t^2} \, \rect{R}(t)
\quad \forall \, (t,\theta) \in \R \times [0,\pi).
\end{equation*}
Consequently, for $(t,\theta) \in \R \times [0,\pi)$, applying the scaling property of the Radon transform yields
\begin{align*}
\Radon g(t,\theta) & = \frac{a b}{\sqrt{a^2 \cos^2(\theta) + b^2 \sin^2(\theta)}} \, \Radon \chi_{B_1(0)}\bigg(\frac{t}{\sqrt{a^2 \cos^2(\theta) + b^2 \sin^2(\theta)}},\atan\Big(\frac{b}{a} \tan(\theta)\Big)\bigg) \\
& = \frac{2 a b}{a^2 \cos^2(\theta) + b^2 \sin^2(\theta)} \, \sqrt{a^2 \cos^2(\theta) + b^2 \sin^2(\theta) - t^2} \, \rect{\sqrt{a^2 \cos^2(\theta) + b^2 \sin^2(\theta)}}(t).
\end{align*}
Furthermore, by defining
\begin{equation*}
c_{a,b,\varphi}(\theta) = \sqrt{a^2 \cos^2(\theta - \varphi) + b^2 \sin^2(\theta - \varphi)}
\end{equation*}
the rotation property of the Radon transform shows that
\begin{equation*}
\Radon g_\varphi(t,\theta) = \Radon g(t, \theta - \varphi) = \frac{2 a b}{c_{a,b,\varphi}^2(\theta)} \, \sqrt{c_{a,b,\varphi}^2(\theta) - t^2} \, \rect{c_{a,b,\varphi}(\theta)}(t).
\end{equation*}
Consequently, by applying the shift property of the Radon transform and setting
\begin{equation*}
t_{h,k}(t,\theta) = t - h \cos(\theta) - k \sin(\theta),
\end{equation*}
for the Radon transform $\Radon f_e$ of $f_e$ follows that
\begin{equation*}
\Radon f_e(t,\theta) = \Radon g_\varphi(t - h \cos(\theta) - k \sin(\theta),\theta) = \frac{2 a b}{c_{a,b,\varphi}^2(\theta)} \, \sqrt{c_{a,b,\varphi}^2(\theta) - t_{h,k}^2(t,\theta)} \, \rect{c_{a,b,\varphi}(\theta)}(t_{h,k}(t,\theta)).
\end{equation*}
\end{example}

Based on Example~\ref{ex:Radon_ellipse} we finally consider two so called {\em mathematical phantoms}, i.e., test cases for numerical simulations, where the Radon transform can be computed analytically.
The first is the popular {\em Shepp-Logan phantom}, which was introduced in~\cite{Shepp1974} and is the superposition of ten ellipses to sketch a cross section of the human head, see Figure~\ref{fig:shepp-logan_phantom}.
Its sinogram, i.e., its Radon transform in the rectangular coordinate system $\R \times [0,\pi)$, is displayed in Figure~\ref{fig:shepp-logan_sinogram}.

\begin{figure}[t]
\centering
\subfigure[Phantom]{\label{fig:shepp-logan_phantom}\includegraphics[viewport=69 17 383 332, height=0.25\textwidth]{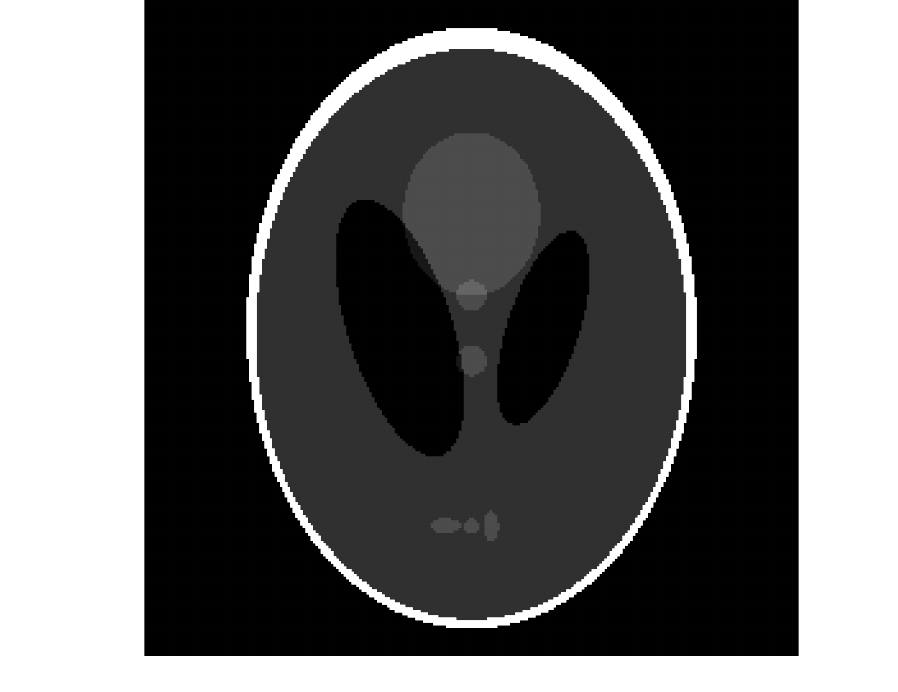}}
\hfil
\subfigure[Sinogram]{\label{fig:shepp-logan_sinogram}\includegraphics[viewport=69 17 383 332, height=0.25\textwidth]{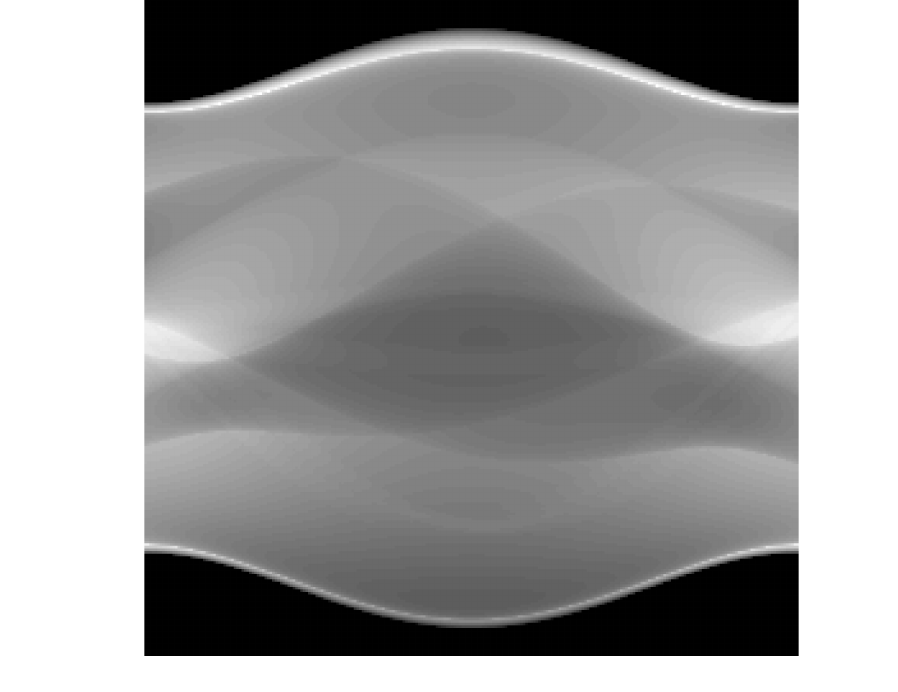}}
\caption{The Shepp-Logan phantom and its sinogram.}
\label{fig:shepp-logan_phantom_sinogram}
\end{figure}

Another mathematical phantom is displayed in Figure~\ref{fig:thorax_phantom}.
It consists of seven ellipses and sketches a cross section of the human thorax.
Its sinogram is displayed in Figure~\ref{fig:thorax_sinogram}.

\begin{figure}[b]
\centering
\subfigure[Phantom]{\label{fig:thorax_phantom}\includegraphics[viewport=69 17 383 332, height=0.25\textwidth]{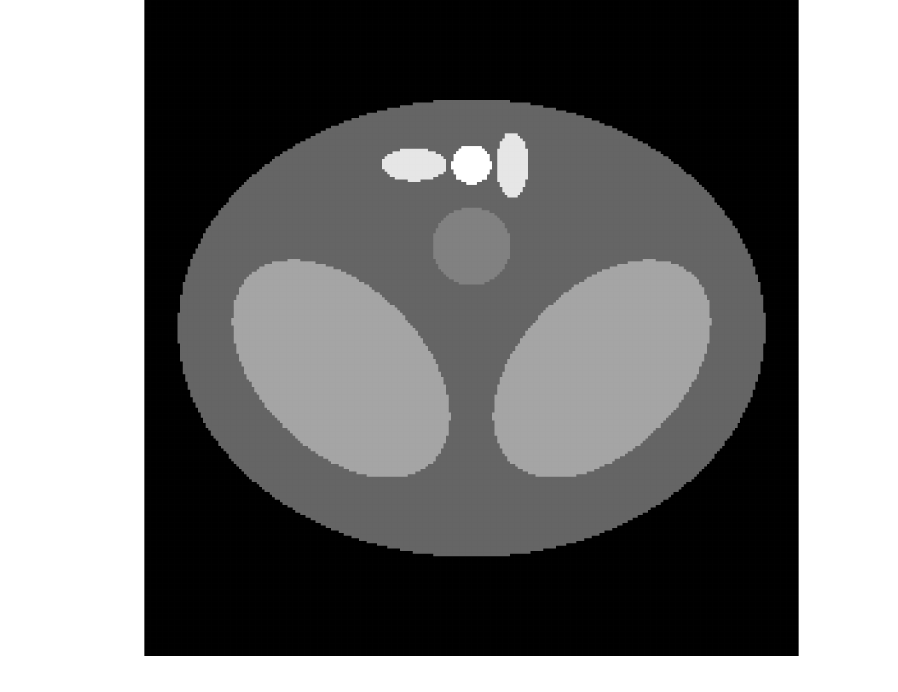}}
\hfil
\subfigure[Sinogram]{\label{fig:thorax_sinogram}\includegraphics[viewport=69 17 383 332, height=0.25\textwidth]{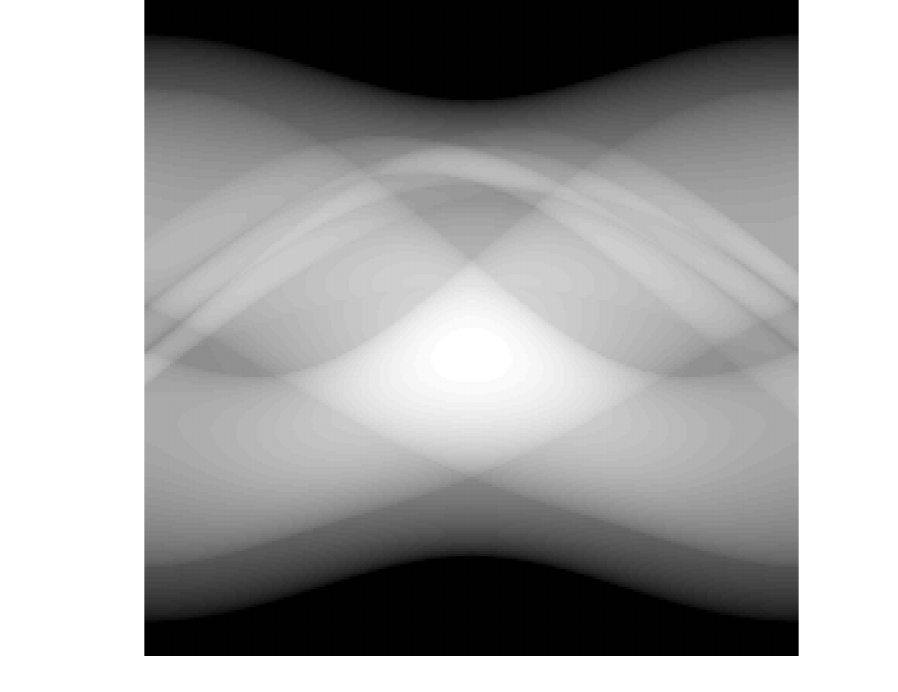}}
\caption{The thorax-shaped phantom and its sinogram.}
\label{fig:thorax_phantom_sinogram}
\end{figure}

\section{Back projection}

We want to recover the function $f \equiv f(x,y)$ from the values $\Radon f(t,\theta)$ with $t \in \R$ and $\theta \in [0,\pi)$.
First, we observe that each fixed point $(x_0,y_0) \in \R^2$ lies on infinitely many different lines $\ell_{t,\theta}$.
But for a fixed angle $\theta \in [0,\pi)$ there exists exactly one $t \in \R$ for which $\ell_{t,\theta}$ passes through the point $(x_0,y_0)$.
Indeed, for suitable $s \in \R$ we have the representation
\begin{equation*}
x_0 = t \cos(\theta) - s \sin(\theta)
\quad \mbox{ and } \quad
y_0 = t \sin(\theta) + s \cos(\theta)
\end{equation*}
if and only if $t = x_0 \cos(\theta) + y_0 \sin(\theta)$, see Remark~\ref{rem:points_on_lines}.
Consequently, the lines passing through $(x_0,y_0)$ are of the form
\begin{equation*}
\ell_{x_0 \cos(\theta) + y_0 \sin(\theta),\theta}
\quad \mbox{ for } \theta \in [0,\pi).
\end{equation*}
The first naive idea is now the following:
To recover $f(x_0,y_0)$, we compute the average value of the line integrals
\begin{equation*}
\Radon f(x_0 \cos(\theta) + y_0 \sin(\theta),\theta)
\quad \mbox{ for } \theta \in [0,\pi)
\end{equation*}
over all lines passing through $(x_0,y_0)$.
This operation is called {\em back projection}.

\begin{definition}[Back projection]
Let $g \in \L^1(\R \times [0,\pi))$ be a bivariate function in polar coordinates.
Then, the {\em back projection} $\Back g$ of $g$ at the point $(x,y) \in \R^2$ is defined as
\begin{equation*}
\Back g(x,y) = \frac{1}{\pi} \int_0^{\pi} g(x \cos(\theta) + y \sin(\theta),\theta) \: \d \theta.
\end{equation*}
\end{definition}

The following proposition shows that the back projection $\Back g$ of a function $g \in \L^1(\R \times [0,\pi))$ is defined almost everywhere and locally integrable on $\R^2$.

\begin{proposition}[Mapping property of the back projection]
For $g \in \L^1(\R \times [0,\pi))$, the back projection $\Back g$ is defined almost everywhere on $\R^2$ and satisfies $\Back g \in \L_\loc^1(\R^2)$.
\end{proposition}

\begin{proof}
Let $g \in \L^1(\R \times [0,\pi))$.
We define the auxiliary function $H: \R^2 \times [0,\pi) \to \R \times [0,\pi)$ as
\begin{equation*}
H(x,y,\theta) = (x \cos(\theta) + y \sin(\theta),\theta)
\quad \mbox{ for } (x,y,\theta) \in \R^2 \times [0,\pi).
\end{equation*}
Then, $H$ is continuous on $\R^2 \times [0,\pi)$ and we have
\begin{equation*}
\Back g(x,y) = \frac{1}{\pi} \int_0^\pi g(H(x,y,\theta)) \: \d \theta
\quad \forall \, (x,y) \in \R^2.
\end{equation*}
Now, let $K \subset \R^2$ be an arbitrary compact subset of $\R^2$.
We then obtain
\begin{align*}
\|\Back g\|_{\L^1(K)} & = \int_K |\Back g(x,y)| \: \d (x,y) = \int_K \bigg| \frac{1}{\pi} \int_0^\pi g(H(x,y,\theta)) \: \d \theta \bigg| \: \d (x,y) \\
& = \frac{1}{\pi} \int_\R \int_\R \bigg| \int_0^\pi g(H(x,y,\theta)) \: \d \theta \bigg| \, \chi_K(x,y) \: \d x \, \d y \\
& \leq \frac{1}{\pi} \int_\R \int_\R \bigg( \int_0^\pi |g(H(x,y,\theta))| \: \d \theta \bigg) \, \chi_K(x,y) \: \d x \, \d y,
\end{align*}
where $\chi_K$ denotes the characteristic function of $K$.
By applying Fubini's theorem for non-negative functions and integration by substitution for real-valued functions with
\begin{equation*}
x = t \cos(\theta) - s \sin(\theta)
\quad \mbox{ and } \quad
y = t \sin(\theta) + s \cos(\theta),
\end{equation*}
i.e., $\d x \, \d y = \d t \, \d s$ and
\begin{equation*}
t = x \cos(\theta) + y \sin(\theta)
\quad \mbox{ and } \quad
s = -x \sin(\theta) + y \cos(\theta),
\end{equation*}
we obtain
\begin{align*}
\|\Back g\|_{\L^1(K)} & \leq \frac{1}{\pi} \int_0^\pi \int_\R \int_\R |g(x \cos(\theta) + y \sin(\theta),\theta)| \, \chi_K(x,y) \: \d x \, \d y \, \d \theta \\
& = \frac{1}{\pi} \int_0^\pi \int_\R \int_\R |g(t,\theta)| \, \chi_K(t \cos(\theta) - s \sin(\theta),t \sin(\theta) + s \cos(\theta)) \: \d t \, \d s \, \d \theta.
\end{align*}

\bigbreak

Using again Fubini's theorem, the definition of the Radon transform $\Radon$ yields
\begin{align*}
\|\Back g\|_{\L^1(K)} & \leq \frac{1}{\pi} \int_0^\pi \int_\R |g(t,\theta)| \, \bigg( \int_\R \chi_K(t \cos(\theta) - s \sin(\theta),t \sin(\theta) + s \cos(\theta)) \: d s \bigg) \: \d t \, \d \theta \\
& = \frac{1}{\pi} \int_0^\pi \int_\R |g(t,\theta)| \, \Radon \chi_K(t,\theta) \: \d t \, \d \theta \leq \frac{1}{\pi} \, \diam(K) \, \|g\|_{\L^1(\R \times [0,\pi))} < \infty.
\end{align*}
Consequently, we have $\Back g \in \L^1(K)$ for all compact subsets $K \subset \R^2$.
In particular, this shows that the back projection $\Back g$ of $g \in \L^1(\R \times [0,\pi))$ is defined almost everywhere on $\R^2$ and satisfies 
\begin{equation*}
\Back g \in \L_\loc^1(\R^2).
\qedhere
\end{equation*}
\end{proof}

Note that the back projection $\Back g$ of an essentially bounded function $g \in \L^\infty(\R \times [0,\pi))$ is also defined almost everywhere and essentially bounded on $\R^2$.
Moreover, the back projection is continuous as a mapping
\begin{equation*}
\Back: \L^\infty(\R \times [0,\pi)) \to \L^\infty(\R^2),
\end{equation*}
where for $g \in \L^\infty(\R \times [0,\pi))$
\begin{equation*}
\|\Back g\|_{\L^\infty(\R^2)} \leq \|g\|_{\L^\infty(\R \times [0,\pi))}.
\end{equation*}

We continue with some basic properties of the back projection operator $\Back$.

\begin{observation}
The back projection operator $\Back$ maps a bivariate function $g \equiv g(t,\theta)$ in polar coordinates onto a bivariate function $\Back g \equiv \Back g(x,y)$ in Cartesian coordinates.
\end{observation}

As the Radon transform $\Radon$, the back projection operator $\Back$ is a positive linear operator.

\begin{proposition}[Linearity of the back projection]
The back projection $\Back$ is a positive linear integral operator, i.e., for all $\alpha,\beta \in \R$ and $g,h \in \L^1(\R \times [0,\pi))$ we have
\begin{equation*}
\Back (\alpha g + \beta h) = \alpha \, \Back g + \beta \, \Back h
\end{equation*}
and
\begin{equation*}
g \geq 0
\quad \implies \quad
\Back g \geq 0.
\end{equation*}
\end{proposition}

\begin{proof}
The statement follows from the positivity and linearity of the integral.
\end{proof}

However, an inspection of Figure~\ref{fig:BP_shepp-logan} shows that the back projection is {\em not} the inverse of the Radon transform.
Instead, we have to apply a {\em filtered} back projection, as we will see later.

\begin{observation}
\label{obs:Back_not_inverse}
The back projection $\Back$ is not the inverse of the Radon transform $\Radon$.
\end{observation}

\begin{figure}[hb]
\centering
\subfigure[Phantom]{\includegraphics[viewport=50 55 257 261, height=0.25\textwidth]{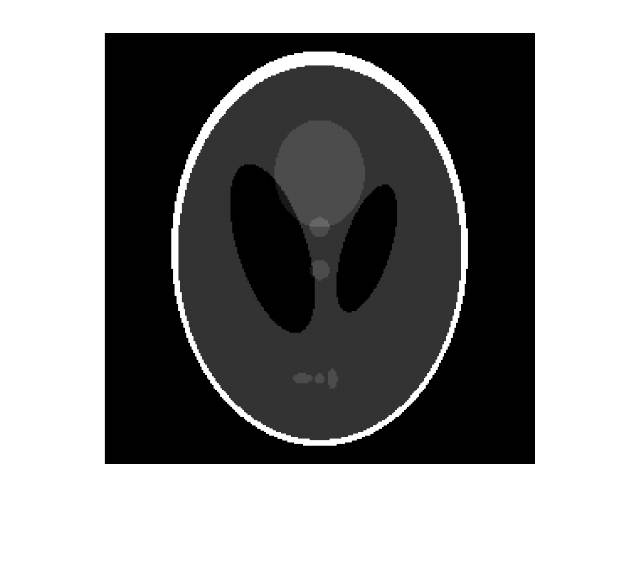}}
\hfil
\subfigure[Back projection]{\includegraphics[viewport=50 55 257 261, height=0.25\textwidth]{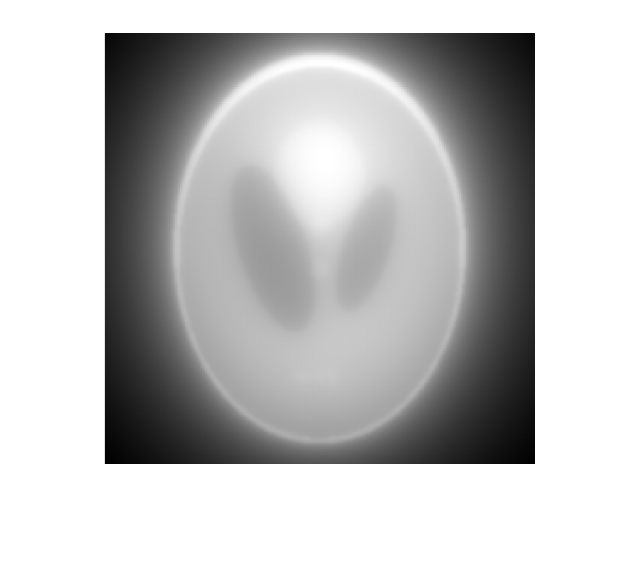}}
\caption{Back projection of the Shepp-Logan phantom.}
\label{fig:BP_shepp-logan}
\end{figure}

We close this paragraph on the back projection operator by stating the following useful relation between the convolution product, the back projection and the Radon transform.
Recall that the convolution product of two bivariate functions $f \equiv f(x,y), \, g \equiv g(x,y) \in \L^1(\R^2)$ in Cartesian coordinates is given by
\begin{equation*}
(f * g)(x,y) = \int_\R \int_\R f(X,Y) \, g(x-X,y-Y) \: \d X \, \d Y
\quad \mbox{ for } (x,y) \in \R^2.
\end{equation*}
Further, we define the convolution product of two bivariate functions $h \equiv h(t,\theta), \, k \equiv k(t,\theta)$ in polar coordinates satisfying $h(\cdot,\theta), \, k(\cdot,\theta) \in \L^1(\R)$ for all $\theta \in [0,\pi)$ as
\begin{equation*}
(h * k)(t,\theta) = \int_\R h(S,\theta) \, k(t-S,\theta) \: \d S
\quad \mbox{ for } (t,\theta) \in \R \times [0,\pi).
\end{equation*}

\bigbreak

The result now reads as follows.

\begin{theorem}
Let $f \equiv f(x,y) \in \L^1(\R^2)$ be a bivariate function in Cartesian coordinates and let $g \equiv g(t,\theta) \in \L^\infty(\R \times [0,\pi))$ be a function in polar coordinates.
Then, we have
\begin{equation}
\Back g * f = \Back(g * \Radon f).
\label{eq:Radon_back_convolution}
\end{equation}
\end{theorem}

\begin{proof}
First of all, we note that for $f \in \L^1(\R^2)$ and $g \in \L^\infty(\R \times [0,\pi))$ both expressions in~\eqref{eq:Radon_back_convolution} are well-defined as functions in $\L^\infty(\R^2)$.
Indeed, for $g \in \L^\infty(\R \times [0,\pi))$ its back projection $\Back g$ is essentially bounded on $\R^2$, i.e., $\Back g \in \L^\infty(\R^2)$, and, therefore, Young's inequality shows that
\begin{equation*}
\Back g * f \in \L^\infty(\R^2).
\end{equation*}
On the other hand, for $f \in \L^1(\R^2)$ we have $\Radon f(\cdot,\theta) \in \L^1(\R)$ for all $\theta \in [0,\pi)$ according to Proposition~\ref{prop:Radon_L1_norm} with
\begin{equation*}
\|\Radon f(\cdot,\theta)\|_{\L^1(\R)} \leq \|f\|_{\L^1(\R^2)}
\quad \forall \, \theta \in [0,\pi).
\end{equation*}
Consequently, the convolution product of $\Radon f$ and $g$ is in $\L^\infty(\R \times [0,\pi))$ and we also obtain
\begin{equation*}
\Back(g * \Radon f) \in \L^\infty(\R^2).
\end{equation*}

Now, for $(X,Y) \in \R^2$ the definitions of the back projection and the convolution product give
\begin{equation*}
(\Back g * f)(X,Y) = \int_{\R} \int_{\R} \Back g(X-x,Y-y) \, f(x,y) \: \d x \, \d y,
\end{equation*}
where
\begin{equation*}
\Back g(X-x,Y-y) = \frac{1}{\pi} \int_0^{\pi} g((X-x) \cos(\theta) + (Y-y)\sin(\theta),\theta) \: \d \theta.
\end{equation*}
By substituting
\begin{equation*}
t = x \cos(\theta) + y \sin(\theta)
\quad \mbox{ and } \quad
s = -x \sin(\theta) + y \cos(\theta),
\end{equation*}
we get $\d x \, \d y = \d s \, \d t$ and, therefore, with
\begin{equation*}
\Radon f(t,\theta) = \int_\R f(t \cos(\theta) - s \sin(\theta),t \sin(\theta) + s \cos(\theta)) \: \d s
\end{equation*}
we can conclude that
\begin{align*}
(\Back g * f)(X,Y) & = \frac{1}{\pi} \int_0^{\pi} \int_{\R} g(X \cos(\theta) + Y \sin(\theta) - t,\theta) \, \Radon f(t,\theta) \: \d t \, \d \theta \\
& = \frac{1}{\pi} \int_0^{\pi} (g * \Radon f)(X \cos(\theta) + Y \sin(\theta),\theta) \: \d \theta \\
& = \Back(g * \Radon f)(X,Y),
\end{align*}
as stated.
\end{proof}

\chapter{Inversion of the Radon transform}

This chapter deals with the inversion of the Radon transform.
To this end, we first state the Fourier slice theorem resulting in the classical {\em filtered back projection formula}, which yields an analytical inversion formula and is the basis for one of the most commonly used reconstruction algorithms in computerized tomography.
The inversion, however, is numerically unstable and we study the degree of ill-posedness of the CT reconstruction problem to explain this fact.

\section{Fourier slice theorem}

One of the most important properties of the Radon transform is given by the classical {\em Fourier slice theorem} (FST), also known as {\em central slice theorem}, which relates the Fourier transform of the Radon transform to the Fourier transform of the function to be reconstructed.

Let us first recall that the Fourier transform $\Fourier f$ of a bivariate function $f \equiv f(x,y) \in \L^1(\R^2)$ in Cartesian coordinates is given by
\begin{equation*}
\Fourier f(X,Y) = \int_\R \int_\R f(x,y) \, \e^{-\i (xX + yY)} \: \d x \, \d y
\quad \mbox{ for } (X,Y) \in \R^2.
\end{equation*}
For a bivariate function $h \equiv h(t,\theta)$ in polar coordinates satisfying $h(\cdot,\theta) \in \L^1(\R)$ for all $\theta \in [0,\pi)$ we define the Fourier transform $\Fourier h$ as the univariate Fourier transform acting only on the radial variable $t$, i.e.,
\begin{equation*}
\Fourier h(S,\theta) = \int_\R h(t,\theta) \, \e^{-\i S t} \: \d t
\quad \mbox{ for } (S,\theta) \in \R \times [0,\pi).
\end{equation*}

Now, the Fourier slice theorem reads as follows.

\begin{theorem}[Fourier slice theorem]
\label{theo:Fourier_slice}
For $f \in \L^1(\R^2)$ we have
\begin{equation*}
\Fourier (\Radon f)(S,\theta) = \Fourier f(S \cos(\theta),S \sin(\theta))
\quad \forall \, (S,\theta) \in \R \times [0,\pi).
\end{equation*}
\end{theorem}

\begin{proof}
For $(S,\theta) \in \R \times [0,\pi)$, the definition of the two-dimensional Fourier transform yields
\begin{equation*}
\Fourier f(S \cos(\theta),S \sin(\theta)) = \int_\R \int_\R f(x,y) \, \e^{-\i S (x \cos(\theta) + y \sin(\theta))} \: \d x \, \d y.
\end{equation*}
Applying the transformation
\begin{equation*}
x = t \cos(\theta) - s \sin(\theta)
\quad \mbox{ and } \quad 
y = t \sin(\theta) + s \cos(\theta),
\end{equation*}
i.e.,
\begin{equation*}
t = x \cos(\theta) + y \sin(\theta)
\quad \mbox{ and } \quad
s = -x \sin(\theta) + y \cos(\theta),
\end{equation*}
again gives $\d x \, \d y = \d t \, \d s$ and, thus, it follows that
\begin{equation*}
\Fourier f(S \cos(\theta),S \sin(\theta)) = \int_\R \int_\R f(t \cos(\theta) - s \sin(\theta),t \sin(\theta) + s \cos(\theta)) \, \e^{-\i S t} \: \d t \, \d s = \Fourier (\Radon f)(S,\theta)
\end{equation*}
by Fubini's theorem and the definition of the Radon transform $\Radon$.
\end{proof}

The importance of the Fourier slice theorem lies in the fact that it links together the Radon transform of a function and its Fourier transform.
This connection can be used to derive properties of the Radon transform from those properties which are known for the Fourier transform.
In particular, the Fourier slice theorem shows the injectivity of the Radon transform on the domain $\L^1(\R^2)$.
Indeed, if $\Radon f$ vanished on $\R \times [0,\pi)$, then $\Fourier f$ vanished on $\R^2$, which implies that $f$ is zero due to the injectivity of the Fourier transform on $\L^1(\R^2)$.

\begin{corollary}[Injectivity of the Radon transform]
For $f \in \L^1(\R^2)$ we have
\begin{equation*}
\Radon f = 0
\quad \implies \quad
f = 0,
\end{equation*}
i.e., the Radon transform $\Radon$ is injective on $\L^1(\R^2)$.
\end{corollary}

Moreover, Theorem~\ref{theo:Fourier_slice} immediately provides a scheme for reconstructing a function from the knowledge of its Radon transform.
Assuming $\Radon f(t,\theta)$ to be known for all $(t,\theta) \in \R \times [0,\pi)$, we can gain knowledge about the two-dimensional Fourier transform of $f$ by computing the one-dimensional Fourier transform of $\Radon f$.
Subsequent application of the inverse two-dimensional Fourier transform would yield the function $f$ we want to reconstruct.
We remark that such reconstruction procedures are known as {\em Fourier reconstruction methods}, cf.~\cite{Natterer2001, Natterer2001a}.

\section{Filtered back projection formula}

Based on the Fourier slice Theorem~\ref{theo:Fourier_slice}, we are now prepared to prove an inversion formula for the Radon transform, which is given by the classical {\em filtered back projection (FBP) formula}.

\begin{theorem}[Filtered back projection formula]
\label{theo:filtered_back_projection}
For $f \in \L^1(\R^2) \cap \Cont(\R^2)$ with $\Fourier f \in \L^1(\R^2)$ the {\em filtered back projection formula}
\begin{equation}
f(x,y) = \frac{1}{2} \, \Back\big(\Fourier^{-1}[|S| \Fourier(\Radon f)(S,\theta)]\big)(x,y)
\label{eq:FBP_formula}
\end{equation}
holds for all $(x,y) \in \R^2$.
\end{theorem}

We remark that the FBP formula is also valid under weaker assumptions on the function $f$.
For the purpose of this course, however, the presented version is sufficient.

\begin{proof}
Let $f \in \L^1(\R^2) \cap \Cont(\R^2)$ with $\Fourier f \in \L^1(\R^2)$ and let $(x,y) \in \R^2$ be fixed.
Applying the two-dimensional Fourier inversion formula to $f$ yields the identity
\begin{equation*}
f(x,y) = \Fourier^{-1}(\Fourier f)(x,y) = \frac{1}{4\pi^2} \int_{\R} \int_{\R} \Fourier f(X,Y) \, \e^{\i (xX + yY)} \: \d X \, \d Y.
\end{equation*}
By changing the variables $(X,Y) \in \R^2$ from Cartesian coordinates to $(S,\theta) \in \R \times [0,\pi)$ in polar coordinates, i.e.,
\begin{equation*}
X = S \cos(\theta)
\quad \mbox{ and } \quad
Y = S \sin(\theta),
\end{equation*}
we get $\d X \, \d Y = |S| \: \d S \, \d \theta$.
Thus, with the Fourier slice Theorem~\ref{theo:Fourier_slice} follows that
\begin{align*}
f(x,y) & = \frac{1}{4\pi^2} \int_0^{\pi} \int_{\R} \Fourier f(S \cos(\theta),S \sin(\theta)) \, \e^{\i S (x \cos(\theta) + y \sin(\theta))} \, |S| \: \d S \, \d \theta \\
& \stackrel{\mathclap{\text{FST}}}{=} \frac{1}{4\pi^2} \int_0^{\pi} \int_{\R} \Fourier (\Radon f)(S,\theta) \, \e^{\i S (x \cos(\theta) + y \sin(\theta))} \, |S| \: \d S \, \d \theta \\
& = \frac{1}{2\pi} \int_0^{\pi} \Fourier^{-1}[|S| \, \Fourier (\Radon f)(S,\theta)](x \cos(\theta) + y \sin(\theta),\theta) \: \d \theta \\
& = \frac{1}{2} \, \Back\big(\Fourier^{-1}[|S| \Fourier(\Radon f)(S,\theta)]\big)(x,y)
\end{align*}
due to the definition of the back projection.
\end{proof}

We remark that without the factor $|S|$ in~\eqref{eq:FBP_formula}, the Fourier transform and its inverse would cancel out and the FBP formula world reduce to simply applying the back projection operator $\Back$ to the Radon data $\Radon f$.
However, as we have seen in Observation~\ref{obs:Back_not_inverse}, this is not sufficient for the exact recovery of the function $f$.

\begin{remark}
The FBP formula~\eqref{eq:FBP_formula} reveals that multiplying the Fourier transform of $\Radon f$ with $|S|$ and applying the inverse Fourier transform is essential before back projecting the Radon data.
In the language of signal processing, we say that the Radon data $\Radon f$ is {\em filtered} by the multiplication with the (radial) filter $|S|$ in Fourier domain, which also explains the expression {\em filtered back projection}.
\end{remark}

With Theorem~\ref{theo:filtered_back_projection} the stated CT reconstruction Problem~\ref{prob:CT_reconstruction_problem_Radon} is solved {\em analytically}.
In practice, however, the application of the FBP formula~\eqref{eq:FBP_formula} causes severe numerical problems.

\begin{observation}[FBP is unstable]
By the application of the filter $|S|$ to the Fourier transform $\Fourier(\Radon f)$ in~\eqref{eq:FBP_formula}, especially the high frequency components in $\Radon f$ are amplified by the magnitude of $|S|$.
Since noise mainly consists of high frequencies, this shows that the filtered back projection formula is highly sensitive with respect to noise and, thus, numerically unstable.
In practice, a direct application of the FBP formula would lead to undesired corruptions in the reconstruction.
\end{observation}

\begin{figure}[t]
\centering
\subfigure[Radon data]{\label{fig:shepp-logan_Radon}\includegraphics[viewport=40 58 210 335, height=0.25\textwidth]{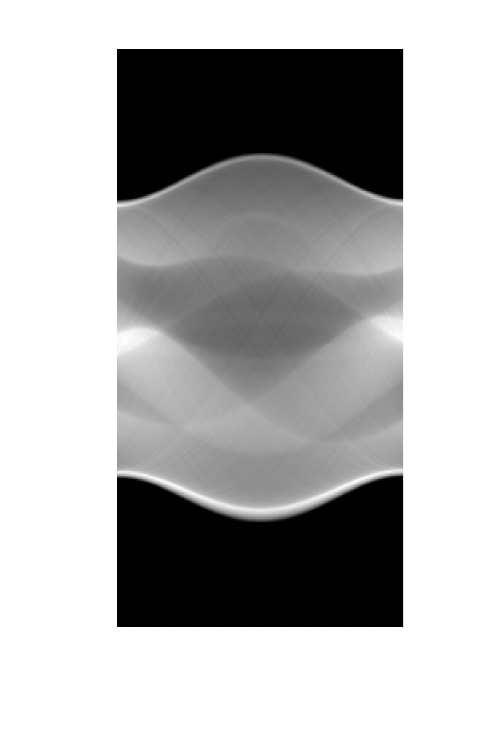}}
\hfill
\subfigure[Noiseless reconstruction]{\label{fig:shepp-logan_reconstruction}\includegraphics[viewport=50 55 257 261, height=0.25\textwidth]{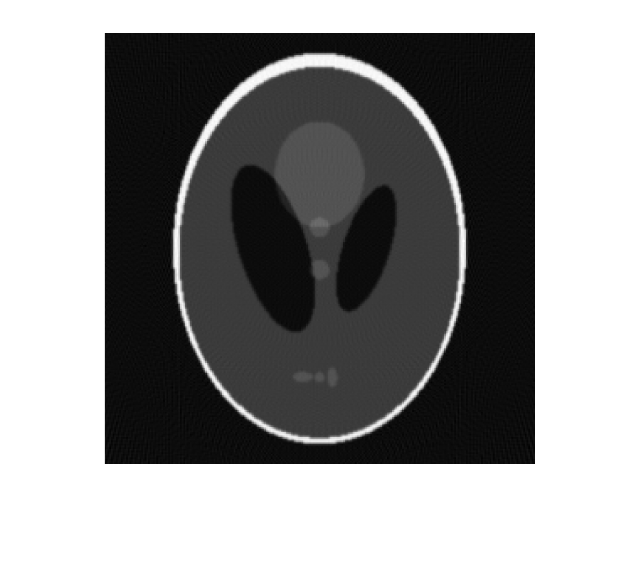}}
\hfill
\subfigure[Noisy data]{\label{fig:shepp-logan_Radon_noisy}\includegraphics[viewport=40 58 210 335, height=0.25\textwidth]{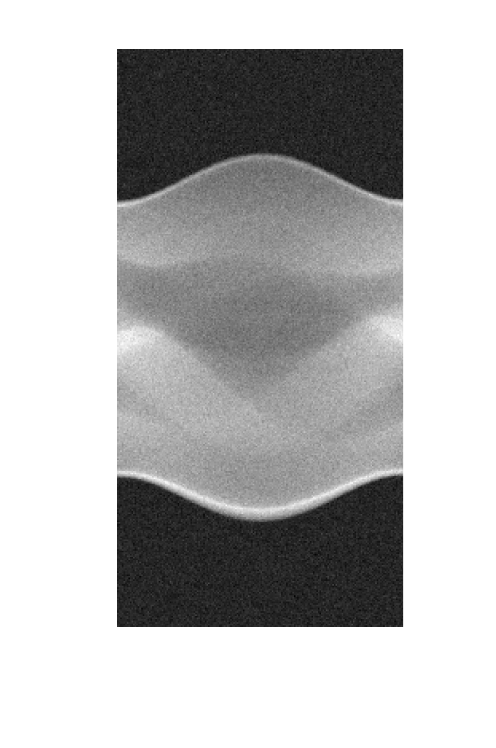}}
\hfill
\subfigure[Noisy reconstruction]{\label{fig:shepp-logan_reconstruction_noisy}\includegraphics[viewport=50 55 257 261, height=0.25\textwidth]{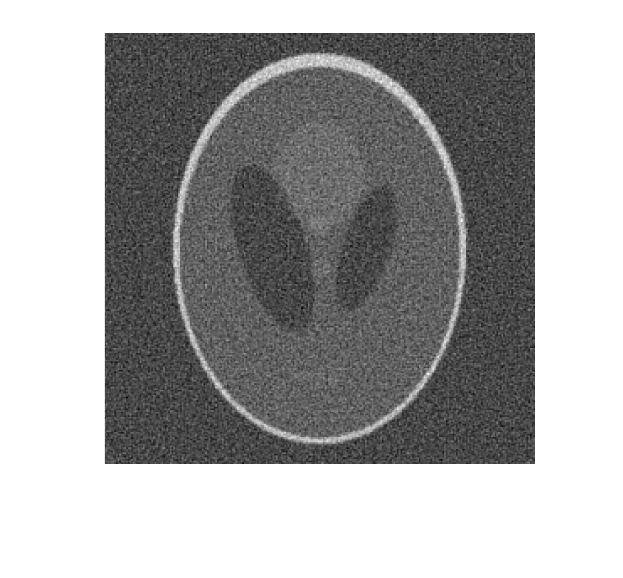}}
\caption{Noise amplification of the filtered back projection formula.}
\label{fig:noise_amplification}
\end{figure}

To illustrate this observation we consider the reconstruction of the Shepp-Logan phantom from {\em noisy} Radon data, see Figure~\ref{fig:noise_amplification}.
The noiseless Radon data is shown in Figure~\ref{fig:shepp-logan_Radon} and the FBP reconstruction from this data can be seen in Figure~\ref{fig:shepp-logan_reconstruction}.
In Figure~\ref{fig:shepp-logan_Radon_noisy} 10\% white Gaussian noise is added to the Radon data and this noisy data is used for the FBP reconstruction in Figure~\ref{fig:shepp-logan_reconstruction_noisy}.
As expected, we clearly observe a severe amplification of the noise.

\smallskip

The noise sensibility of the inversion formula can be explained by analysing the ill-posedness of the CT reconstruction Problem~\ref{prob:CT_reconstruction_problem_Radon}.
Before that, we first want to pass the following remark.

\begin{remark}
The filtered back projection formula assumes the Radon data $\Radon f(t,\theta)$ to be available for all straight lines $\ell_{t,\theta}$ in the plane.
But in practice, only a finite number of X-ray samples are taken and, consequently, we have to recover the function $f$ from a finite set of Radon data
\begin{equation*}
\bigl\{ \Radon f(t_j,\theta_j) \bigm| j = 1,\ldots,M \bigr\}
\quad \mbox{ for some } M \in \N.
\end{equation*}
The reconstruction from discrete Radon data will be discussed in later chapters.
\end{remark}

\section{Ill-posedness}

In the previous paragraph we have presented an inversion formula for the Radon transform, given by the filtered back projection formula.
For the application of this reconstruction formula to real data, we have to assume that perfect Radon data are given.
However, this is not a realistic scenario, because in practice the data is always corrupted by noise.

\bigbreak

Therefore, it is important to know how small perturbations in the measurement data are propagated through the reconstruction process.
In particular, we have to analyse whether the inversion of the Radon transform is continuous, i.e., whether small measurement errors only lead to small errors in the reconstruction.
This leads us to Hadamard's classical definition of {\em well-posed} problems and the notion of {\em ill-posedness}.

For a comprehensive treatment of ill-posed problems and the theory of inverse problems and their regularization, we refer the reader to the monograph~\cite{Engl2000}.

\begin{definition}[Well-posedness]
Let $A: \X \to \Y$ be a mapping between topological spaces $\X$ and $\Y$.
The problem $Ax = y$ is called {\em well-posed} in the sense of Hadamard if the following three conditions are satisfies:
\begin{enumerate}
\item {\em Existence:} For every $y \in \Y$ there exists an $x \in \X$ such that $Ax = y$.

\item {\em Uniqueness:} For every $y \in \Y$ the solution $x \in \X$ of $Ax = y$ is uniquely determined.

\item {\em Stability:} The inverse mapping $A^{-1}: \Y \to \X$ is continuous, i.e., the solution $x \in \X$ depends continuously on the data $y \in \Y$.
\end{enumerate}
If at least one of the above conditions is violated, the problem is called {\em ill-posed}.
\end{definition}

A standard method to classify the degree of ill-posedness of an operator $A$ is based on the decay behaviour of its singular values, see~\cite[Section 2.2]{Engl2000}.
If $A$ is defined on $\L^2$-spaces, another classification of ill-posedness can be defined in terms of Sobolev spaces $\H^\alpha$.

\begin{definition}[Degree of ill-posedness]
\label{def:degree_ill_posedness}
Let $A: \X \to \Y$ be an operator between $\L^2$-spaces $\X$ and $\Y$.
The problem $Ax = y$ is called {\em ill-posed of degree $\alpha > 0$} if, for some $C_1, C_2 > 0$,
\begin{equation*}
C_1 \, \|x\|_{\L^2} \leq \|Ax\|_{\H^\alpha} \leq C_2 \, \|x\|_{\L^2}.
\end{equation*}
\end{definition}

We now state a continuity result for the Radon transform on Sobolev spaces of fractional order, which implies that the CT reconstruction Problem~\ref{prob:CT_reconstruction_problem_Radon} is ill-posed of degree $\frac{1}{2}$.

Let us first recall the definitions of the involved Sobolev spaces.
For functions $f \equiv f(x,y)$ in Cartesian coordinates $(x,y) \in \R^2$ 
the Sobolev space $\H^\alpha(\R^2)$ of order $\alpha \in \R$, defined as
\begin{equation*}
\H^\alpha(\R^2) = \bigl\{ f \in \Schwartz'(\R^2) \bigm| \|f\|_{\H^\alpha(\R^2)} < \infty \bigr\},
\end{equation*}
is equipped with the norm
\begin{equation*}
\|f\|_{\H^\alpha(\R^2)} = \bigg(\int_{\R} \int_{\R} (1+X^2+Y^2)^{\alpha} \, |\Fourier f(X,Y)|^2 \: \d X \, \d Y\bigg)^{\nicefrac{1}{2}}.
\end{equation*}
Further, for an open subset $\Omega \subset \R^2$ the space $\H_0^\alpha(\Omega)$ consists of those Sobolev functions with support in $\overline{\Omega}$, i.e.,
\begin{equation*}
\H_0^\alpha(\Omega) = \bigl\{ f \in \H^\alpha(\R^2) \bigm| \supp(f) \subset \overline{\Omega} \bigr\}.
\end{equation*}

For functions $g \equiv g(t,\theta)$ in polar coordinates $(t,\theta) \in \R \times [0,\pi)$, we define the Sobolev space $\H^\alpha(\R \times [0,\pi))$, $\alpha \in \R$, as the space of all functions $g$ with $g(\cdot,\theta) \in \H^\alpha(\R)$ for almost all $\theta \in [0,\pi)$ and
\begin{equation*}
\|g\|_{\H^\alpha(\R \times [0,\pi))} = \bigg(\int_0^{\pi} \int_{\R} (1+S^2)^{\alpha} \, |\Fourier g(S,\theta)|^2 \: \d S \, \d \theta\bigg)^{\nicefrac{1}{2}} < \infty.
\end{equation*}

Then, we obtain the following continuity result for the Radon transform.
Here, it is essential to assume that we deal with functions with compact support.

\begin{theorem}
\label{theo:Radon_Sobolev_norm}
Let $\Omega \subset \R^2$ be an open and bounded set and let $f \in \L^1(\R^2) \cap \H_0^\alpha(\Omega)$ with $\alpha \in \R$.
Then, we have $\Radon f \in \H^{\alpha+\nicefrac{1}{2}}(\R \times [0,\pi))$ and there exists a constant $C_{\alpha} > 1$ such that
\begin{equation*}
\|f\|_{\H^\alpha(\R^2)} \leq \|\Radon f\|_{\H^{\alpha+\nicefrac{1}{2}}(\R \times [0,\pi))} \leq C_{\alpha} \, \|f\|_{\H^\alpha(\R^2)}.
\end{equation*}
\end{theorem}

\begin{proof}
Since $f \in \L^1(\R^2)$, the Fourier slice Theorem~\ref{theo:Fourier_slice} gives
\begin{equation*}
\Fourier(\Radon f)(S,\theta) = \Fourier f(S \cos(\theta),S \sin(\theta))
\quad \forall \, (S,\theta) \in \R \times [0,\pi)
\end{equation*}
and we obtain
\begin{equation*}
\|\Radon f\|_{\H^{\alpha+\nicefrac{1}{2}}(\R \times [0,\pi))}^2 = \int_0^{\pi} \int_{\R} (1+S^2)^{\alpha+\frac{1}{2}} \, |\Fourier f(S \cos(\theta),S \sin(\theta))|^2 \: \d S \, \d \theta.
\end{equation*}
Using the estimate
\begin{equation*}
(1+S^2)^{\frac{1}{2}} \geq |S|
\quad \forall \, S \in \R,
\end{equation*}
this implies that
\begin{equation*}
\|\Radon f\|_{\H^{\alpha+\nicefrac{1}{2}}(\R \times [0,\pi))}^2 \geq \int_0^{\pi} \int_{\R} (1+S^2)^\alpha \, |\Fourier f(S \cos(\theta),S \sin(\theta))|^2 \, |S| \: \d S \, \d \theta.
\end{equation*}
By applying the transformation
\begin{equation*}
X = S \cos(\theta)
\quad \mbox{ and } \quad
Y = S \sin(\theta),
\end{equation*}
we get $\d X \, \d Y = |S| \: \d S \, \d \theta$ and with $S^2 = X^2 + Y^2$ follows that
\begin{equation*}
\|\Radon f\|_{\H^{\alpha+\nicefrac{1}{2}}(\R \times [0,\pi))}^2 \geq \int_{\R} \int_{\R} (1+X^2+Y^2)^\alpha \, |\Fourier f(X,Y)|^2 \: \d X \, \d Y = \|f\|_{\H^\alpha(\R^2)}^2,
\end{equation*}
which proves the first inequality.
On the other hand, we have
\begin{align*}
\|\Radon f\|_{\H^{\alpha+\nicefrac{1}{2}}(\R \times [0,\pi))}^2 & = \int_0^{\pi} \int_{\R} (1+S^2)^{\alpha+\frac{1}{2}} \, |\Fourier f(S \cos(\theta),S \sin(\theta))|^2 \: \d S \, \d \theta \\
& = \int_{\R} \int_{\R} (1+X^2+Y^2)^{\alpha+\frac{1}{2}} \, (X^2+Y^2)^{-\frac{1}{2}} \, |\Fourier f(X,Y)|^2 \: \d X \, \d Y,
\end{align*}
where we used the transformation from above and
\begin{equation*}
|S| = \sqrt{X^2 + Y^2}.
\end{equation*}
We now split the above representation of the $\H^{\alpha+\nicefrac{1}{2}}$-norm of $\Radon f$ into the sum of two integrals,
\begin{equation*}
\|\Radon f\|_{\H^{\alpha+\nicefrac{1}{2}}(\R \times [0,\pi))}^2 = I_1 + I_2,
\end{equation*}
where we let
\begin{equation*}
I_1 = \int_{X^2+Y^2 \leq 1} (1+X^2+Y^2)^{\alpha+\frac{1}{2}} \, (X^2+Y^2)^{-\frac{1}{2}} \, |\Fourier f(X,Y)|^2 \: \d (X,Y)
\end{equation*}
and
\begin{equation*}
I_2 = \int_{X^2+Y^2 > 1} (1+X^2+Y^2)^{\alpha+\frac{1}{2}} \, (X^2+Y^2)^{-\frac{1}{2}} \, |\Fourier f(X,Y)|^2 \: \d (X,Y).
\end{equation*}
In what follows, we estimate the integrals $I_1$ and $I_2$ separately with respect to the $\H^\alpha$-norm of~$f$.
To bound the second integral $I_2$, we use the estimate
\begin{equation*}
X^2+Y^2 \geq \frac{1}{2} \, (1+X^2+Y^2)
\quad \forall \, X^2+Y^2 \geq 1.
\end{equation*}
With this, $I_2$ can be estimated from above by
\begin{align*}
I_2 & \leq \sqrt{2} \int_{X^2+Y^2 > 1} (1+X^2+Y^2)^{\alpha+\frac{1}{2}} \, (1+X^2+Y^2)^{-\frac{1}{2}} \, |\Fourier f(X,Y)|^2 \: \d (X,Y) \\
& = \sqrt{2} \int_{X^2+Y^2 > 1} (1+X^2+Y^2)^{\alpha} \, |\Fourier f(X,Y)|^2 \: \d (X,Y) \leq \sqrt{2} \|f\|_{\H^\alpha(\R^2)}^2.
\end{align*}
For the first integral $I_1$, we use the fact that $\Fourier f \in \Cont_0(\R^2)$ for $f \in \L^1(\R^2)$ and get
\begin{align*}
I_1 & \leq \int_{X^2+Y^2 \leq 1} (1+X^2+Y^2)^{\alpha+\frac{1}{2}} \, (X^2+Y^2)^{-\frac{1}{2}} \: \d (X,Y) \sup_{X^2+Y^2 \leq 1} |\Fourier f(X,Y)|^2 \\
& \leq C_1(\alpha) \sup_{X^2+Y^2 \leq 1} |\Fourier f(X,Y)|^2
\end{align*}
for some constant $0 < C_1(\alpha) < \infty$.
In order to estimate the supremum, we choose an even function $\chi \in \Cont_c^\infty(\R^2)$ which is $1$ on $\Omega$ and, for fixed $(X,Y) \in \R^2$, we define the function
\begin{equation*}
\chi_{(X,Y)}(x,y) = \e^{-\i (xX + yY)} \, \chi(x,y)
\quad \mbox{ for } (x,y) \in \R^2.
\end{equation*}
Then, its inverse Fourier transform $\Fourier^{-1} \chi_{(X,Y)}$ exists and is given by
\begin{equation*}
\Fourier^{-1} \chi_{(X,Y)}(x,y) = \Fourier^{-1} \chi(x-X,y-Y)
\quad \forall \, (x,y) \in \R^2.
\end{equation*}
Applying Parseval's identity yields
\begin{align*}
|\Fourier f(X,Y)|^2 & = \bigg|\int_{\R} \int_{\R} f(x,y) \, \e^{-\i (xX+yY)} \: \d x \, \d y\bigg|^2 = \bigg|\int_{\R} \int_{\R} \chi_{(X,Y)}(x,y) \, f(x,y) \: \d x \, \d y\bigg|^2 \\
& = \bigg|\int_{\R} \int_{\R} \Fourier^{-1} \chi_{(X,Y)}(x,y) \, \Fourier f(x,y) \: \d x \, \d y\bigg|^2 \\
& = \bigg|\int_{\R} \int_{\R} \Fourier^{-1} \chi_{(X,Y)}(x,y) \, (1+x^2+y^2)^{-\frac{\alpha}{2}} \, (1+x^2+y^2)^{\frac{\alpha}{2}} \, \Fourier f(x,y) \: \d x \, \d y\bigg|^2
\end{align*}
and with the Cauchy-Schwarz inequality follows that
\begin{align*}
|\Fourier f(X,Y)|^2 \leq \bigg(\int_{\R} \int_{\R} (1+x^2+y^2)^{-\alpha} \, |\Fourier^{-1} \chi_{(X,Y)}(x,y)|^2 \: \d x \, \d y\bigg) \, \|f\|_{\H^\alpha(\R^2)}^2,
\end{align*}
since $\chi \in \Cont_c^\infty(\R^2)$ and $f \in \H_0^\alpha(\Omega)$, i.e., $(1+x^2+y^2)^{\frac{\alpha}{2}} \, |\Fourier f| \in \L^2(\R^2)$.
We further have
\begin{equation*}
\Fourier^{-1} \chi_{(X,Y)}(x,y) = \frac{1}{4\pi^2} \, \Fourier \chi_{(X,Y)}(-x,-y)
\quad \forall \, (x,y) \in \R^2,
\end{equation*}
which implies that
\begin{equation*}
|\Fourier f(X,Y)|^2 \leq \frac{1}{16\pi^4} \, \|\chi_{(X,Y)}\|_{\H^{-\alpha}(\R^2)}^2 \, \|f\|_{\H^\alpha(\R^2)}^2
\quad \forall \, (X,Y) \in \R^2.
\end{equation*}
The $\H^{-\alpha}$-norm of $\chi_{(X,Y)}$ is a continuous function of $(X,Y)$ and, consequently, there exists a constant $C_2(\alpha) > 0$ such that
\begin{equation*}
\sup_{X^2+Y^2 \leq 1} \frac{1}{16\pi^4} \, \|\chi_{(X,Y)}\|_{\H^{-\alpha}(\R^2)}^2 \leq C_2(\alpha).
\end{equation*}
Combining the estimates yields
\begin{equation*}
\|\Radon f\|_{\H^{\alpha+\nicefrac{1}{2}}(\R \times[0,\pi))}^2 \leq \Big(\sqrt{2} + C_1(\alpha) \, C_2(\alpha)\Big) \, \|f\|_{\H^\alpha(\R^2)}^2 = C_\alpha^2 \, \|f\|_{\H^\alpha(\R^2)}^2
\end{equation*}
and the second inequality is also satisfied.
In particular, we have $\Radon f \in \H^{\alpha+\nicefrac{1}{2}}(\R \times [0,\pi))$.
\end{proof}

\begin{remark}
Theorem~\ref{theo:Radon_Sobolev_norm} shows that the Radon transform $\Radon$ admits an inverse and that this inverse is continuous as an operator from $\H^{\alpha+\nicefrac{1}{2}}(\R \times [0,\pi))$ into $\H_0^\alpha(\Omega)$.
Here, the order of the latter Sobolev space is optimal in the sense that the result is wrong for any larger index.
In particular, if $\alpha$ is chosen such that $\H^{\alpha+\nicefrac{1}{2}}(\R \times [0,\pi))$ is simply the $\L^2$-space, i.e., $\alpha = -\frac{1}{2}$, then $\H_0^\alpha(\Omega)$ is a Sobolev space of negative order, whose norm is weaker than the $\L^2$-norm on $\L^2(\Omega)$.
This shows that $\Radon^{-1}$ is not continuous in an $\L^2$-setting, i.e., the CT reconstruction Problem~\ref{prob:CT_reconstruction_problem_Radon} is {\em ill-posed} in the sense of Hadamard.
\end{remark}

We remark that choosing $\alpha = 0$ in Theorem~\ref{theo:Radon_Sobolev_norm} shows that the problem of reconstructing a bivariate function from its Radon data is ill-posed of degree $\frac{1}{2}$ in the sense of Definition~\ref{def:degree_ill_posedness}.
In particular, the Radon transform $\Radon$ smoothes by an order of $\frac{1}{2}$ in the Sobolev scale.
Since the inversion has to reverse the smoothing, the reconstruction process is unstable and regularization strategies have to be applied to stabilize the inversion.

For this purpose, we will follow a standard approach and replace the filter $|S|$ in the filtered back projection formula~\eqref{eq:FBP_formula} by a so called {\em low-pass filter} $F_L(S)$ of finite bandwidth $L$ and with a compactly supported window function.
For other approaches we refer to the standard literature on inverse problems and regularization, see, for example,~\cite{Engl2000}.

\section{Overview of inversion strategies}

To close this chapter we collect different strategies for the inversion of the Radon transform, which all suffer from the ill-posedness of the CT reconstruction problem.

\subsection*{Direct Fourier reconstruction}

The direct Fourier reconstruction method is based on the Fourier slice Theorem~\ref{theo:Fourier_slice}, i.e.,
\begin{equation*}
\Fourier (\Radon f)(S,\theta) = \Fourier f(S \cos(\theta),S \sin(\theta))
\quad \forall \, (S,\theta) \in \R \times [0,\pi),
\end{equation*}
and consists of the following steps:
\begin{enumerate}
\item[(1)] For each angle $\theta \in [0,\pi)$, compute the 1D Fourier transform of the sinogram $\Radon f$
\begin{equation*}
\hat{p}_\theta(S) = \Fourier (\Radon f) (S,\theta)
\quad \mbox{ for } S \in \R.
\end{equation*}
\item[(2)] Create polar representation of the Fourier transform of the target function $f$ via
\begin{equation*}
\hat{f}(S \cos(\theta),S \sin(\theta)) = \hat{p}_\theta(S)
\quad \mbox{ for } (S,\theta) \in \R \times [0,\pi).
\end{equation*}
\item[(3)] Convert from polar coordinates to Cartesian coordinates,
\begin{equation*}
\hat{f}(S \cos(\theta),S \sin(\theta)) \longrightarrow \hat{f}(X,Y).
\end{equation*}
\item[(4)] Compute 2D inverse Fourier transform to reconstruct $f$ via
\begin{equation*}
f(x,y) = \Fourier^{-1} \hat{f}(x,y)
\quad \mbox{ for } (x,y) \in \R^2.
\end{equation*}
\end{enumerate}

\subsection*{Filtered back projection}

The exact filtered back projection method is based on the filtered back projection formula~\eqref{eq:FBP_formula}, i.e.,
\begin{equation*}
f(x,y) = \frac{1}{2} \, \Back\big(\Fourier^{-1}[|S| \Fourier(\Radon f)(S,\theta)]\big)(x,y)
\quad \forall \, (x,y) \in \R^2,
\end{equation*}
and consists of the following steps:
\begin{enumerate}
\item[(1)] For each angle $\theta \in [0,\pi)$, compute the 1D Fourier transform of the sinogram $\Radon f$
\begin{equation*}
\hat{p}_\theta(S) = \Fourier (\Radon f) (S,\theta)
\quad \mbox{ for } S \in \R.
\end{equation*}
\item[(2)] Multiply $\hat{p}_\theta$ by the filter $|S|$ to obtain
\begin{equation*}
\hat{h}_\theta(S) = |S| \, \hat{p}_\theta(S)
\quad \mbox{ for } S \in \R.
\end{equation*}
\item[(3)] Apply 1D inverse Fourier transform to obtain the filtered sinogram
\begin{equation*}
h(S,\theta) = \Fourier^{-1} \hat{h}_\theta(S)
\quad \mbox{ for } (S,\theta) \in \R \times [0,\pi).
\end{equation*}
\item[(4)] Back project the filtered sinogram to reconstruct $f$ via
\begin{equation*}
f(x,y) = \frac{1}{2} \, \Back h(x,y)
\quad \mbox{ for } (x,y) \in \R^2.
\end{equation*}
\end{enumerate}

\subsection*{Filtering the laminogram}

Instead of first filtering the sinogram and then back projecting the result, one can also first back project the sinogram, leading to the so called {\em laminogram},  and then filter the laminogram.
This approach is called filtering the laminogram and involves the {\em Riesz potential} $\Lambda^\alpha f$ of $f \in \Schwartz(\R^2)$, defined as
\begin{equation*}
\Lambda^\alpha f = \Fourier^{-1}\big(\|\cdot\|_{\R^2}^{-\alpha} \, \Fourier f\big)
\quad \mbox{ for } \alpha < 2.
\end{equation*}

\begin{theorem}
For $f \in \Schwartz(\R^2)$ we have the inversion formula
\begin{equation*}
f(x,y) = \frac{1}{2} \, \Lambda^{-1}\big(\Back(\Radon f)\big)(x,y)
\quad \forall \, (x,y) \in \R^2.
\end{equation*}
\end{theorem}

\begin{proof}
Let $(x,y) \in \R^2$ be fixed.
Since for $f \in \Schwartz(\R^2)$ we have $\Lambda f \in \L^1(\R^2)$ and $\Fourier(\Lambda f) \in \L^1(\R^2)$, applying the two-dimensional Fourier inversion formula to $\Lambda f$ yields the identity
\begin{equation*}
\Lambda f(x,y) = \Fourier^{-1}(\Fourier \Lambda f)(x,y) = \frac{1}{4\pi^2} \int_{\R} \int_{\R} \|(X,Y)\|_{\R^2}^{-1} \, \Fourier f(X,Y) \, \e^{\i (xX + yY)} \: \d X \, \d Y.
\end{equation*}
By changing the variables $(X,Y) \in \R^2$ from Cartesian coordinates to $(S,\theta) \in \R \times [0,\pi)$ in polar coordinates, i.e.,
\begin{equation*}
X = S \cos(\theta)
\quad \mbox{ and } \quad
Y = S \sin(\theta),
\end{equation*}
we get $\d X \, \d Y = |S| \: \d S \, \d \theta$.
Thus, with the Fourier slice Theorem~\ref{theo:Fourier_slice} follows that
\begin{align*}
\Lambda f(x,y) & = \frac{1}{4\pi^2} \int_0^{\pi} \int_{\R} |S|^{-1} \, \Fourier f(S \cos(\theta),S \sin(\theta)) \, \e^{\i S (x \cos(\theta) + y \sin(\theta))} \, |S| \: \d S \, \d \theta \\
& \stackrel{\mathclap{\text{FST}}}{=} \frac{1}{4\pi^2} \int_0^{\pi} \int_{\R} \Fourier (\Radon f)(S,\theta) \, \e^{\i S (x \cos(\theta) + y \sin(\theta))} \: \d S \, \d \theta \\
& = \frac{1}{2\pi} \int_0^{\pi} \Fourier^{-1}[\Fourier (\Radon f)(S,\theta)](x \cos(\theta) + y \sin(\theta),\theta) \: \d \theta = \frac{1}{2} \, \Back\big(\Radon f\big)(x,y)
\end{align*}
due to the one-dimensional Fourier inversion formula and the definition of the back projection.
Consequently, for the target function $f$ follows that
\begin{equation*}
f = \Fourier^{-1} \big(\Fourier f\big) = \Fourier^{-1} \big(\|\cdot\|_{\R^2} \, \|\cdot\|_{\R^2}^{-1} \, \Fourier f\big) = \Fourier^{-1} \big(\|\cdot\|_{\R^2} \, \Fourier(\Lambda f)\big) = \Lambda^{-1} \big(\Lambda f\big) = \frac{1}{2} \, \Lambda^{-1}\big(\Back(\Radon f)\big),
\end{equation*}
as stated.
\end{proof}

The filtering the laminogram method is now based on the inversion formula
\begin{equation*}
f(x,y) = \frac{1}{2} \Fourier^{-1}\big(\|(X,Y)\|_{\R^2} \, \Fourier(\Back(\Radon f))(X,Y)\big)(x,y)
\quad \forall \, (x,y) \in \R^2
\end{equation*}
and consists of the following steps:
\begin{enumerate}
\item[(1)] Back project the sinogram $\Radon f$ to obtain the laminogram
\begin{equation*}
f_b(x,y) = \Back\big(\Radon f\big)(x,y)
\quad \mbox{ for } (x,y) \in \R^2.
\end{equation*}
\item[(2)] Compute the 2D Fourier transform of the laminogram to obtain
\begin{equation*}
\hat{f}_b(X,Y) = \Fourier f_b(X,Y)
\quad \mbox{ for } (X,Y) \in \R^2.
\end{equation*}
\item[(3)] Multiply $\hat{f}_b$ by the filter $\|(X,Y)\|_{\R^2}$ to obtain
\begin{equation*}
\hat{H}(X,Y) = \|(X,Y)\|_{\R^2} \, \hat{f}_b(X,Y)
\quad \mbox{ for } (X,Y) \in \R^2.
\end{equation*}
\item[(4)] Apply 2D inverse Fourier transform to reconstruct $f$ via
\begin{equation*}
f(x,y) = \frac{1}{2} \, \Fourier^{-1} \hat{H}(x,y)
\quad \mbox{ for } (x,y) \in \R^2.
\end{equation*}
\end{enumerate}

\chapter{Method of filtered back projection}

We have seen that the basic reconstruction problem in CT can be formulated as the problem of reconstructing a bivariate function $f \in \L^1(\R^2)$ from given Radon data
\begin{equation*}
\left\{ \Radon f (t,\theta) \mid t \in \R, \, \theta \in [0,\pi) \right\}.
\end{equation*}
Based on the FBP formula~\eqref{eq:FBP_formula} and under suitable assumptions on~$f$ we found the reconstruction formula
\begin{equation*}
f(x,y) = \frac{1}{2} \, \Back \big(\Fourier^{-1} [|S| \Fourier(\Radon f)(S,\theta)]\big)(x,y)
\quad \forall \, (x,y) \in \R^2,
\end{equation*}
which is highly sensitive with respect to noise and, thus, cannot be used in practice.

In this chapter we explain how the above FBP formula can be stabilized by incorporating a low-pass filter $F_L: \R \to \R$ of the form
\begin{equation*}
F_L(S) = |S| \, W(\nicefrac{S}{L})
\quad \mbox{ for } S \in \R
\end{equation*}
with bandwidth $L>0$ and an even window $W \in \L^\infty(\R)$ of compact support $\supp(W) \seq [-1,1]$.
This standard approach reduces the noise sensitivity of the reconstruction scheme, but leads to an {\em approximate} FBP reconstruction, which we denote by $f_L$ and can be rewritten as
\begin{equation}
f_L = \frac{1}{2} \, \Back \big(\Fourier^{-1} F_L * \Radon f\big).
\label{eq:FBP_approximate_form}
\end{equation}
An application of~\eqref{eq:FBP_approximate_form} will be called {\em method of filtered back projection} or, in short, {\em FBP method}.

\section{Approximate reconstruction formula}

Based on the FBP formula~\eqref{eq:FBP_formula} we now define the approximate FBP reconstruction $f_L$ by incorporating a suitable low-pass filter $F_L$.

\begin{definition}[Low-pass filter]
Let $L > 0$ and let $W \in \L^\infty(\R)$ be even and compactly supported with
\begin{equation*}
\supp(W) \subseteq [-1,1].
\end{equation*}
A function $F_L: \R \to \R$ of the form
\begin{equation*}
F_L(S) = |S| \, W(\nicefrac{S}{L})
\quad \mbox{ for } S \in \R
\end{equation*}
is called {\em low-pass filter} for the stabilization of the FBP formula~\eqref{eq:FBP_formula}, where $L$ denotes its {\em bandwidth} and $W$ is its {\em window function}.
For the sake of brevity, we set $F \equiv F_1$ so that
\begin{equation*}
F_L(S) = L \, F(\nicefrac{S}{L})
\quad \forall \, S \in \R.
\end{equation*}
\end{definition}

In the following, let $F_L$ be a low-pass filter with bandwidth $L$ and window $W$.
Because of the compact support of $W$ we have $F_L \in \L^p(\R)$ for all $1 \leq p \leq \infty$ and
\begin{equation*}
\supp(F_L) \seq [-L,L].
\end{equation*}
Now, let the target function $f$ satisfy $f \in \L^1(\R^2)$.
Based on the FBP formula~\eqref{eq:FBP_formula} we define the {\em approximate FBP reconstruction} $f_L$ via
\begin{equation}
f_L(x,y) = \frac{1}{2} \, \Back \big(\Fourier^{-1} [F_L(S) \Fourier(\Radon f)(S,\theta)]\big)(x,y)
\quad \mbox{ for } (x,y) \in \R^2.
\label{eq:f_L}
\end{equation}
We will see that~\eqref{eq:f_L} defines a {\em band-limited} approximation of the target function~$f$.

\begin{definition}[Band-limited function]
A function $f$ whose Fourier transform $\Fourier f$ has compact support is called a {\em band-limited} function.
\end{definition}

In the first theorem we show that $f_L$ is defined almost everywhere on $\R^2$ and can be simplified as
\begin{equation*}
f_L = \frac{1}{2} \, \Back \big(\kappa_L * \Radon f\big),
\end{equation*}
where we define the band-limited function $\kappa_L: \R \times [0,\pi) \to \R$ via
\begin{equation*}
\kappa_L(S,\theta) = \Fourier^{-1} F_L(S)
\quad \mbox{ for } (S,\theta) \in \R \times [0,\pi).
\end{equation*}
Note that $\kappa_L$ is well-defined on $\R \times [0,\pi)$ and satisfies $\kappa_L \in \L^\infty(\R \times [0,\pi))$, since $F_L \in \L^1(\R)$.
Furthermore, the definition of $\kappa_L$ is independent of the angle $\theta \in [0,\pi)$ and only depends on the radial variable $S \in \R$.
For the sake of brevity, we set $\kappa \equiv \kappa_1$ and the scaling property of the Fourier transform gives
\begin{equation}
\kappa_L(S,\theta) = L^2 \, \kappa(LS,\theta)
\quad \forall \, (S,\theta) \in \R \times [0,\pi).
\label{eq:scaling_kappa_l}
\end{equation}

\begin{theorem}[Simplification of $f_L$]
Let $f \in \L^1(\R^2)$ and let $F_L$ be be a low-pass filter such that $\kappa \in \L^1(\R \times [0,\pi))$.
Then, the approximate FBP reconstruction $f_L$, given by
\begin{equation*}
f_L(x,y) = \frac{1}{2} \, \Back \big(\Fourier^{-1} [F_L(S) \Fourier(\Radon f)(S,\theta)]\big)(x,y)
\quad \mbox{ for } (x,y) \in \R^2,
\end{equation*}
is defined almost everywhere on $\R^2$ and satisfies $f_L \in \L^\infty(\R^2)$.
Furthermore, $f_L$ can be rewritten as
\begin{equation}
f_L = \frac{1}{2} \, \Back \big(\kappa_L*\Radon f\big).
\label{eq:FBP_method}
\end{equation}
\end{theorem}

\begin{proof}
Since $F_L \in \L^1(\R)$ and $\kappa_L \in \L^1(\R \times [0,\pi))$ by assumption, the Fourier inversion theorem gives
\begin{equation*}
\Fourier \kappa_L(S,\theta) = F_L(S)
\quad \forall \, (S,\theta) \in \R \times [0,\pi)
\end{equation*}
and with the Fourier convolution theorem follows that
\begin{equation*}
F_L(S) \Fourier(\Radon f)(S,\theta) = \Fourier(\kappa_L * \Radon f)(S,\theta)
\quad \forall \, (S,\theta) \in \R \times [0,\pi),
\end{equation*}
where we use $\Radon f \in \L^1(\R \times [0,\pi))$ with
\begin{equation*}
\|\Radon f(\cdot,\theta)\|_{\L^1(\R)} \leq \|f\|_{\L^1(\R^2)}
\quad \forall \, \theta \in [0,\pi).
\end{equation*}
Applying again the Fourier inversion theorem yields the desired representation
\begin{equation*}
f_L = \frac{1}{2} \, \Back \big(\kappa_L*\Radon f\big) \in \L^\infty(\R^2),
\end{equation*}
since $(\kappa_L * \Radon f) \in \L^\infty(\R \times [0,\pi))$.
In particular, $f_L$ is defined almost everywhere on $\R^2$.
\end{proof}

We can further simplify the representation~\eqref{eq:FBP_method} of the approximate FBP reconstruction~$f_L$ by applying the back projection convolution formula~\eqref{eq:Radon_back_convolution}, i.e.,
\begin{equation*}
\Back(g * \Radon f) = \Back g * f.
\end{equation*}
To this end, we define the convolution kernel $K_L: \R^2 \to \R$ by
\begin{equation*}
K_L(x,y) = \frac{1}{2} \, \Back \kappa_L(x,y)
\quad \mbox{ for } (x,y) \in \R^2.
\end{equation*}
Since $\kappa_L \in \L^\infty(\R \times [0,\pi))$, the convolution kernel $K_L$ is defined almost everywhere on $\R^2$ and satisfies $K_L \in \L^\infty(\R^2)$.
For the sake of brevity, we set $K \equiv K_1$ and with~\eqref{eq:scaling_kappa_l} follows that
\begin{equation}
K_L(x,y) = L^2 \, K(Lx,Ly)
\quad \forall \, (x,y) \in \R^2.
\label{eq:scaling_K_L}
\end{equation}

\begin{corollary}
\label{cor:FBP_method_K_L}
Let $f \in \L^1(\R^2)$ and let $F_L$ be be a low-pass filter such that $\kappa \in \L^1(\R \times [0,\pi))$.
Then, the approximate FBP reconstruction $f_L$ can be rewritten as
\begin{equation}
f_L = f * K_L.
\label{eq:FBP_method_K_L}
\end{equation}
\end{corollary}

The representation~\eqref{eq:FBP_method_K_L} allows us to show that $f_L$ is a band-limited function.
To this end, we first determine the Fourier transform of the convolution kernel $K_L$, which in turn requires the definition of the bivariate window function $W_L: \R^2 \to \R$ as
\begin{equation*}
W_L(x,y) = W\bigg(\frac{r(x,y)}{L}\bigg)
\quad \mbox{ for } (x,y) \in \R^2,
\end{equation*}
where we let
\begin{equation*}
r(x,y) = \sqrt{x^2 + y^2}
\quad \mbox{ for } (x,y) \in \R^2.
\end{equation*}

\begin{theorem}[Convolution kernel $K_L$]
\label{theo:convolution_kernel}
Let $F_L$ be be a low-pass filter such that $K \in \L^1(\R^2)$.
Then, the convolution kernel satisfies $K_L \in \Cont_0(\R^2)$ for all $L > 0$ and its Fourier transform is given by
\begin{equation*}
\Fourier K_L(x,y) = W_L(x,y)
\quad \forall \, (x,y) \in \R^2.
\end{equation*}
\end{theorem}

\begin{proof}
Since $W \in \L^\infty(\R)$ has compact support, the bivariate window function $W_L$ is compactly supported and satisfies $W_L \in \L^1(\R^2)$.
Hence, with the Riemann-Lebesgue lemma follows that $\Fourier^{-1} W_L \in \Cont_0(\R^2)$.
Furthermore, for all $(x,y) \in \R^2$ we obtain
\begin{align*}
\Fourier^{-1} W_L(x,y) & = \frac{1}{4\pi^2} \int_\R \int_\R W_L(X,Y) \, \e^{\i(xX + yY)} \: \d X \, \d Y \\
& = \frac{1}{4\pi^2} \int_0^\pi \int_\R W(\nicefrac{S}{L}) \, |S| \, \e^{\i S(x \cos(\theta) + y \sin(\theta))} \: \d S \, \d \theta \\
& = \frac{1}{4\pi^2} \int_0^\pi \int_\R F_L(S) \, \e^{\i S(x \cos(\theta) + y \sin(\theta))} \: \d S \, \d \theta
\end{align*}
by transforming $(X,Y) = (S \cos(\theta),S \sin(\theta))$ from Cartesian coordinates to polar coordinates.
With the definition of the band-limited function $\kappa_L$ and the back projection operator $\Back$ follows that
\begin{equation*}
\Fourier^{-1} W_L(x,y) = \frac{1}{2\pi} \int_0^\pi \kappa_L(x \cos(\theta) + y \sin(\theta),\theta) \: \d \theta = \frac{1}{2} \, \Back \kappa_L(x,y) = K_L(x,y).
\end{equation*}
Consequently, we have $K_L \in \Cont_0(\R^2)$ and applying the Fourier inversion formula shows that
\begin{equation*}
\Fourier K_L = W_L,
\end{equation*}
since $K \in \L^1(\R^2)$ implies $K_L \in \L^1(\R^2)$ for all $L > 0$ due to the scaling property~\eqref{eq:scaling_K_L}.
\end{proof}

Before we proceed, we wish to add one remark concerning the convolution kernel $K_L$.

\begin{remark}
Since the bivariate window $W_L$ has compact support and satisfies $W_L \in \L^1(\R^2)$, its inverse Fourier transform is analytic due to the Paley-Wiener theorem.
Consequently, the convolution kernel $K_L$ not only satisfies $K_L \in \Cont_0(\R^2)$, but also lies in $\Cont^\infty(\R^2)$.
Furthermore, due to the Riemann-Lebesgue lemma the assumption $K_L \in \L^1(\R^2)$ implies that $W_L$ is continuous on $\R^2$ and, thus, the univariate window $W$ is continuous on $\R$.
\end{remark}

Combining Theorem~\ref{theo:convolution_kernel} and Corollary~\ref{cor:FBP_method_K_L} allows us to determine the Fourier transform $\Fourier f_L$ of the approximate FBP reconstruction $f_L$.

\begin{corollary}[Fourier transform of $f_L$]
\label{cor:Fourier_f_L}
Let $f \in \L^1(\R^2)$ and let $F_L$ be a low-pass filter such that $\kappa \in \L^1(\R \times [0,\pi))$ and $K \in \L^1(\R^2)$.
Then, the Fourier transform $\Fourier f_L$ of the approximate FBP reconstruction $f_L$ is given by
\begin{equation*}
\Fourier f_L = W_L \cdot \Fourier f.
\end{equation*}
\end{corollary}

Since the window $W$ is assumed to be compactly supported, Corollary~\ref{cor:Fourier_f_L} shows that the approximate FBP reconstruction formula~\eqref{eq:f_L} provides a band-limited approximation $f_L$ to the target function $f$.
In particular, the approximation $f_L$ is arbitrarily smooth, $f_L \in \Cont_0^\infty(\R^2)$.
Moreover, the assumptions $f \in \L^1(\R^2)$ and $K \in \L^1(\R^2)$ ensure that we also have $f_L \in \L^1(\R^2)$.

\begin{remark}
All statements of this section also hold in $\L^2$-sense without assuming additional properties of $F_L$.
To be more precise, let $f \in \L^1(\R^2) \cap \L^2(\R^2)$ and let $F_L$ be be a low-pass filter.
Then, the approximate FBP reconstruction $f_L$ satisfies $f_L \in \L^2(\R^2)$ and can be rewritten as
\begin{equation*}
f_L = \frac{1}{2} \, \Back \big(\kappa_L*\Radon f\big) = f * K_L
\quad \mbox{ a.e.\ on } \R^2.
\end{equation*}
Moreover, its Fourier transform $\Fourier f_L \in \L^2(\R^2)$ is given by
\begin{equation*}
\Fourier f_L = W_L \cdot \Fourier f
\quad \mbox{ a.e.\ on } \R^2.
\end{equation*}
If the window $W$ is continuous at $0$ and satisfies $W(0) = 1$, one can prove the {\em $\L^2$-convergence}
\begin{equation*}
\|f - f_L\|_{\L^2(\R^2)} \xrightarrow{L \to \infty} 0.
\end{equation*}
On top of that, one can show that the approximate FBP reconstruction operator
\begin{equation*}
R_L g = \frac{1}{2} \, \Back \big(\kappa_L * g\big)
\end{equation*}
defines a {\em continuous} linear operator $R_L: \L^1(\R \times [0,\pi)) \cap \L^2(\R \times [0,\pi)) \to \L^2(\R^2)$.
\end{remark}

\section{Low-pass filters}
\label{sec:low-pass_filter}

In the method of filtered back projection we replace the exact filer $|S|$ in the FBP formula~\eqref{eq:FBP_formula} by a low-pass filter $F_L$.
In the general context of Fourier analysis, a low-pass filter is a function $F \equiv F(S)$ of the frequency variable $S$ which maps the high-frequency parts of a signal to zero.
To this end, one usually requires a compact support $\supp(F) \subseteq [-L,L]$ for a bandwidth $L > 0$ so that
\begin{equation*}
F(S) = 0
\quad \forall \, |S| > L.
\end{equation*}
In the particular context of the FBP method, we require a sufficient approximation quality for the low-pass filter $F_L$ within the frequency band $[-L,L]$ in the sense that
\begin{equation*}
F_L(S) \approx |S|
\enspace \mbox{ on } [-L,L]
\quad \mbox{ and } \quad
F_L(S) \xrightarrow{L \to \infty} |S|
\enspace \forall \, S \in \R.
\end{equation*}
Therefore, we use the ansatz
\begin{equation*}
F_L(S) = |S| \, W(\nicefrac{S}{L})
\quad \mbox{ for } S \in \R
\end{equation*}
with an even window $W \in \L^\infty(\R)$ with $\supp(W) \subseteq [-1,1]$ and $W(0) = 1$.

\bigbreak

We now list some classical low-pass filters, which are widely used in the FBP method of CT.
To this end, recall that $\rect{L}: \R \to \R$ denotes the characteristic function of the interval $[-L,L]$, i.e.,
\begin{equation*}
\rect{L}(S) = \begin{cases}
1 & \text{for } |S| \leq L \\
0 & \text{for } |S| > L.
\end{cases}
\end{equation*}
For the sake of brevity, we set $\rect{} \equiv \rect{1}$.

\begin{example}
The {\em Ram-Lak filter} is given by the window function
\begin{equation*}
W(S) = \rect{}(S)
\quad \mbox{ for } S \in \R
\end{equation*}
such that
\begin{equation*}
F_L(S) = |S| \cdot \rect{L}(S)
= \begin{cases}
|S| & \text{for } |S| \leq L \\
0 & \text{for } |S| > L.
\end{cases}
\end{equation*}
The Ram-Lak window and filter are shown in Figure~\ref{fig:ram-lak_window} and Figure~\ref{fig:ram-lak_filter}, respectively.
\end{example}

\begin{figure}[t]
\centering
\subfigure[Ram-Lak window]{\label{fig:ram-lak_window} \includegraphics[width=0.3\textwidth]{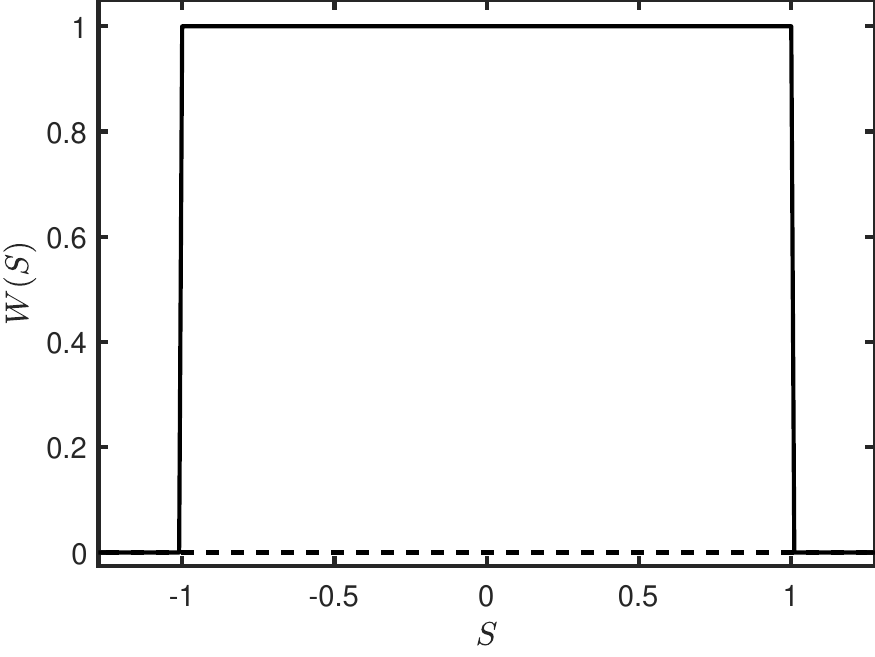}}
\hfil
\subfigure[Shepp-Logan window]{\label{fig:shepp-logan_window} \includegraphics[width=0.3\textwidth]{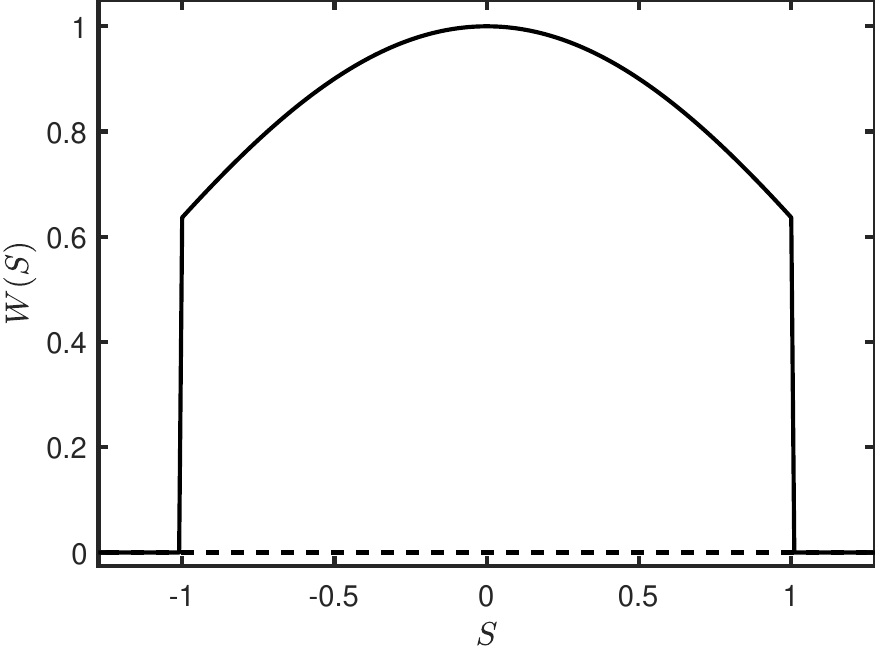}}
\hfil
\subfigure[Cosine window]{\label{fig:cosine_window} \includegraphics[width=0.3\textwidth]{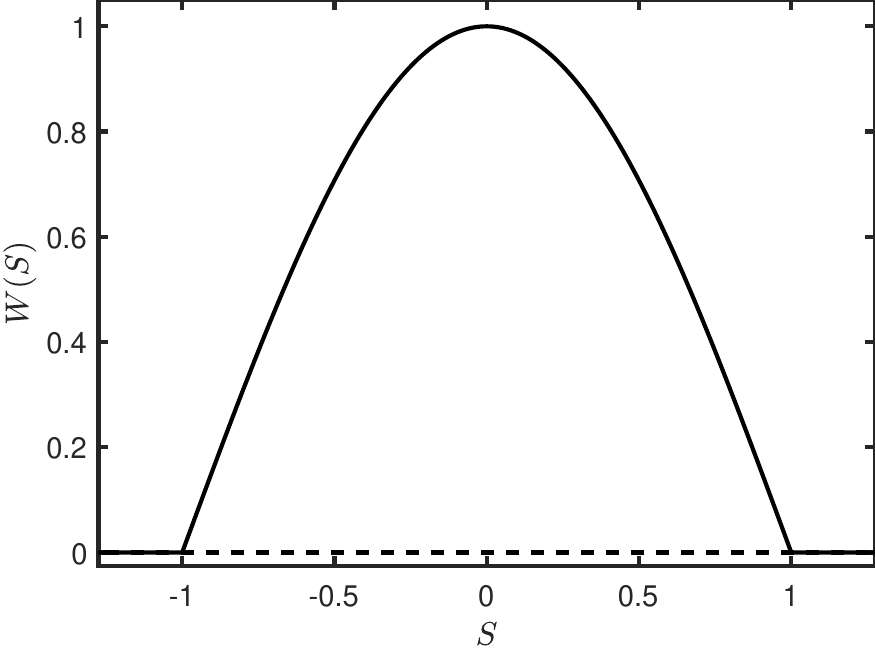}}
\hfil
\subfigure[Ram-Lak filter]{\label{fig:ram-lak_filter} \includegraphics[width=0.3\textwidth]{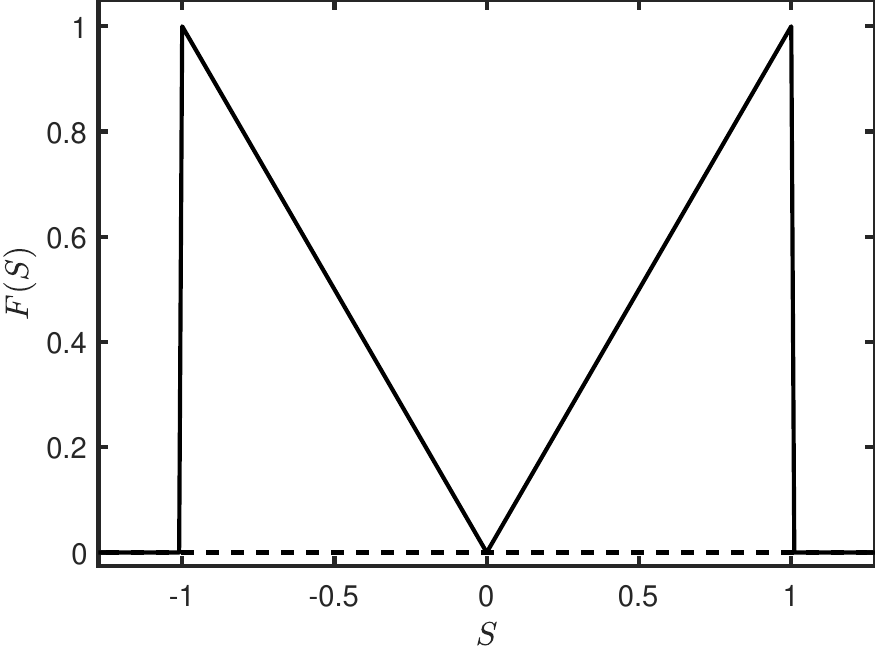}}
\hfil
\subfigure[Shepp-Logan filter]{\label{fig:shepp-logan_filter} \includegraphics[width=0.3\textwidth]{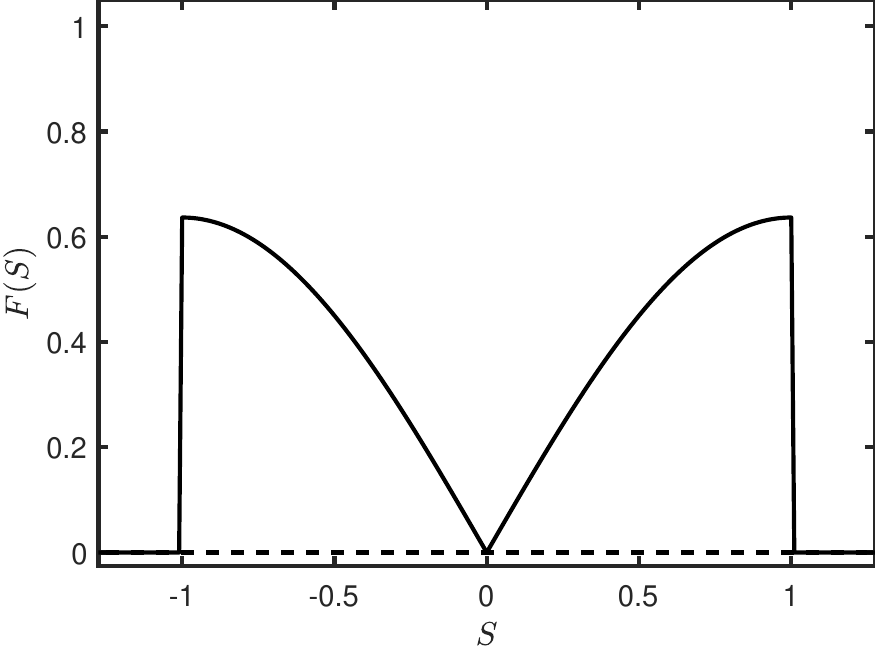}}
\hfil
\subfigure[Cosine filter]{\label{fig:cosine_filter} \includegraphics[width=0.3\textwidth]{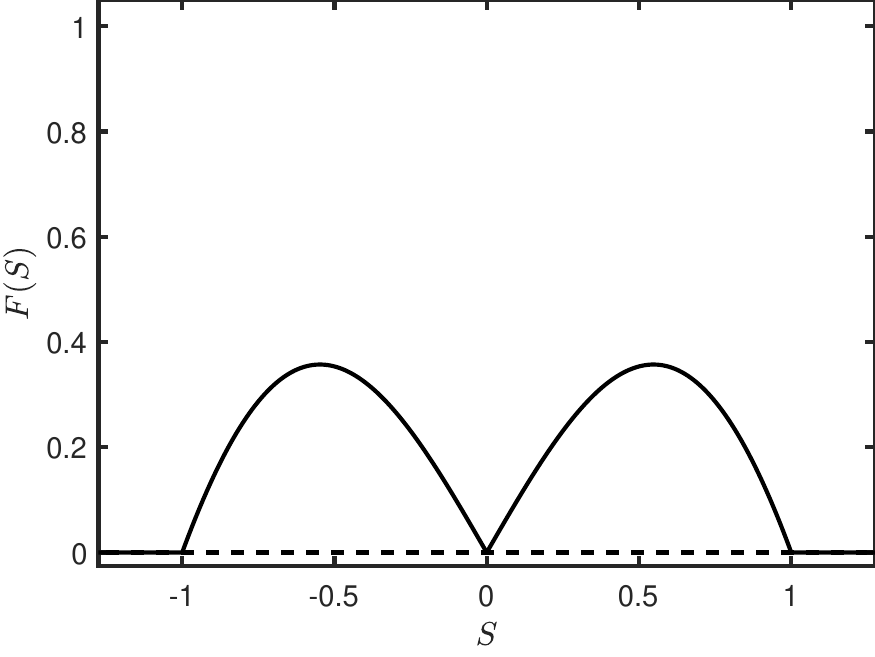}}
\caption{Three typical low-pass filters.}
\end{figure}

In the following, let $\sinc$ denote the {\em unnormalized} cardinal sine function, i.e.,
\begin{equation*}
\sinc(t) = \frac{\sin(t)}{t}
\quad \mbox{ for } t \in \R.
\end{equation*}

\begin{example}
The {\em Shepp-Logan filter} is given by the window function
\begin{equation*}
W(S) = \sinc\Big(\frac{\pi S}{2}\Big) \cdot \rect{}(S)
\quad \mbox{ for } S \in \R
\end{equation*}
such that
\begin{equation*}
F_L(S) = |S| \cdot \sinc\Big(\frac{\pi S}{2L}\Big) \cdot \rect{L}(S)
= \begin{cases}
\frac{2L}{\pi} \big|\sin\big(\frac{\pi S}{2L}\big)\big| & \text{for } |S| \leq L \\
0 & \text{for } |S| > L.
\end{cases}
\end{equation*}
The Shepp-Logan window and filter are shown in Figure~\ref{fig:shepp-logan_window} and Figure~\ref{fig:shepp-logan_filter}, respectively.
\end{example}

Note that the Ram-Lak and Shepp-Logan window have jump discontinuities at $S \in \{-1,1\}$.
In contrast to this, the window of the next low-pass filter is continuous on the whole real line.

\begin{example}
The {\em Cosine filter} is given by the window function
\begin{equation*}
W(S) = \cos\Big(\frac{\pi S}{2}\Big) \cdot \rect{}(S)
\quad \mbox{ for } S \in \R
\end{equation*}
such that
\begin{equation*}
F_L(S) = |S| \cdot \cos\Big(\frac{\pi S}{2L}\Big) \cdot \rect{L}(S)
= \begin{cases}
|S| \cdot \cos\big(\frac{\pi S}{2L}\big) & \text{for } |S| \leq L \\
0 & \text{for } |S| > L.
\end{cases}
\end{equation*}
The Cosine window and filter are shown in Figure~\ref{fig:cosine_window} and Figure~\ref{fig:cosine_filter}, respectively.
\end{example}

Combining the Ram-Lak and a modified Cosine filter yields the so called Hamming filters.

\begin{example}
The {\em Hamming filter} with parameter $\beta \in \big[\frac{1}{2},1\big]$ is given by the window function
\begin{equation*}
W(S) = (\beta + (1-\beta) \cos(\pi S)) \cdot \rect{}(S)
\quad \mbox{ for } S \in \R
\end{equation*}
such that
\begin{equation*}
F_L(S) = \begin{cases}
|S| \cdot \Big(\beta + (1-\beta) \cos\big(\frac{\pi S}{L}\big)\Big) & \text{for } |S| \leq L \\
0 & \text{for } |S| > L.
\end{cases}
\end{equation*}
The Hamming window and filter are shown Figure~\ref{fig:Hamming_filter} for parameter $\beta \in \{0.5,0.75,0.95\}$.
\end{example}

\begin{figure}[t]
\centering
\subfigure[Hamming window ($\beta = 0.5$)]{\includegraphics[width=0.3\textwidth]{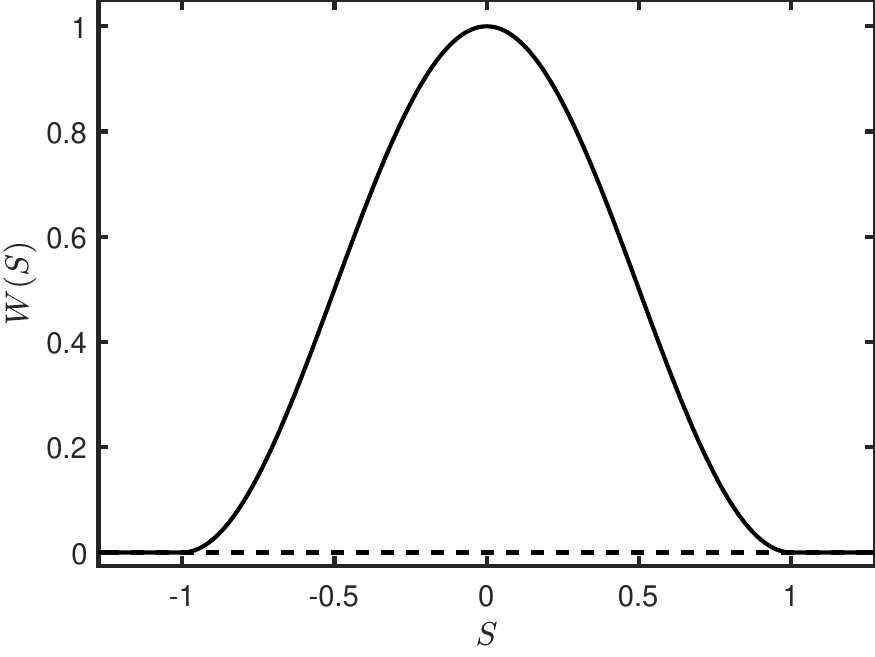}}
\hfil
\subfigure[Hamming window ($\beta = 0.75$)]{\includegraphics[width=0.3\textwidth]{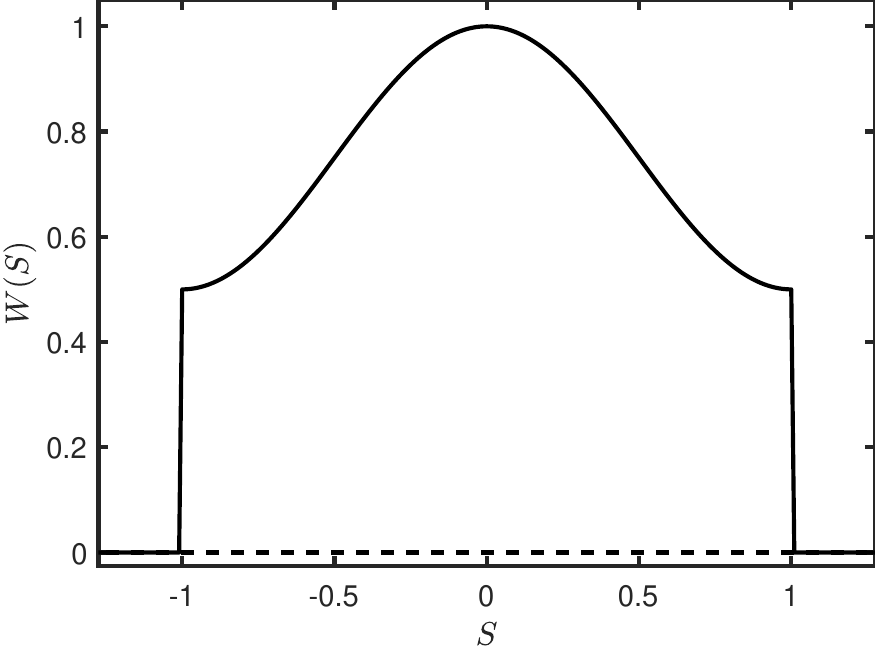}}
\hfil
\subfigure[Hamming window ($\beta = 0.95$)]{\includegraphics[width=0.3\textwidth]{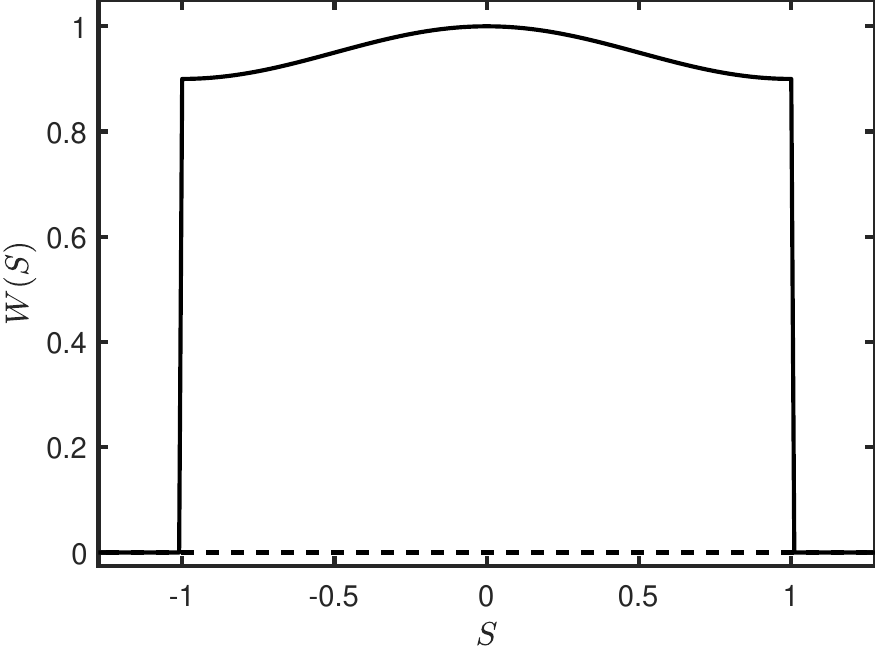}}
\hfil
\subfigure[Hamming filter ($\beta = 0.5$)]{\includegraphics[width=0.3\textwidth]{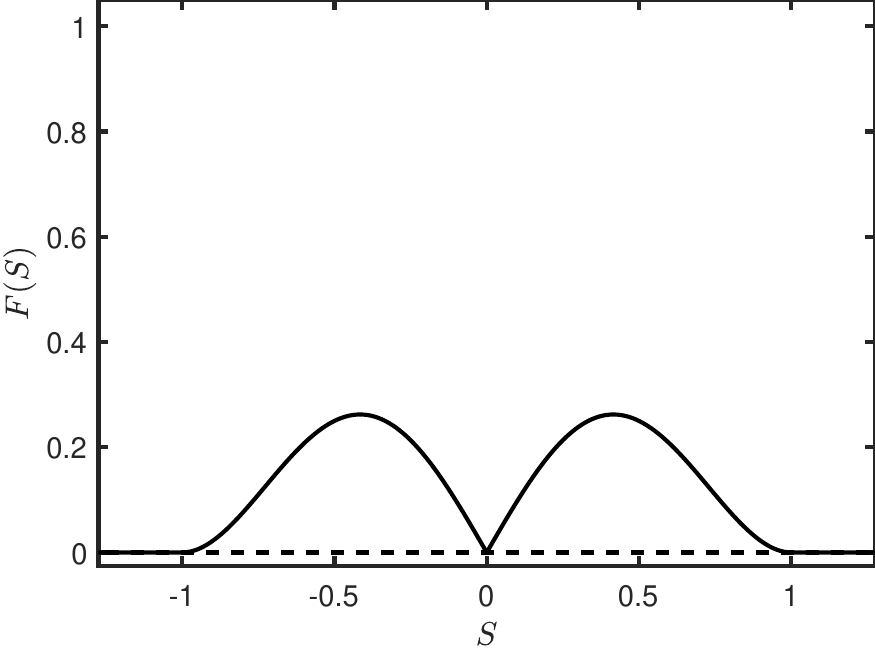}}
\hfil
\subfigure[Hamming filter ($\beta = 0.75$)]{\includegraphics[width=0.3\textwidth]{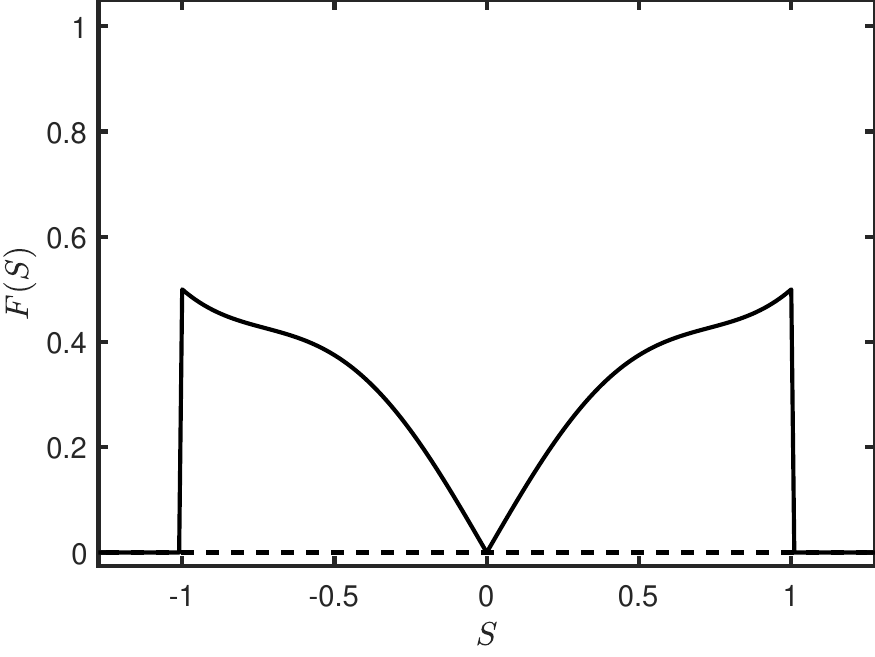}}
\hfil
\subfigure[Hamming filter ($\beta = 0.95$)]{\includegraphics[width=0.3\textwidth]{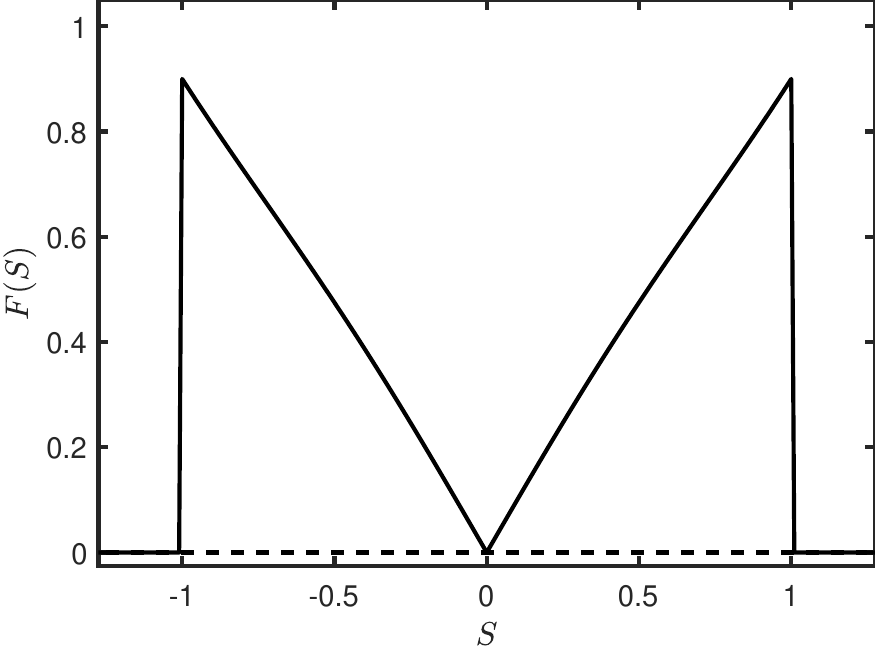}}
\caption{The Hamming filter for $\beta \in \{0.5,0.75,0.95\}$.}
\label{fig:Hamming_filter}
\end{figure}

Another class of low-pass filters depending on a parameter is given by the Gaussian filters.

\begin{example}
The {\em Gaussian filter} with parameter $\beta > 1$ is given by the window function
\begin{equation*}
W(S) = \exp\bigg(-\Big(\frac{\pi S}{\beta}\Big)^2\bigg) \cdot \rect{}(S)
\quad \mbox{ for } S \in \R
\end{equation*}
such that
\begin{equation*}
F_L(S) = \begin{cases}
|S| \cdot \exp\Big(-\big(\frac{\pi S}{\beta L}\big)^2\Big) & \text{for } |S| \leq L \\
0 & \text{for } |S| > L.
\end{cases}
\end{equation*}
The Gaussian window and filter are shown in Figure~\ref{fig:Gaussian_filter} for parameter $\beta \in \{2.5,5,7.5\}$.
\end{example}

\begin{figure}[t]
\centering
\subfigure[Gaussian window ($\beta = 2.5$)]{\includegraphics[width=0.3\textwidth]{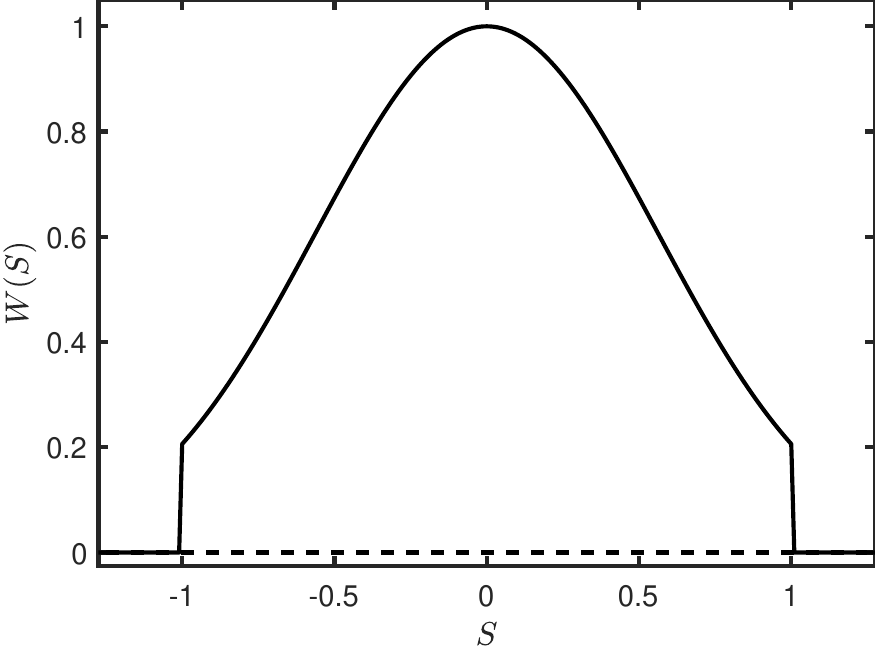}}
\hfil
\subfigure[Gaussian window ($\beta = 5$)]{\includegraphics[width=0.3\textwidth]{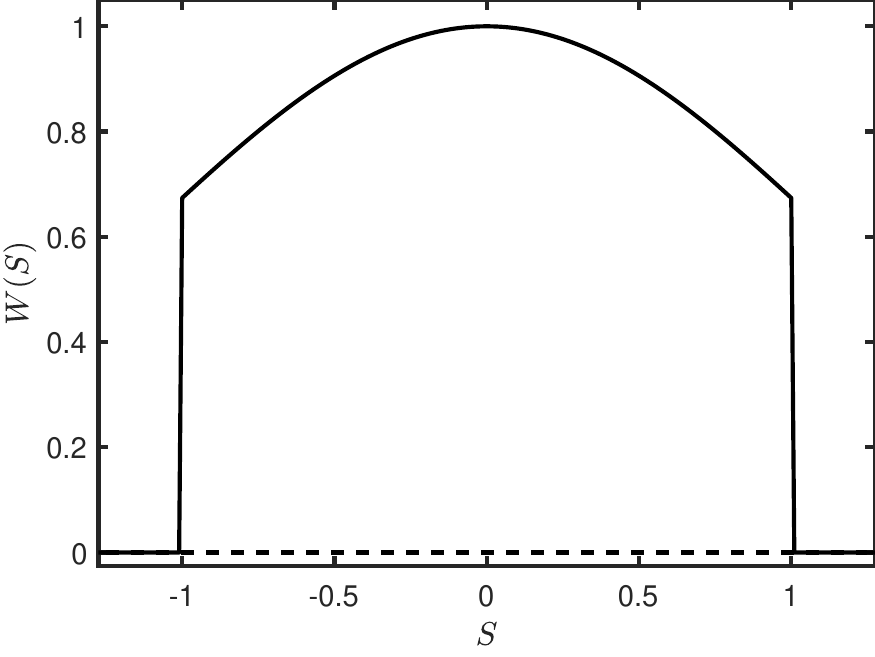}}
\hfil
\subfigure[Gaussian window ($\beta = 7.5$)]{\includegraphics[width=0.3\textwidth]{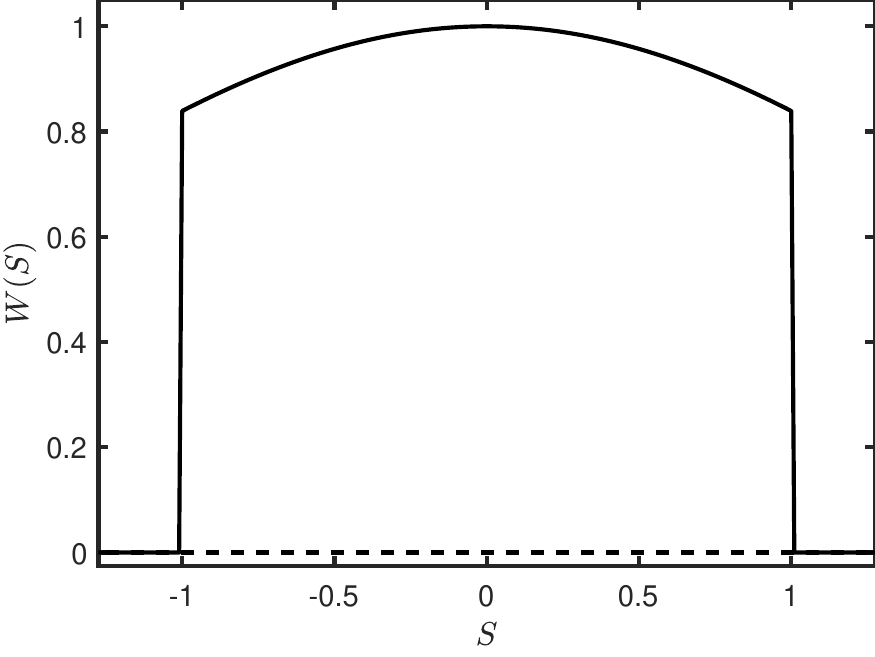}}
\hfil
\subfigure[Gaussian filter ($\beta = 2.5$)]{\includegraphics[width=0.3\textwidth]{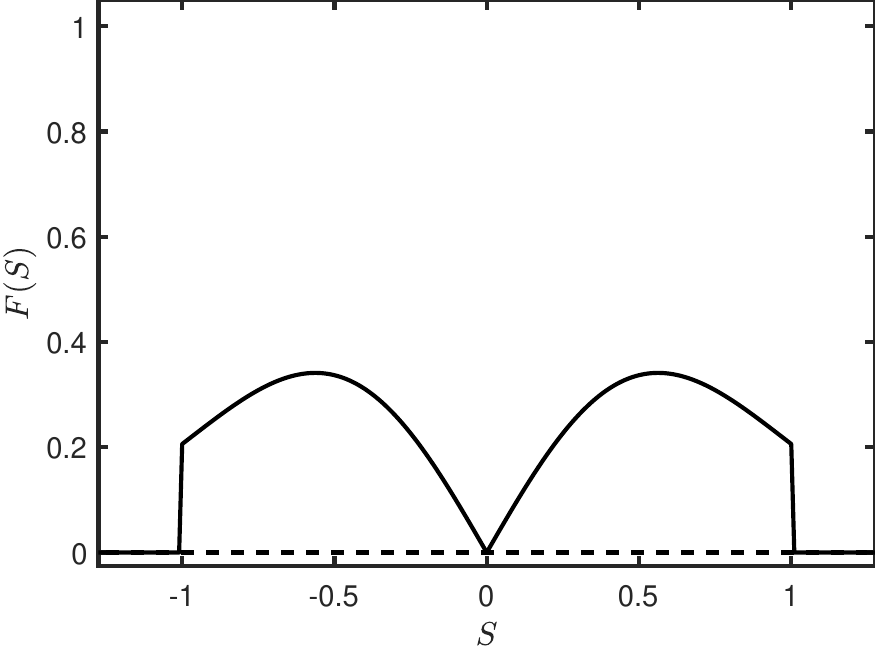}}
\hfil
\subfigure[Gaussian filter ($\beta = 5$)]{\includegraphics[width=0.3\textwidth]{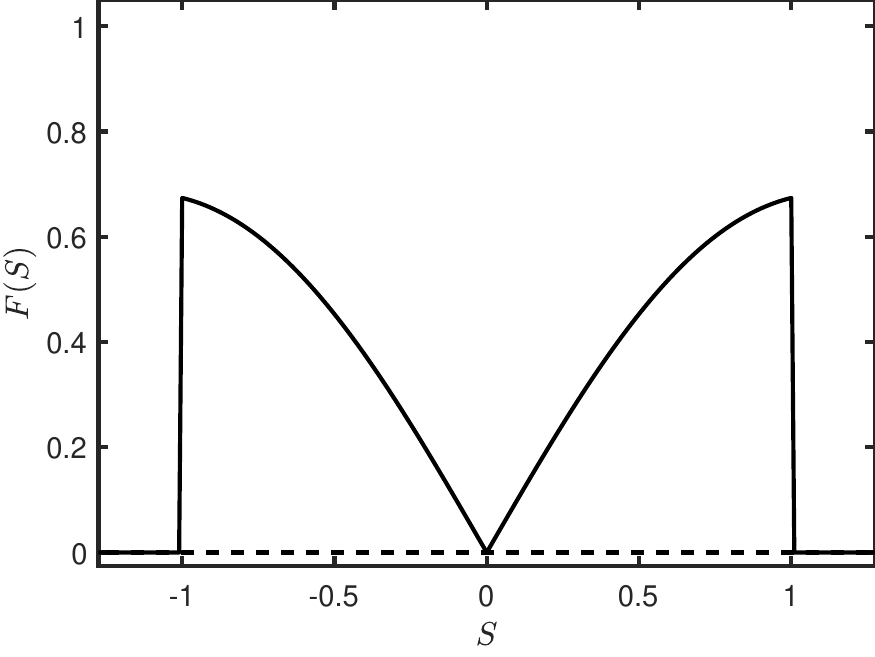}}
\hfil
\subfigure[Gaussian filter ($\beta = 7.5$)]{\includegraphics[width=0.3\textwidth]{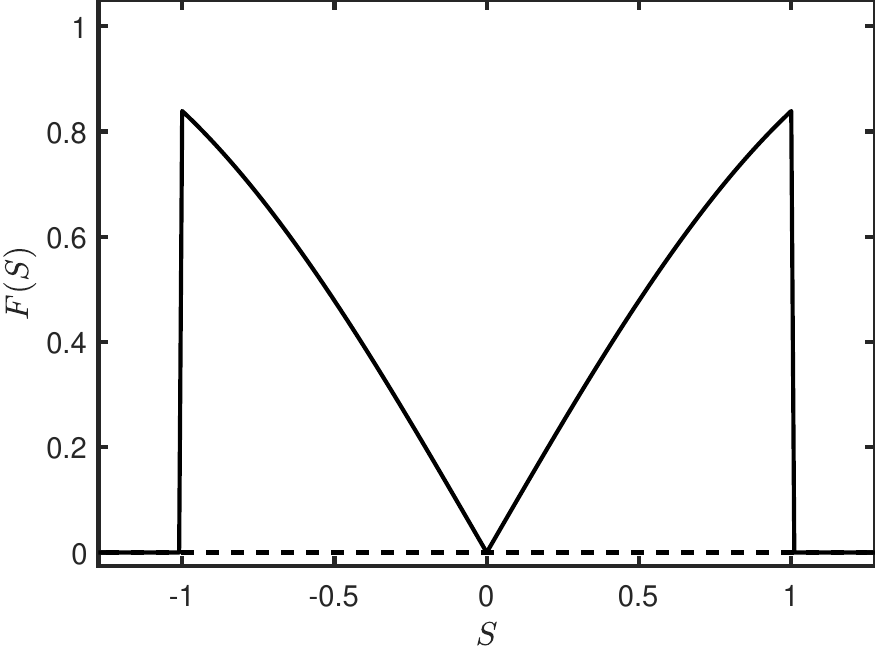}}
\caption{The Gaussian filter for $\beta \in \{2.5,5,7.5\}$.}
\label{fig:Gaussian_filter}
\end{figure}

\section{Reconstruction in parallel beam geometry}

The approximate FBP reconstruction formula~\eqref{eq:FBP_approximate_form} assumes the Radon data $\Radon f(t,\theta)$ to be available for {\em all} $(t,\theta) \in \R \times [0,\pi)$.
In practice, however, only finitely many Radon samples are given and we have to recover the target function $f$ from a finite set of Radon data
\begin{equation*}
\left\{ \Radon f(t_j,\theta_j) \mid j = 1,\ldots,J \right\}
\quad \mbox{ for some } J \in \N.
\end{equation*}
Thus, the implementation of the FBP method requires a suitable discretization of formula~\eqref{eq:FBP_approximate_form}.
To be more precise, we have to discretize the convolution product $*$ and the back projection operator $\Back$.
This also includes the specification of a sampling scheme for the Radon transform $\Radon f$ and the inverse Fourier transform $\Fourier^{-1} F_L$ of the chosen low-pass filter.

A commonly used sampling scheme is given by the {\em parallel beam geometry}, where the Radon lines $\ell_{t,\theta}$ are equally spaced in both the radial variable $t \in \R$ and the angular variable $\theta \in [0,\pi)$.
More precisely, for $N$ uniformly distributed angles we collect Radon samples along $2M+1$ parallel lines per angle with a fixed spacing $d > 0$.
Hence, the Radon data are of the form
\begin{equation}
(\Radon f)_{j,k} = \Radon f(t_j,\theta_k)
\label{eq:parallel_beam_data}
\end{equation}
with
\begin{equation*}
t_j = j \cdot d
\enspace \mbox{ for } j = -M,\ldots,M
\quad \mbox{ and } \quad
\theta_k = k \cdot \frac{\pi}{N}
\enspace \mbox{ for } k = 0,\ldots,N-1
\end{equation*}
so that in total $N \cdot (2M+1)$ Radon samples are taken.
For illustration, Figure~\ref{fig:Parallel_beam_geometry} shows the arrangement of $108$ Radon lines in $[-1,1]^2$ with $N = 12$, $M = 4$ and sampling spacing $d = 0.25$.

\begin{figure}[t]
\centering
\includegraphics[height=0.3\textwidth]{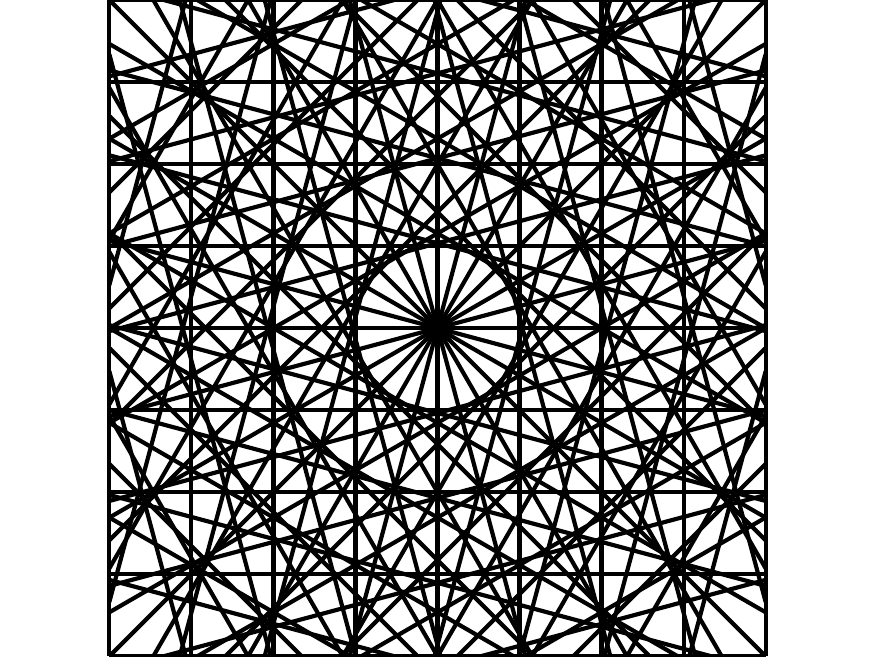}
\caption{Parallel beam geometry with $N = 12$, $M = 4$ and $d = 0.25$.}
\label{fig:Parallel_beam_geometry}
\end{figure}

\subsubsection*{Shannon sampling}

Before we come to the discretization of the FBP method~\eqref{eq:FBP_approximate_form}, we first discuss the sampling process in more detail, which will help to select a sample spacing $d$.
In this paragraph we will see that we can uniquely recover a band-limited function $h \in \L^2(\R)$ from uniformly spaced discrete function values $h(t_j)$, for $j \in \Z$, if the sampling distance $d = t_{j+1}-t_j$ is chosen reasonably.
This is the statement of the classical {\em Shannon sampling theorem} and the sample spacing $d$ corresponds to the smallest detail in $h$ that is still recognizable after sampling the function.

\begin{theorem}[Shannon sampling theorem]
Let $h \in \L^2(\R)$ be a band-limited function with $\supp(\Fourier h) \seq [-L,L]$ for some bandwidth $L > 0$.
Then, $h$ is uniquely determined by the discrete values $h\big(\frac{\pi k}{L}\big)$, for $k \in \Z$, and we have
\begin{equation}
h(t) = \sum_{k\in\Z} h\Big(\frac{\pi k}{L}\Big) \, \sinc(Lt-k\pi)
\quad \forall \, t \in \R.
\label{eq:Shannon_interpolation}
\end{equation}
\end{theorem}

\begin{proof}
Let $h \in \L^2(\R)$ be band-limited with $\supp(\Fourier h) \seq [-L,L]$ for some bandwidth $L>0$.
By the Rayleigh-Plancherel theorem we then have $\Fourier h \in \L^2([-L,L]) \subset \L^1([-L,L])$, which implies that $h$ has a continuous representative due to the Riemann-Lebesgue lemma.
Thus, the Fourier inversion formula holds pointwise and in $\L^2$-sense yielding
\begin{equation*}
h(t) = \Fourier^{-1} (\Fourier h)(t) = \frac{1}{2\pi} \int_{-L}^L \Fourier h(\omega) \, \e^{\i \omega t} \: \d \omega = \frac{L}{2\pi^2} \int_{-\pi}^\pi \Fourier h\Big(\frac{L \omega}{\pi}\Big) \, \e^{\i \omega \nicefrac{L t}{\pi}} \: \d \omega
\quad \forall \, t \in \R.
\end{equation*}
For fixed $t \in \R$ we now consider the Fourier expansion of $\e^{\i \omega \nicefrac{L t}{\pi}}$ as a function in $\L^2([-\pi,\pi])$, which is given by
\begin{equation*}
\e^{\i \omega \nicefrac{L t}{\pi}} = \sum_{k \in \Z} c_k \, \e^{\i \omega k}
\quad \forall \, \omega \in [-\pi,\pi]
\end{equation*}
with the Fourier coefficients
\begin{equation*}
c_k = \frac{1}{2\pi} \int_{-\pi}^\pi \e^{\i \omega \nicefrac{L t}{\pi}} \, \e^{-\i \omega k} \: \d \omega = \frac{1}{2} \int_{-1}^1 \e^{\i \, (Lt - k\pi) \, \omega} \: \d \omega = \sinc(Lt - k\pi)
\quad \forall \, k \in \Z.
\end{equation*}
Recall that the partial sums of the above Fourier series converge in the $\L^2$-norm, i.e.,
\begin{equation*}
\int_{-\pi}^\pi \bigg|\e^{\i \omega \nicefrac{L t}{\pi}} - \sum_{k=-n}^n c_k \, \e^{\i \omega k}\bigg|^2 \: \d \omega \xrightarrow{n \to \infty} 0.
\end{equation*}
Since $\Fourier h \in \L^2([-L,L])$, this in combination with the Cauchy-Schwarz inequality implies that
\begin{equation*}
\left|\int_{-\pi}^\pi \Fourier h\Big(\frac{L \omega}{\pi}\Big) \, \e^{\i \omega \nicefrac{L t}{\pi}} \: \d \omega - \sum_{k=-n}^n c_k \int_{-\pi}^\pi \Fourier h\Big(\frac{L \omega}{\pi}\Big) \, \e^{\i \omega k} \: \d \omega\right| \xrightarrow{n \to \infty} 0.
\end{equation*}
Consequently, we can interchange the order of summation and integration so that, for all $t \in \R$,
\begin{align*}
h(t) & = \frac{L}{2\pi^2} \int_{-\pi}^\pi \Fourier h\Big(\frac{L \omega}{\pi}\Big) \sum_{k \in \Z} \sinc(Lt - k\pi) \, \e^{\i \omega k} \: \d \omega = \sum_{k \in \Z} \sinc(Lt - k\pi) \frac{L}{2\pi^2} \int_{-\pi}^\pi \Fourier h\Big(\frac{L \omega}{\pi}\Big) \, \e^{\i \omega k} \: \d \omega \\
& = \sum_{k \in \Z} \sinc(Lt - k\pi) \, \frac{1}{2\pi} \int_{-L}^L \Fourier h(\omega) \, \e^{\i \omega \nicefrac{\pi k}{L}} \: \d \omega = \sum_{k\in\Z} h\Big(\frac{\pi k}{L}\Big) \, \sinc(Lt - k\pi),
\end{align*}
where we again use the Fourier inversion formula.
\end{proof}

We remark that the formula~\eqref{eq:Shannon_interpolation} is also called {\em Shannon-Whittaker interpolation formula} and more generally we have
\begin{equation*}
h(t) = \sum_{k\in\Z} h(k \cdot d) \, \sinc\Big(\frac{\pi}{d} \, (t - k \cdot d)\Big)
\quad \forall \, t \in \R
\end{equation*}
if the spacing $d > 0$ of the discrete samples $h(k \cdot d)$, with $k \in \Z$, satisfies the {\em Nyquist condition}
\begin{equation*}
d \leq \frac{\pi}{L}.
\end{equation*}
The largest possible sample spacing $d = \frac{\pi}{L}$ is known as the {\em Nyquist rate}.

\begin{remark}
If $h \in \Schwartz(\R)$ is only {\em essentially $L$-band-limited} in the sense that, for $0 < \varepsilon \ll 1$,
\begin{equation*}
\int_{\R \setminus [-L,L]} |\Fourier h(\omega)| \: \d \omega \leq \varepsilon,
\end{equation*}
the reconstruction of $h$ from discrete samples $h(k \cdot d)$, for $k \in \Z$ and sampling distance $d > 0$, by using
\begin{equation*}
S_d h(t) = \sum_{k\in\Z} h(k \cdot d) \, \sinc\Big(\frac{\pi}{d} \, (t - k \cdot d)\Big)
\quad \mbox{ for } t \in \R
\end{equation*}
is no longer exact.
If $d \leq \frac{\pi}{L}$, however, \cite[Theorem III.1.3]{Natterer2001} shows that the reconstruction error can be bounded by
\begin{equation*}
\|S_d h - h\|_{\L^\infty(\R)} \leq \frac{\varepsilon}{\pi}.
\end{equation*}
\end{remark}

\subsection*{Discrete FBP reconstruction formula for parallel beam geometry}

We now address the discretization of the FBP method~\eqref{eq:FBP_approximate_form} for the approximate reconstruction of a target function $f$ from discrete Radon data $\{(\Radon f)_{j,k}\}$ given in parallel beam geometry~\eqref{eq:parallel_beam_data}.
To this end, from now on we assume that $f$ is compactly supported with
\begin{equation*}
\supp(f) \seq B_r(0)
\quad \mbox{ for some } r \in \N.
\end{equation*}

We start with discretizing the convolution product $*$ in~\eqref{eq:FBP_approximate_form} between the Radon data $\Radon f$ and the inverse Fourier transform $\Fourier^{-1} F_L$ of the low-pass filter $F_L$.
Here, for fixed angle $\theta \in [0,\pi)$, we have to approximate the convolution integral
\begin{equation*}
(\Fourier^{-1} F_L * \Radon f)(S,\theta) = \int_\R \Fourier^{-1} F_L(S-t) \, \Radon(t,\theta) \: \d t
\quad \mbox{ for } S \in \R
\end{equation*}
by only using the discrete data
\begin{equation*}
\Radon f(t_j,\theta) = \Radon f(j \cdot d,\theta)
\quad \mbox{ for } j \in \Z
\end{equation*}
taken at equally spaced sampling points $t_j = j \cdot d$, for $j \in \Z$, with fixed sampling distance $d > 0$.
To achieve this, we apply the composite trapezoidal rule and replace the above convolution integral by the (infinite) sum
\begin{equation*}
(\Fourier^{-1} F_L * \Radon f)(S,\theta) \approx d \sum_{j \in \Z} \Fourier^{-1} F_L(S-t_j) \, \Radon f(t_j,\theta)
\quad \mbox{ for } (S,\theta) \in \R \times [0,\pi).
\end{equation*}
Since $f$ is assumed to have compact support, the above sum is in fact finite and, consequently, we obtain
\begin{equation*}
(\Fourier^{-1} F_L * \Radon f)(S,\theta) \approx d \sum_{j = -M}^M \Fourier^{-1} F_L(S-t_j) \, \Radon f(t_j,\theta)
\quad \mbox{ for } (S,\theta) \in \R \times [0,\pi),
\end{equation*}
where $M \in \N$ is chosen sufficiently large such that, for any angle $\theta \in [0,\pi)$,
\begin{equation*}
\Radon f(t,\theta) = 0
\quad \forall \, |t| > M \cdot d.
\end{equation*}

\bigbreak

Let us continue with the discretization of the back projection operator $\Back$, which is for a function $h \equiv h(S,\theta)$ in polar coordinates given by
\begin{equation*}
\Back h(x,y) = \frac{1}{\pi} \int_0^\pi h(x \cos(\theta) + y \sin(\theta),\theta) \: \d \theta
\quad \mbox{ for } (x,y) \in \R^2.
\end{equation*}
In~\eqref{eq:FBP_approximate_form}, this has to be applied to the function
\begin{equation*}
h(S,\theta) = (\Fourier^{-1} F_L * \Radon f)(S,\theta)
\quad \mbox{ for } (S,\theta) \in \R \times [0,\pi),
\end{equation*}
where the Radon data $\Radon f(t,\theta)$ is only known for a finite set of $N$ angles
\begin{equation*}
\theta_k = k \cdot \frac{\pi}{N}
\quad \mbox{ for } k = 0,\ldots,N-1.
\end{equation*}
Thus, for the discretization of $\Back$ we again use the composite trapezoidal rule and replace the above integral by the sum
\begin{equation*}
\Back h(x,y) \approx \frac{1}{N} \sum_{k=0}^{N-1} h(x \cos(\theta_k) + y \sin(\theta_k),\theta_k)
\quad \mbox{ for } (x,y) \in \R^2.
\end{equation*}

Combining the discretization steps leads us to a discrete version of the FBP method~\eqref{eq:FBP_approximate_form} given by
\begin{equation*}
f_{\text{D}} (x,y) = \frac{d}{2N} \sum_{k=0}^{N-1} \sum_{j = -M}^M \Fourier^{-1} F_L(x\cos(\theta_k)+y\sin(\theta_k) - t_j) \, \Radon f(t_j,\theta_k)
\quad \mbox{ for } (x,y) \in \R^2,
\end{equation*}
which we write in compact form as
\begin{equation*}
f_{\text{D}}  = \frac{1}{2} \, \Back_{\text{D}} \big(\Fourier^{-1} F_L *_{\text{D}}  \Radon f\big).
\end{equation*}

The evaluation of the discrete reconstruction $f_{\text{D}} $ requires the computation of the values
\begin{equation*}
(\Fourier^{-1} F_L *_{\text{D}}  \Radon f)(x \cos(\theta_k) + y \sin(\theta_k),\theta_k)
\quad \forall \, 0 \leq k \leq N-1
\end{equation*}
for each reconstruction point $(x,y) \in \R^2$.
To reduce the computational costs, we evaluate, for each $0 \leq k \leq N-1$, the function
\begin{equation*}
h(t,\theta_k) = (\Fourier^{-1} F_L *_{\text{D}}  \Radon f)(t,\theta_k) = d \sum_{j = -M}^M \Fourier^{-1} F_L(t-t_j) \, \Radon f(t_j,\theta_k)
\quad \mbox{ for } t \in \R
\end{equation*}
only at the sampling points $t_i = i \cdot d$ for $i \in I$ with a sufficiently large index set $I \subset \Z$.
For each reconstruction point $(x,y) \in \R^2$ we then interpolate the value $h(t,\theta_k)$ at $t = x \cos(\theta_k) + y \sin(\theta_k)$ by using a suitable interpolation method $\Int$.
This leads us to the {\em discrete FBP reconstruction formula}
\begin{equation}
f_\FBP = \frac{1}{2} \, \Back_{\text{D}}  \Big(\Int \big[\Fourier^{-1} F_L *_{\text{D}}  \Radon f\big]\Big).
\label{eq:FBP_reconstruction_discrete_parallel}
\end{equation}

There are many possible choices for the interpolation method $\Int$.
In the following, we give two examples that are commonly used.
For the sake of brevity, we set $h_k = h(\cdot,\theta_k)$ for $0 \leq k \leq N-1$.
\begin{itemize}
\item {\bf Nearest neighbour interpolation:}
Let $t \in [t_m,t_{m+1})$ for some $m \in \Z$.
Then, the function value $h_k(t)$ is approximated by
\begin{equation*}
\Int_0 h_k(t) = \begin{cases}
h_k(t_m) & \text{for } t-t_m \leq t_{m+1}-t \\
h_k(t_{m+1}) & \text{for } t-t_m > t_{m+1}-t.
\end{cases}
\end{equation*}
This defines a piecewise constant interpolant $\Int_0 h_k$ of $h_k$, which is discontinuous in general.
\item {\bf Linear spline interpolation:}
Let $t \in [t_m,t_{m+1})$ for some $m \in \Z$.
Then, the function value $h_k(t)$ is approximated by
\begin{equation*}
\Int_1 h_k(t) = \frac{1}{d} \left[(t-t_m) \, h_k(t_{m+1}) + (t_{m+1}-t) \, h_k(t_m)\right].
\end{equation*}
This defines a piecewise linear interpolant $\Int_1 h_k$ of $h_k$, which is globally continuous.
\end{itemize}

\bigbreak

According to~\cite[Section 5.1.1]{Natterer2001a} the optimal sampling conditions for the reconstruction of an essentially $L$-band-limited target function $f$ supported in $B_r(0)$ are given by
\begin{equation*}
d \leq \frac{\pi}{L}, \quad M \geq \frac{r}{d}, \quad N \geq rL
\end{equation*}
leading to the well-known optimal sampling relation
\begin{equation*}
N = \pi \cdot M.
\end{equation*}
Here, the restriction $d \leq \frac{\pi}{L}$ ensures that the convolution $*$ in~\eqref{eq:FBP_approximate_form} is properly discretized, while $N \geq rL$ guarantees a satisfactory discretization of the back projection $\Back$ via the trapezoidal rule.
Since for fixed angle $\theta \in [0,\pi)$ the function
\begin{equation*}
h(S) = (\Fourier^{-1} F_L * \Radon f)(S,\theta)
\quad \mbox{ for } S \in \R
\end{equation*}
is band-limited with bandwidth $L$, the condition on $d$ corresponds to the Nyquist rate for $h$ according to the Shannon sampling theorem.
Finally, the relation $M \geq \frac{r}{d}$ ensures that the whole support of the target function $f$ is covered during the acquisition of the Radon data.

Since we assume that $f$ is supported in $B_r(0)$ for some $r \in \N$ and $N,M$ have to be integers, we couple the discretization parameters $d > 0$ and $M,N \in \N$ with the bandwidth $L$ via
\begin{equation*}
d = \frac{\pi}{L}, \quad M = r \cdot \frac{L}{\pi}, \quad N = 3 \cdot M
\end{equation*}
and choose $L$ to be a multiple of $\pi$, i.e., $L = \pi \cdot J$ for some $J \in \N$.

The computation of the discretized approximate FBP reconstruction $f_\FBP$ in~\eqref{eq:FBP_reconstruction_discrete_parallel} requires the evaluation of the inverse Fourier transform $\Fourier^{-1} F_L$ of the utilized low-pass filter~$F_L$ at the sampling points
\begin{equation*}
t_j = j \cdot \frac{\pi}{L}
\quad \mbox{ for } j \in \Z.
\end{equation*}
Therefore, in the following we give analytical expressions for the samples $\Fourier^{-1} F_L\big(\frac{j \pi}{L}\big)$, for $j \in \Z$, for the Ram-Lak, Shepp-Logan and Cosine filter:
\begin{itemize}
\item[(i)] For the Ram-Lak filter we have
\begin{equation*}
\Fourier^{-1} F_L\Big(\frac{j \pi}{L}\Big) = \begin{cases}
\frac{L^2}{2\pi} & \text{for } j = 0 \\
0 & \text{for } j \neq 0 \text{ even} \\
-\frac{2 L^2}{\pi^3 j^2} & \text{for } j \neq 0 \text{ odd}.
\end{cases}
\end{equation*}
\item[(ii)] For the Shepp-Logan filter we have
\begin{equation*}
\Fourier^{-1} A_L\Big(\frac{j \pi}{L}\Big) = \frac{4 L^2}{\pi^3 (1 - 4 j^2)}.
\end{equation*}
\item[(iii)] For the Cosine filter we have
\begin{equation*}
\Fourier^{-1} A_L\Big(\frac{j \pi}{L}\Big) = \frac{2 L^2}{\pi^2} \left(\frac{(-1)^j}{1 - 4 j^2} - \frac{2 (1 + 4 j^2)}{\pi (1 - 4 j^2)^2}\right).
\end{equation*}
\end{itemize}

We summarize the discrete FBP method in the following image reconstruction algorithm, where we assume that the reconstruction $f_\FBP$ in~\eqref{eq:FBP_reconstruction_discrete_parallel} is evaluated in Cartesian grid points
\begin{equation*}
\left\{ (x_m,y_n) \in \R^2 \mid (m,n) \in I_x \times I_y \right\}
\end{equation*}
with finite index sets $I_x \times I_y \subset \N \times \N$.

\bigbreak

\begin{algorithm}[t]
\caption{Discrete FBP method in parallel beam geometry}
\label{algo:FBP_method}
\begin{algorithmic}[1]
\Require Radon data $(\Radon f)_{j,k} = \Radon f(t_j,\theta_k)$ for $j=-M,\ldots,M$, $k=0,\ldots,N-1$
\Statex
\State \textbf{choose} low-pass filter $F_L$ with bandwidth $L>0$
\Statex
\For{$k=0,\ldots,N-1$}\Comment{Computation of the discrete convolution}
\For{$i \in I$}
\State ${\displaystyle h(t_i,\theta_k) = \frac{\pi}{L} \sum_{j=-M}^M \Fourier^{-1} F_L(t_i - t_j) \, \Radon f(t_j,\theta_k)}$
\EndFor
\EndFor
\Statex
\State \textbf{choose} interpolation method $\Int$
\Statex
\For{$m \in I_x$}\Comment{Computation of the discrete back projection}
\For{$n \in I_y$}
\State ${\displaystyle f_\FBP(x_m,y_n) = \frac{1}{2N} \sum_{k=0}^{N-1} \Int h(x_m \cos(\theta_k) + y_n \sin(\theta_k),\theta_k)}$
\EndFor
\EndFor
\Statex
\Ensure Approximate reconstruction $f_\FBP$ on Cartesian grid $\{(x_m,y_n) \mid (m,n) \in I_x \times I_y\}$
\end{algorithmic}
\end{algorithm}

For illustration, we use Algorithm~\ref{algo:FBP_method} to reconstruct the Shepp-Logan phantom (see Figure~\ref{fig:shepp-logan_phantom_sinogram}) and the thorax phantom (see Figure~\ref{fig:thorax_phantom_sinogram}) from finite Radon data.
Both phantoms are supported in $B_1(0)$, i.e., we set $r = 1$.
The FBP reconstructions of the phantoms are displayed in Figure~\ref{fig:phantom_reconstruction}, where we used linear interpolation and the Ram-Lak filter with window function
\begin{equation*}
W(S) = \rect{}(S)
\quad \mbox{ for } S \in \R
\end{equation*}
and bandwidth $L = 50\pi$ so that
\begin{equation*}
F_L(S) = \begin{cases} |S| & \text{for } |S| \leq 50\pi \\
0 & \text{for } |S| > 50\pi.
\end{cases}
\end{equation*}
This corresponds to $M = 50$ and $N = 150$ so that in total $(2M+1)N = 15150$ Radon samples were taken.
Both reconstructions were evaluated on a square grid with $256 \times 256$ pixels.

\begin{figure}[b]
\centering
\subfigure[Shepp-Logan phantom]{\includegraphics[viewport=69 17 383 332, height=0.25\textwidth]{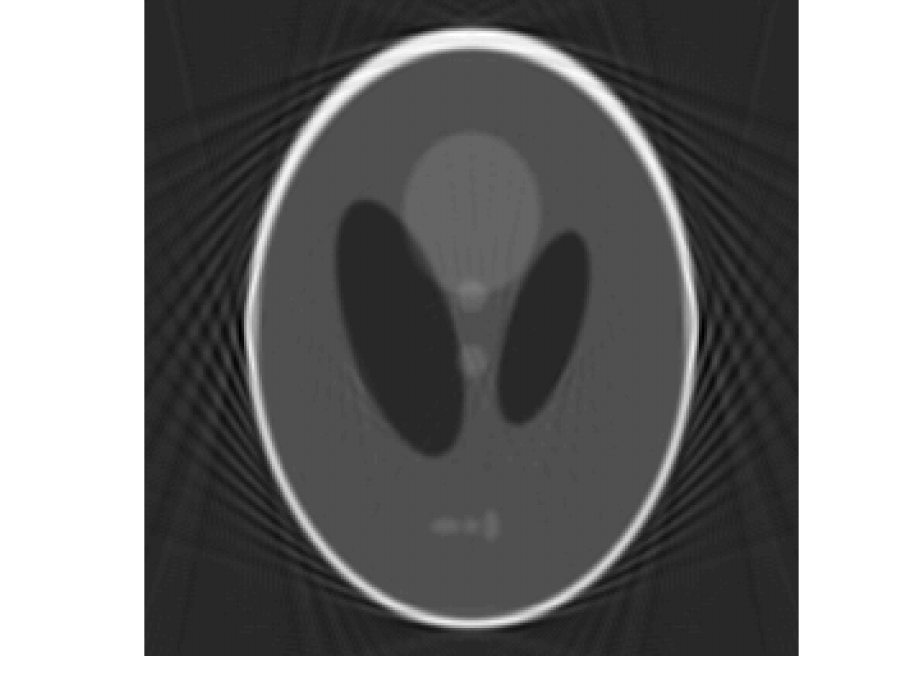}}
\hfil
\subfigure[Thorax phantom]{\includegraphics[viewport=69 17 383 332, height=0.25\textwidth]{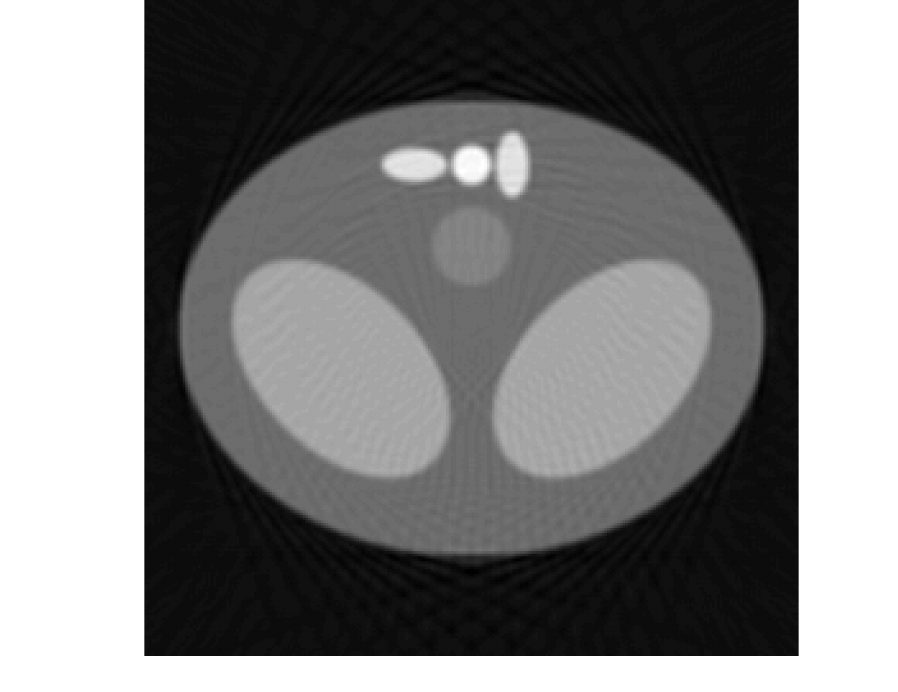}}
\caption{FBP reconstructions of two phantoms with the Ram-Lak filter and $L = 50\pi$.}
\label{fig:phantom_reconstruction}
\end{figure}

\section{Reconstruction in fan beam geometry}
\label{sec:standard_fan_beam}

In this section we describe another commonly used sampling scheme for the Radon transform, called {\em fan beam geometry}.
In this scanning protocol the X-ray source rotates around the object under investigation and emits a fan of X-ray beams that are detected by a large detector arc situated opposite of the source, see Figure~\ref{fig:fan_beam}.
In the following, we again assume that the target function $f$ is compactly supported with
\begin{equation*}
\supp(f) \seq B_r(0)
\quad \mbox{ for some } r > 0
\end{equation*}
and that the X-ray source travels on a full circle of radius $D > r$ so that its position $\bfx_{\text{S}} \in \R^2$ can be expressed as
\begin{equation*}
\bfx_{\text{S}} = D \begin{pmatrix} \cos(\beta) \\ \sin(\beta) \end{pmatrix} = D\bfn_\beta
\quad \mbox{ for } \beta \in [0,2\pi).
\end{equation*}
This gives rise to define the {\em fan beam transform} $\Fan f$ of $f$ as follows, see also Figure~\ref{fig:fan_beam_parametrization}.

\begin{definition}[Fan beam transform]
Let $f \in \L^1(\R^2)$ be a function in Cartesian coordinates.
Then, the {\em fan beam transform} $\Fan f$ of $f$ is defined as
\begin{equation}
\Fan f(\alpha,\beta) = \int_{L_{\alpha,\beta}} f(x,y) \: \d (x,y)
\quad \mbox{ for } (\alpha,\beta) \in \Bigl(-\frac{\pi}{2},\frac{\pi}{2}\Bigr) \times [0,2\pi),
\label{eq:Fan_beam_transform}
\end{equation}
where the {\em fan beam} $L_{\alpha,\beta} \subset \R^2$ denotes the unique straight line passing through $D\bfn_\beta$ that makes the angle $\alpha$ with the line joining $D\bfn_\beta$ and the origin.
\end{definition}

\begin{figure}[b]
\centering
\includegraphics[height=6.5cm,keepaspectratio]{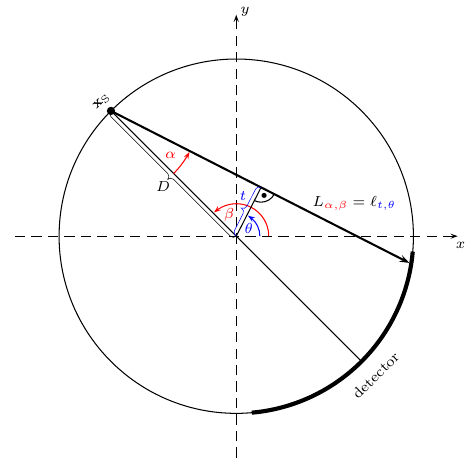}
\caption{Illustration of the fan beam parametrization.}
\label{fig:fan_beam_parametrization}
\end{figure}

An inspection of Figure~\ref{fig:fan_beam_parametrization} reveals that we can express the fan beam transform $\Fan f$ in~\eqref{eq:Fan_beam_transform} in terms of the Radon transform $\Radon f$ in~\eqref{eq:Radon_transform}.

\begin{observation}
For any $(\alpha,\beta) \in (-\frac{\pi}{2},\frac{\pi}{2}) \times [0,2\pi)$ we have $L_{\alpha,\beta} = \ell_{t,\theta}$ with
\begin{equation*}
t = D\sin(\alpha)
\quad \mbox{ and } \quad
\theta = \alpha + \beta - \frac{\pi}{2}
\end{equation*}
so that
\begin{equation}
\Fan f(\alpha,\beta) = \Radon f(D\sin(\alpha),\alpha+\beta-\nicefrac{\pi}{2}).
\label{eq:fan_parallel}
\end{equation}
\end{observation}

\bigbreak

In standard fan beam geometry, the fan beams $L_{\alpha,\beta}$ are uniformly distributed in both angular variables $\alpha \in (-\frac{\pi}{2},\frac{\pi}{2})$ and $\beta \in [0,2\pi)$.
More precisely, for $p$ equally spaced angular positions of the X-ray source we collect samples along $2q+1$ fan beams with fixed angular spacing $\Delta\alpha > 0$.
Hence, the fan beam data are of the form
\begin{equation}
(\Fan f)_{j,k} = \Fan f(\alpha_j,\beta_k)
\label{eq:fan_beam_data}
\end{equation}
with
\begin{equation*}
\alpha_j = j \cdot \Delta\alpha
\enspace \mbox{ for } j = -q,\ldots,q
\quad \mbox{ and } \quad
\beta_k = k \cdot \Delta\beta
\enspace \mbox{ for } k = 0,\ldots,p-1
\end{equation*}
so that in total $p \cdot (2q+1)$ fan beam samples are taken.
Let $\varphi \in (0,\pi)$ denote the opening angle of the X-ray fan.
Then, the angular spacings $\Delta\alpha$ and $\Delta\beta$ are chosen as
\begin{equation*}
\Delta\alpha = \frac{\varphi}{2q}
\quad \mbox{ and } \quad
\Delta\beta = \frac{2\pi}{p}.
\end{equation*}
The source distance $D > r$ has to be chosen such that the whole reconstruction region $B_r(0)$ is covered by the fan beams, i.e.,
\begin{equation*}
r \leq D\sin(\nicefrac{\varphi}{2}).
\end{equation*}
For illustration, Figure~\ref{fig:Fan_beam_geometry} shows the fan beam arrangement of $102$ lines with $p = 6$ and $q = 8$.
In Figure~\ref{fig:Fan_beam_geometry_3} we chose $D = 3$ and $\varphi = \frac{2\pi}{9}$, which corresponds to the angular spacing $\Delta\alpha = \frac{\pi}{54}$, whereas in Figure~\ref{fig:Fan_beam_geometry_2} we chose $D = 2$ and $\varphi = \frac{\pi}{3}$ so that in this case we have $\Delta\alpha = \frac{\pi}{48}$.

\begin{figure}[t]
\centering
\subfigure[$D = 3$, $\varphi = \frac{2\pi}{9}$]{\label{fig:Fan_beam_geometry_3}\includegraphics[height=0.3\textwidth]{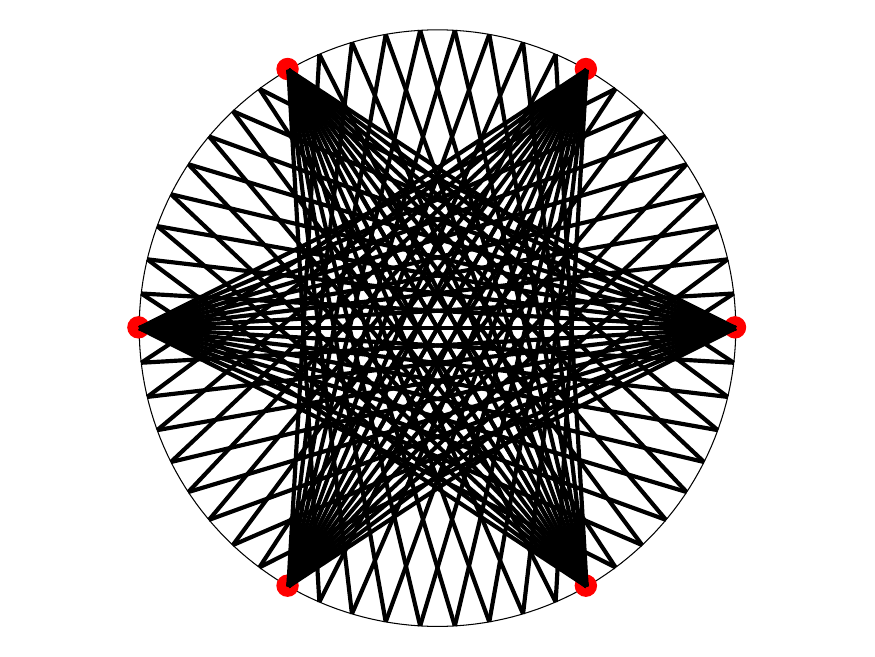}}
\hfil
\subfigure[$D = 2$, $\varphi = \frac{\pi}{3}$]{\label{fig:Fan_beam_geometry_2}\includegraphics[height=0.3\textwidth]{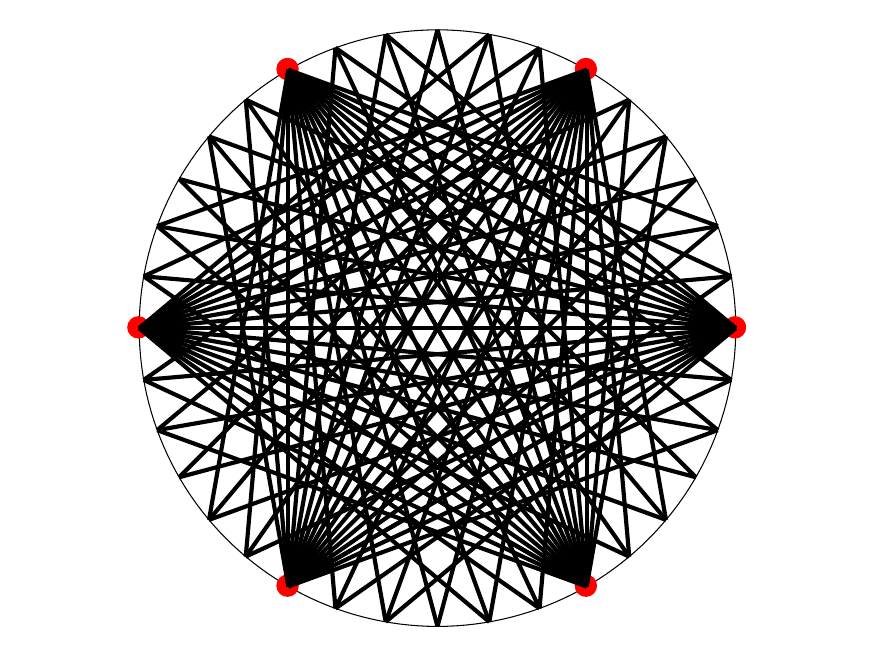}}
\caption{Fan beam geometry with $p = 6$ and $q = 8$.}
\label{fig:Fan_beam_geometry}
\end{figure}

\subsection*{Discrete FBP reconstruction formula for fan beam geometry}

Based on~\eqref{eq:fan_parallel}, one way to reconstruct $f$ from fan beam data~\eqref{eq:fan_beam_data} is to transform the data into parallel beam data by interpolation and apply the discrete FBP reconstruction formula~\eqref{eq:FBP_reconstruction_discrete_parallel}.
This procedure is called {\em rebinning}, but introduces artefacts in the reconstruction.
Hence, it is preferable to develop a specially adapted reconstruction algorithm for the fan beam geometry.

Starting point for the derivation of the reconstruction formula for fan beam data is the approximate FBP formula~\eqref{eq:FBP_approximate_form}, i.e.,
\begin{equation*}
f_L(x,y) = \frac{1}{2\pi} \int_0^\pi \int_\R (\Fourier^{-1} F_L)(x\cos(\theta)+y\sin(\theta)-t) \, \Radon f(t,\theta) \: \d t \, \d \theta,
\end{equation*}
where $F_L$ is a low-pass filter of bandwidth $L > 0$.
Since $\supp(f) \subseteq B_r(0)$ with $r < D$ and the Radon transform satisfies the evenness condition
\begin{equation*}
\Radon f(t,\theta+\pi) = \Radon f(-t,\theta)
\quad \forall \, (t,\theta) \in \R \times [0,\pi),
\end{equation*}
this can be rewritten as
\begin{equation*}
f_L(x,y) = \frac{1}{4\pi} \int_0^{2\pi} \int_{-D}^D (\Fourier^{-1} F_L)(x\cos(\theta)+y\sin(\theta)-t) \, \Radon f(t,\theta) \: \d t \, \d \theta.
\end{equation*}

\bigbreak

According to~\eqref{eq:fan_parallel}, we now apply the transformation
\begin{equation*}
t = D\sin(\alpha)
\quad \mbox{ and } \quad
\theta = \alpha + \beta - \frac{\pi}{2}
\end{equation*}
yielding the approximate FBP formula for the fan beam transform
\begin{equation}
f_L(x,y) = \frac{D}{4\pi} \int_0^{2\pi} \int_{-\nicefrac{\pi}{2}}^{\nicefrac{\pi}{2}} (\Fourier^{-1} F_L)(S_{x,y}(\alpha,\beta)) \, \cos(\alpha) \, \Fan f(\alpha,\beta) \: \d \alpha \, \d \beta,
\label{eq:FBP_approximate_form_fan}
\end{equation}
where
\begin{equation*}
S_{x,y}(\alpha,\beta) = x\cos(\alpha+\beta-\nicefrac{\pi}{2}) + y\sin(\alpha+\beta-\nicefrac{\pi}{2}) - D\sin(\alpha).
\end{equation*}

As the function $\Fourier^{-1} F_L$ is radially symmetric, we now investigate the term $|S_{x,y}(\alpha,\beta)|$ for a fixed reconstruction point $(x,y) \in \R^2$ and fixed angles $(\alpha,\beta) \in (-\frac{\pi}{2},\frac{\pi}{2}) \times [0,2\pi)$.
To this end, let $\bfx = (x,y)$ and again $\bfx_{\text{S}} = D \bfn_{\beta} = (D\cos(\beta),D\sin(\beta))$.
Moreover, let $\gamma$ be the angle between $\bfx - \bfx_{\text{S}}$ and $-\bfx_{\text{S}}$ and let $\bfy$ be the orthogonal projection of $\bfx$ onto the line $L_{\alpha,\beta}$, see Figure~\ref{fig:fan_beam_simplification}.
Then, we have
\begin{equation*}
|S_{x,y}(\alpha,\beta)| = |\bfx \cdot \bfn_{\alpha+\beta-\nicefrac{\pi}{2}} - D\sin(\alpha)| = \|\bfx - \bfy\|_{\R^2} = \|\bfx - \bfx_{\text{S}}\|_{\R^2} \, \sin(|\gamma - \alpha|),
\end{equation*}
where
\begin{equation*}
\|\bfx - \bfx_{\text{S}}\|_{\R^2}^2 = (x-D\cos(\beta))^2 + (y-D\sin(\beta))^2
\end{equation*}
and
\begin{equation*}
\cos(\gamma) = -\frac{(\bfx - \bfx_{\text{S}}) \cdot \bfx_{\text{S}}}{\|\bfx - \bfx_{\text{S}}\|_{\R^2} \, \|\bfx_{\text{S}}\|_{\R^2}} = \frac{D - x\cos(\beta) - y\sin(\beta)}{\sqrt{(x-D\cos(\beta))^2 + (y-D\sin(\beta))^2}}.
\end{equation*}

\begin{figure}[b]
\centering
\includegraphics[height=5.75cm,keepaspectratio]{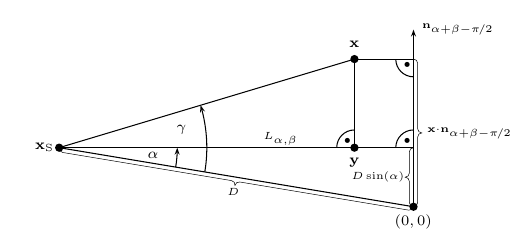}
\caption{Simplification of the fan beam parametrization.}
\label{fig:fan_beam_simplification}
\end{figure}

Moreover, the function $\Fourier^{-1} F_L$ has the homogeneity property
\begin{equation*}
(\Fourier^{-1} F_L)(\sigma S) = \sigma^{-2} (\Fourier^{-1} F_{\sigma L})(S)
\quad \forall \, \sigma > 0
\end{equation*}
so that
\begin{equation*}
(\Fourier^{-1} F_L)(S_{x,y}(\alpha,\beta)) = \|\bfx - \bfx_{\text{S}}\|_{\R^2}^{-2} \, (\Fourier^{-1} F_{\|\bfx - \bfx_{\text{S}}\|_{\R^2} \, L})(\sin(\gamma-\alpha)).
\end{equation*}
With this, equation~\eqref{eq:FBP_approximate_form_fan} can be rewritten as
\begin{equation*}
f_L(x,y) = \frac{D}{4\pi} \int_0^{2\pi} \|\bfx - \bfx_{\text{S}}\|_{\R^2}^{-2} \int_{-\nicefrac{\pi}{2}}^{\nicefrac{\pi}{2}} (\Fourier^{-1} F_{\|\bfx - \bfx_{\text{S}}\|_{\R^2} \, L})(\sin(\gamma-\alpha)) \, \cos(\alpha) \, \Fan f(\alpha,\beta) \: \d \alpha \, \d \beta,
\end{equation*}
where $\|\bfx - \bfx_{\text{S}}\|_{\R^2}$ and $\gamma$ are independent of $\alpha$.
Therefore, the $\alpha$-integral is a convolution.
Unfortunately, the convolution kernel $\Fourier^{-1} F_{\|\bfx - \bfx_{\text{S}}\|_{\R^2} \, L}$ depends on $\bfx$ so that the convolution has to be done for every reconstruction point $\bfx = (x,y) \in \R^2$, which is computationally costly.

\begin{algorithm}[t]
\caption{Discrete FBP method in fan beam geometry}
\label{algo:FBP_method_fan}
\begin{algorithmic}[1]
\Require Fan beam data $(\Fan f)_{j,k} = \Fan f(\alpha_j,\beta_k)$ for $j=-q,\ldots,q$, $k=0,\ldots,p-1$
\Statex
\State \textbf{choose} low-pass filter $F_L$ with bandwidth $L>0$
\Statex
\For{$k=0,\ldots,p-1$}\Comment{Computation of the discrete convolution}
\For{$i \in I$}
\State ${\displaystyle h(\alpha_i,\beta_k) = \Delta\alpha \sum_{j=-q}^q \Fourier^{-1} F_L(D \sin(\alpha_i - \alpha_j)) \, \cos(\alpha_j) \, \Fan f(\alpha_j,\beta_k)}$
\EndFor
\EndFor
\Statex
\State \textbf{choose} interpolation method $\Int$
\Statex
\For{$m \in I_x$}\Comment{Computation of the discrete back projection}
\For{$n \in I_y$}
\State ${\displaystyle f_\FBP(x_m,y_n) = \frac{D^3}{2p} \sum_{k=0}^{p-1} \bigl((x_m - D\cos(\beta_k))^2 + (y_n - D\sin(\beta_k))^2\bigr)^{-1} \, \Int h(\gamma_{m,n,k},\beta_k)}$
\State \textbf{with} ${\gamma_{m,n,k}} = \sgn(x_m\sin(\beta_k) - y_n\cos(\beta_k)) \, \arccos\biggl(\frac{D - x_m\cos(\beta_k) - y_n\sin(\beta_k)}{\sqrt{(x_m - D\cos(\beta_k))^2 + (y_n - D\sin(\beta_k))^2}}\biggr)$
\EndFor
\EndFor
\Statex
\Ensure Approximate reconstruction $f_\FBP$ on Cartesian grid $\{(x_m,y_n) \mid (m,n) \in I_x \times I_y\}$
\end{algorithmic}
\end{algorithm}

To reduce the complexity of the calculations, we make the following approximation.
Assume that $r \ll D$.
Then, $\|\bfx - \bfx_{\text{S}}\|_{\R^2} \approx D$ for $\bfx \in B_r(0)$ and~\eqref{eq:FBP_approximate_form_fan} can be approximately written as
\begin{equation}
f_L(x,y) \approx \frac{D^3}{4\pi} \int_0^{2\pi} \|\bfx - \bfx_{\text{S}}\|_{\R^2}^{-2} \int_{-\nicefrac{\pi}{2}}^{\nicefrac{\pi}{2}} (\Fourier^{-1} F_{L})(D\sin(\gamma-\alpha)) \, \cos(\alpha) \, \Fan f(\alpha,\beta) \: \d \alpha \, \d \beta.
\label{eq:FBP_approximate_fan}
\end{equation}
This formula can now be treated exactly as~\eqref{eq:FBP_approximate_form} in the parallel case, leading to Algorithm~\ref{algo:FBP_method_fan}.
According to~\cite[Section 5.1.3]{Natterer2001a} the fan beam parameters have to satisfy
\begin{equation*}
q \geq \frac{\varphi}{2\pi} DL, \quad p \geq \frac{2D}{D+r} rL, \quad D \geq 3r.
\end{equation*}

\begin{figure}[b]
\centering
\subfigure[Fan beam data]{\includegraphics[viewport=27 33 425 317, height=0.225\textwidth]{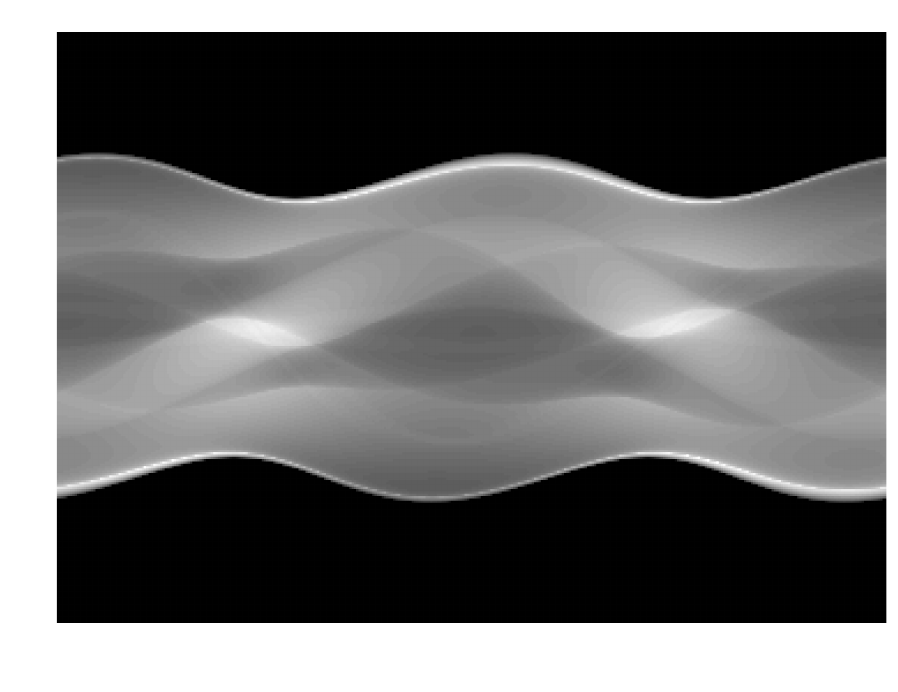}}
\hfil
\subfigure[FBP reconstruction]{\includegraphics[viewport=69 17 383 332, height=0.225\textwidth]{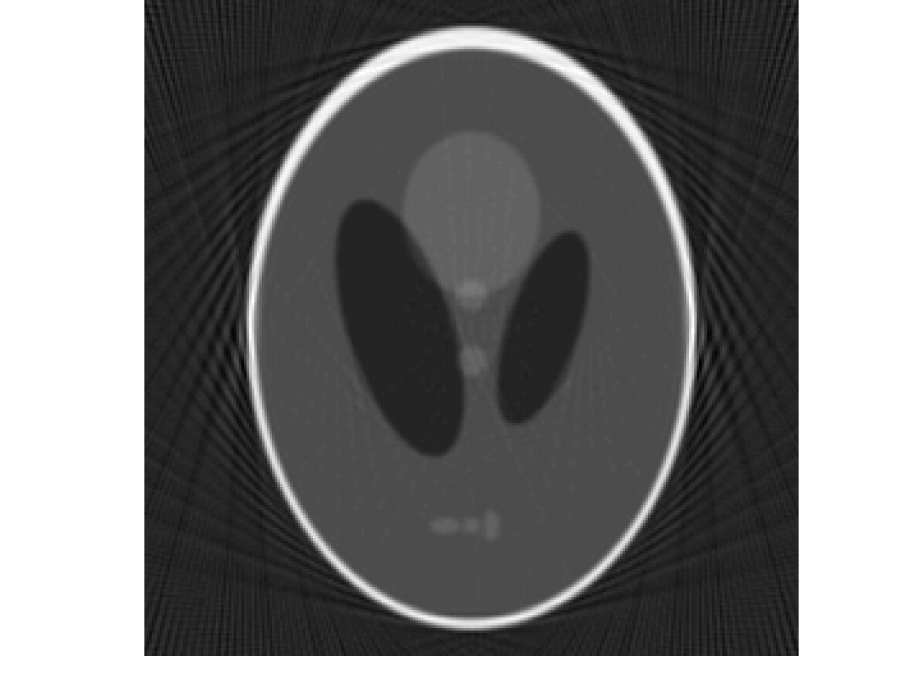}}
\caption{Fan beam reconstruction of Shepp-Logan phantom with Ram-Lak filter and $L = 180$.}
\label{fig:phantom_reconstruction_fan}
\end{figure}

For illustration, we use Algorithm~\ref{algo:FBP_method_fan} to reconstruct the Shepp-Logan phantom from fan beam samples.
The data and FBP reconstruction are displayed in Figure~\ref{fig:phantom_reconstruction_fan}, where we chose $r = 1$, $\varphi = \frac{\pi}{3}$, $D = 3$, $p=270$, $q = 90$ as well as linear interpolation and Ram-Lak filter with $L = 180$.

\chapter{Algebraic reconstruction techniques}

In this chapter we describe another popular approach, given by the class of {\em algebraic} methods, to solve the basic CT reconstruction problem of recovering a compactly supported target function $f \in \L^2(\Omega)$ with $\supp(f) \subseteq B_r(0)$, for $r>0$, on a rectangular image domain $\Omega \subset \R^2$ from its Radon data
\begin{equation*}
\left\{ \Radon f (t,\theta) \mid t \in \R, \, \theta \in [0,\pi) \right\}.
\end{equation*}
In contrast to the analytic approaches described in the previous chapters, algebraic techniques are not based on analytical inversion formulas for the Radon transform but on a fully discrete formulation of the above reconstruction problem.
One of the most popular algebraic methods is given by the so called {\em algebraic reconstruction technique} (ART), which is an implementation of the classical Kaczmarz method for iteratively solving systems of linear equations.

\section{Discrete reconstruction problem}

We start with deriving the fully discrete formulation of the CT reconstruction problem.
In this setting, the sought function $f$ is discretized beforehand to solving the reconstruction problem.
To this end, a set of basis function $\phi_k \in \L^2(\Omega)$, $k=1,\ldots,N$, of the reconstruction space is fixed and the function $f$ is assumed to be expressible as a linear combination of these basis functions, i.e.,
\begin{equation}
f = \sum_{k=1}^N c_k \, \phi_k
\label{eq:algebraic_ansatz}
\end{equation}
for some coefficient vector $c = (c_1,\ldots,c_N)^T \in \R^N$.
Furthermore, we assume that we deal with a finite number of Radon data
\begin{equation*}
y = (\Radon f(t_1,\theta_1),\ldots,\Radon f(t_M,\theta_M))^T \in \R^M.
\end{equation*}
Then, using the ansatz~\eqref{eq:algebraic_ansatz} for the target function $f$, the linearity of the Radon transform $\Radon$ gives
\begin{equation*}
y_j = \Radon f(t_j,\theta_j) = \sum_{k=1}^N c_k \, \Radon \phi_k(t_j,\theta_j)
\quad \forall \, j=1,\ldots,M
\end{equation*}
and, thus, the fully discrete version of the CT reconstruction problem is given by the linear system of equations
\begin{equation}
Ac = y,
\label{eq:reconstruction_problem_discrete}
\end{equation}
whose system matrix
\begin{equation*}
A = (a_{j,k})_{\substack{j=1,\ldots,M \\ k=1,\ldots,N}} \in \R^{M \times N}
\end{equation*}
is called {\em Radon matrix} and consists of the matrix elements
\begin{equation*}
a_{j,k} = \Radon \phi_k(t_j,\theta_j).
\end{equation*}

\bigbreak

The most common choice is the so called {\em pixel basis}.
In this case, the image domain $\Omega \subset \R^2$ is discretized by using a grid of $I_r \times I_c$ small squares, called {\em picture elements} or, in short, {\em pixels}.
To specify the pixel basis, the pixels $\Box_1,\ldots,\Box_N$ are labelled column-wise with $N = I_r \cdot I_c$ and for $k \in \{1,\ldots,N\}$ we define the basis function $\chi_k$ as
\begin{equation*}
\chi_k(x,y) = \begin{cases}
1 & \text{for } (x,y) \in \Box_k \\
0 & \text{for } (x,y) \not\in \Box_k.
\end{cases}
\end{equation*}
Then, the image $f$ can be written as
\begin{equation*}
f = \sum_{k=1}^N c_k \, \chi_k,
\end{equation*}
where the coefficient $c_k$ corresponds to the image's greyscale value at pixel $\Box_k$.
Moreover, the elements of the Radon matrix are of the form
\begin{equation*}
\Radon \chi_k(t_j,\theta_j) = \int_{\ell_{t_j,\theta_j}} \chi_k(x,y) \: \d(x,y) = \Length\bigl(\ell_{t_j,\theta_j} \cap \Box_k\bigr)
\end{equation*}
i.e., the element $a_{j,k}$ corresponds to the length of the intersection of the $j$th Radon line~$\ell_{t_j,\theta_j}$ with the $k$th pixel $\Box_k$ in the reconstruction grid.
Thus, the Radon matrix $A$ is rather sparse and usually very large.
For example, if we use a reconstruction grid with $256 \times 256$ pixels and Radon data in parallel beam geometry with $300$ angles and $201$ parallel lines per angle, we have $N = 65536$ and $M = 60300$ so that $A \in \R^{60300 \times 65536}$ has $3.95 \cdot 10^{9}$ elements.

\begin{figure}[b]
\centering
\subfigure{\includegraphics[height=0.3\textwidth]{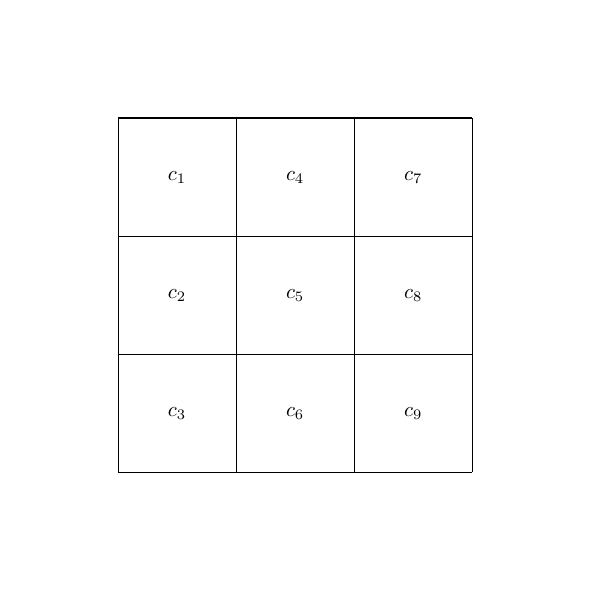}}
\hfil
\subfigure{\includegraphics[height=0.3\textwidth]{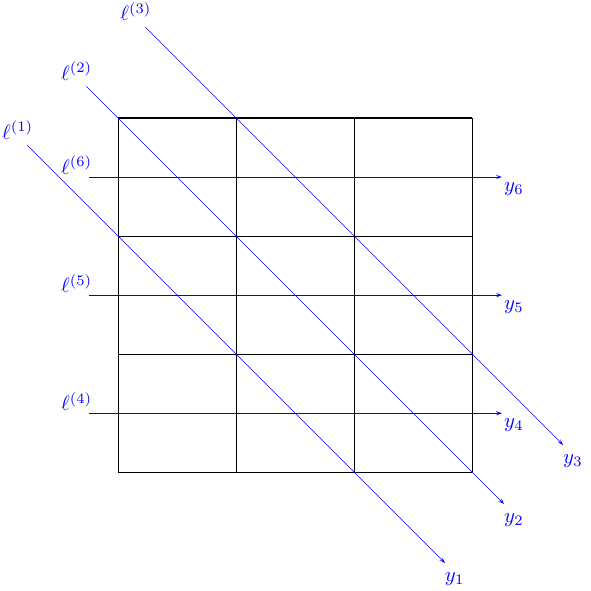}}
\caption{Illustration of the discretization and labeling leading to the linear system in~\eqref{eq:example_lin_system}.}
\label{fig:Radon_matrix_example}
\end{figure}

\begin{example}
We consider the discretization of the image domain $\Omega = [-1,1]^2$ with $3 \times 3$ pixels and the six Radon lines
\begin{equation*}
\ell^{(1)} = \ell_{-\frac{\sqrt{2}}{3},\frac{\pi}{4}}, \quad
\ell^{(2)} = \ell_{0,\frac{\pi}{4}}, \quad
\ell^{(3)} = \ell_{\frac{\sqrt{2}}{3},\frac{\pi}{4}}, \quad
\ell^{(4)} = \ell_{-\frac{2}{3},\frac{\pi}{2}}, \quad
\ell^{(5)} = \ell_{0,\frac{\pi}{2}}, \quad
\ell^{(6)} = \ell_{\frac{2}{3},\frac{\pi}{2}}
\end{equation*}
see also Figure~\ref{fig:Radon_matrix_example}.
Then, the linear system $Ac = y$ takes the form
\begin{equation}
\frac{2}{3} \begin{pmatrix}
0 & \sqrt{2} & 0 & 0 & 0 & \sqrt{2} & 0 & 0 & 0 \\
\sqrt{2} & 0 & 0 & 0 & \sqrt{2} & 0 & 0 & 0 & \sqrt{2} \\
0 & 0 & 0 & \sqrt{2} & 0 & 0 & 0 & \sqrt{2} & 0 \\
0 & 0 & 1 & 0 & 0 & 1 & 0 & 0 & 1 \\
0 & 1 & 0 & 0 & 1 & 0 & 0 & 1 & 0 \\
1 & 0 & 0 & 1 & 0 & 0 & 1 & 0 & 0
\end{pmatrix}
\begin{pmatrix}
c_1 \\ c_2 \\ c_3 \\ c_4 \\ c_5 \\ c_6 \\ c_7 \\ c_8 \\ c_9 
\end{pmatrix}
=
\begin{pmatrix}
y_1 \\ y_2 \\ y_3 \\ y_4 \\ y_5 \\ y_6
\end{pmatrix}.
\label{eq:example_lin_system}
\end{equation}
Note that this system is {\em underdetermined} and, thus, has infinitely many solutions if it is solvable.
\end{example}

\section{Linear least squares approximation}

Recall that the linear system $Ac = y$ in~\eqref{eq:reconstruction_problem_discrete} is called {\em overdetermined} if there are more equations that unknowns, i.e.,
\begin{equation*}
M > N,
\end{equation*}
and {\em underdetermined} if there are more unknowns than equations, i.e.,
\begin{equation*}
N > M.
\end{equation*}
In the case of an overdetermined system, the equation $Ac=y$ is likely to not have an exact solution and one typically applies {\em linear least squares approximation} to instead solve the minimization problem
\begin{equation}
\|Ac - y\|_{\R^M}^2 \longrightarrow \min_{c \in \R^N} !.
\label{eq:least_squares_problem}
\end{equation}
Consequently, its solution
\begin{equation*}
c^* = \argmin_{c \in \R^N} \|Ac - y\|_{\R^M}^2
\end{equation*}
provides the vector $y^* = Ac^*$ which is the closest vector to $y$ in the range $\im(A)$ of $A$.

To solve the minimization problem~\eqref{eq:least_squares_problem} for a given measurement vector $y \in \R^M$, we define the objective function $J: \R^N \to \R_{\geq 0}$ as
\begin{equation*}
J(c) = \|Ac - y\|_{\R^M}^2 = (Ac-y)^T(Ac-y) = c^T A^TA c - 2 c^T A^T y + y^T y
\quad \mbox{ for } c \in \R^N
\end{equation*}
with the gradient
\begin{equation*}
\nabla J(c) = 2 A^TA c - 2 A^T y
\end{equation*}
and the Hessian
\begin{equation*}
\nabla^2 J(c) = 2 A^TA,
\end{equation*}
which is positive semi-definite.
Consequently, any solution to~\eqref{eq:least_squares_problem} is a root of $\nabla J$ and satisfies the so called {\em normal equation}
\begin{equation}
A^TA c = A^T y,
\label{eq:normal_equation}
\end{equation}
i.e., the vector $Ac - y$ belongs to the nullspace $\ker(A^T)$ of $A^T$ and, hence, is orthogonal to $\im(A)$.
If $A \in \R^{M \times N}$ has full rank
\begin{equation*}
\rank(A) = N,
\end{equation*}
the matrix $A^TA \in \R^{N \times N}$ is positive definite so that the normal equation~\eqref{eq:normal_equation} has a unique solution $c^* \in \R^N$, which then also uniquely solves the least squares minimization problem~\eqref{eq:least_squares_problem}.
In general, however, $A^TA$ need not be invertible so that the {\em least squares solution} $c^*$ is not uniquely determined.
Nonetheless, the corresponding range element $y^* = Ac^* \in \R^M$ is the unique element in $\im(A)$ that is closest to the given measurement vector $y \in \R^M$.

The Gram matrix $A^TA \in \R^{N \times N}$ is typically very large, dense and ill-conditioned so that solving~\eqref{eq:normal_equation} is numerically infeasible.
An efficient and numerically stable method for solving the least squares problem~\eqref{eq:least_squares_problem} is based on the QR factorization of the Radon matrix $A \in \R^{M \times N}$,
\begin{equation*}
A = Q R,
\end{equation*}
with an orthogonal matrix $Q \in \R^{M \times M}$ and an upper triangular matrix $R \in \R^{M \times N}$ of the form
\begin{equation*}
R = \begin{pmatrix}
\widehat{R} \\ 0
\end{pmatrix},
\end{equation*}
where
\begin{equation*}
\widehat{R} = \begin{pmatrix}
r_{11} & \ldots & r_{1N} \\
& \ddots & \vdots \\
0 & & r_{NN}
\end{pmatrix}
\in \R^{N \times N}.
\end{equation*}
Note that $A$ has full rank if and only if no diagonal entry of $\widehat{R}$ vanishes, i.e.,
\begin{equation*}
r_{kk} \neq 0
\quad \forall \, k=1,\ldots,N.
\end{equation*}

\bigbreak

Since the orthogonal matrix $Q$ has the inverse $Q^{-1} = Q^T$ and $Q^T$ is norm-preserving, the objective function $J: \R^N \to \R_{\geq 0}$ can be rewritten as
\begin{equation*}
J(c) = \|Ac - y\|_{\R^M}^2 = \|QRc - y\|_{\R^M}^2 = \|Rc -Q^T  y\|_{\R^M}^2 = \|\widehat{R}c - \widehat{y_1}\|_{\R^N}^2 + \|\widehat{y}_2\|_{\R^{M-N}}^2,
\end{equation*}
where
\begin{equation*}
Q^T y = \begin{pmatrix}
\widehat{y}_1 \\ \widehat{y}_2
\end{pmatrix}
\mbox{ with } \widehat{y}_1 \in \R^N \mbox{ and } \widehat{y}_2 \in \R^{M-N}.
\end{equation*}
Consequently, a solution to the minimization problem~\eqref{eq:least_squares_problem} can be computed by solving the triangular system
\begin{equation*}
\widehat{R}c = \widehat{y}_1
\end{equation*}
with a backward substitution and the {\em least squares error} is given by
\begin{equation*}
\|Ac^* - y\|_{\R^M}^2 = \|\widehat{y}_2\|_{\R^{M-N}}^2,
\end{equation*}
where the least squares solution $c^* \in \R^N$ is uniquely determined if and only if $A$ has full rank.

\subsection*{Tikhonov regularization}

When solving the discrete reconstruction problem~\eqref{eq:reconstruction_problem_discrete}, one still has to consider the ill-posedness of the CT reconstruction problem and apply a regularization method.
A typical strategy is to incorporate prior knowledge into the reconstruction procedure.
This can be done by applying the general framework of {\em variational regularization}.
In this setting, a regularized solution of the linear system~\eqref{eq:reconstruction_problem_discrete} is computed by minimizing a {\em Tikhonov type functional} $J_\gamma: \R^N \to \R_{\geq 0}$, defined as
\begin{equation*}
J_\gamma(c) = \|Ac - y\|_{\R^M}^2 + \gamma \, \Lambda(c)
\quad \mbox{ for } c \in \R^N,
\end{equation*}
where $\gamma > 0$ is the regularization parameter and $\Lambda: \R^N \to \R_{\geq 0}$ is the {\em prior function}.
The first term denotes the {\em data fidelity term} and controls the data error, whereas the second term acts as a {\em penalty term} and encodes the prior knowledge about the solution.
The parameter $\gamma > 0$ is used to compromise between the approximation quality of a solution $c^*_\gamma$ and its regularity.

There are many possible choices for the prior term, depending on the particular features of the unknown object that shall be preserved or emphasized.
One prominent example is the total variation (TV) seminorm, which can be used for enforcing edge-preserving reconstructions.
In this paragraph, we focus on another relevant special case, where the functional $\Lambda$ is of the form
\begin{equation*}
\Lambda(c) = \|c\|_B^2 = c^TBc
\quad \mbox{ for } c \in \R^N
\end{equation*}
with a symmetric and positive definite matrix $B \in \R^{N \times N}$.
We remark that choosing the identity matrix $B = \Id$ corresponds to the classical {\em Tikhonov regularization}.

\begin{theorem}[Regularized least squares]
Let $A \in \R^{M \times N}$ and let $B \in \R^{N \times N}$ be symmetric and positive definite.
Then, for any $\gamma > 0$ and $y \in \R^M$ the {\em regularized least squares problem}
\begin{equation}
\|Ac - y\|_{\R^M}^2 + \gamma \, \|c\|_B^2 \longrightarrow \min_{c \in \R^N} !
\label{eq:least_squares_problem_regularized}
\end{equation}
has the unique solution
\begin{equation*}
c^*_\gamma = (A^TA + \gamma \, B)^{-1} A^T y \in \R^N,
\end{equation*}
which satisfies
\begin{equation*}
c^*_\gamma \xrightarrow{\gamma \to \infty} 0
\quad \mbox{ and } \quad
c^*_\gamma \xrightarrow{\gamma \to 0} c^*_0,
\end{equation*}
where $c^*_0 \in \R^N$ is that solution of the linear least squares problem
\begin{equation*}
\|Ac - y\|_{\R^M}^2 \longrightarrow \min_{c \in \R^N} !
\end{equation*}
that minimizes the norm $\|\cdot\|_B$ induced by the matrix $B$.
\end{theorem}

\begin{proof}
For fixed $\gamma > 0$ and $y \in \R^M$ we regard the cost function $J_\gamma: \R^N \to \R_{\geq 0}$ with
\begin{equation*}
J_\gamma(c) = \|Ac - y\|_{\R^M}^2 + \gamma \, \|c\|_B^2 = c^T (A^TA + \gamma \, B) c - 2 c^T A^T y + y^T y.
\end{equation*}
Then, for any $c \in \R^N$ we have
\begin{equation*}
\nabla J_\gamma(c) = 2(A^TA + \gamma \, B) c - 2 A^T y
\quad \mbox{ and } \quad
\nabla^2 J_\gamma(c) = 2(A^TA + \gamma \, B).
\end{equation*}
Since $B$ is positive definite by assumption, for any $v \in \R^N\setminus\{0\}$ follows that
\begin{equation*}
v^T(A^TA + \gamma \, B)v = \|Av\|_{\R^M}^2 + \gamma \, \|v\|_B^2 > 0,
\end{equation*}
i.e., $A^TA + \gamma \, B$ is positive definite as well.
Consequently, $J_\gamma$ is strictly convex and possesses a unique global minimum which is characterized by the generalized normal equation
\begin{equation*}
(A^TA + \gamma \, B) \, c = A^T y,
\end{equation*}
i.e., the regularized least squares problem~\eqref{eq:least_squares_problem_regularized} has the unique solution
\begin{equation*}
c^*_\gamma = (A^TA + \gamma \, B)^{-1} A^T y.
\end{equation*}

To characterize the asymptotic behaviour of $c^*_\gamma$ for $\gamma \to \infty$ and $\gamma \to 0$, we first establish a suitable representation of $c^*_\gamma$.
Since $B$ is positive definite, we can rewrite $J_\gamma$ as
\begin{equation*}
J_\gamma(c) = \|Ac - y\|_{\R^M}^2 + \gamma \, \|c\|_B^2 = \|A B^{-\nicefrac{1}{2}} B^{\nicefrac{1}{2}} c - y\|_{\R^M}^2 + \gamma \, \|B^{\nicefrac{1}{2}} c\|_{\R^N}^2
\end{equation*}
and, thus, by defining
\begin{equation*}
C = A B^{-\nicefrac{1}{2}} \in \R^{M \times N}
\quad \mbox{ and } \quad
b = B^{\nicefrac{1}{2}} c \in \R^N
\end{equation*}
the minimization problem~\eqref{eq:least_squares_problem_regularized} can be equivalently written as
\begin{equation}
\|Cb - y\|_{\R^M}^2 + \gamma \, \|b\|_{\R^N}^2 \longrightarrow \min_{b \in \R^N} !.
\label{eq:least_squares_problem_regularized_reformulated}
\end{equation}
To solve~\eqref{eq:least_squares_problem_regularized_reformulated}, we employ the singular value decomposition of $C$,
\begin{equation*}
C = U \Sigma V^T,
\end{equation*}
where $U = (u_1,\ldots,u_M) \in \R^{M \times M}$ and $V = (v_1,\ldots,v_N) \in  \R^{N \times N}$ are orthogonal matrices and
\begin{equation*}
\Sigma = \begin{pmatrix}
\sigma_1 & & & 0 \\
& \ddots & & \\
& &  \sigma_\rho & \\
0 & & & 0 
\end{pmatrix}
\in \R^{M \times N}
\end{equation*}
carries the singular values $\sigma_1 \geq \ldots \geq \sigma_\rho > 0$ of $C$ with $\rho = \rank(A)$.
Then, the solution to~\eqref{eq:least_squares_problem_regularized_reformulated} is given by
\begin{equation*}
b^*_\gamma = \sum_{j=1}^\rho \frac{\sigma_j}{\sigma_j^2 + \gamma} \, (u_j^T y) \, v_j \in \R^N
\end{equation*}
and satisfies
\begin{equation*}
b^*_\gamma \xrightarrow{\gamma \to \infty} 0
\end{equation*}
as well as
\begin{equation*}
b^*_\gamma \xrightarrow{\gamma \to 0} b^*_0 = \sum_{j=1}^\rho \sigma_j^{-1} \, (u_j^T y) \, v_j = C^+ y,
\end{equation*}
where $C^+$ is the pseudoinverse of $C$, i.e., $b^*_0$ is the norm-minimal solution of the non-regularized least squares problem
\begin{equation*}
\|Cb - y\|_{\R^M}^2 \longrightarrow \min_{b \in \R^N} !.
\end{equation*}
Consequently, for the solution $c^*_\gamma = B^{-\nicefrac{1}{2}} b^*_\gamma$ to~\eqref{eq:least_squares_problem_regularized} follows that
\begin{equation*}
c^*_\gamma \xrightarrow{\gamma \to \infty} 0
\quad \mbox{ and } \quad
c^*_\gamma \xrightarrow{\gamma \to 0} B^{-\nicefrac{1}{2}} b^*_0 = c^*_0.
\qedhere
\end{equation*}
\end{proof}

\bigbreak

For illustration, we use Tikhonov regularized least squares approximation with $B = \Id$ and $\gamma = 0.05$ to reconstruct the Shepp-Logan phantom (see Figure~\ref{fig:shepp-logan_phantom}) and the thorax phantom (see Figure~\ref{fig:thorax_phantom}) from finite Radon data in parallel beam geometry with $M = 130$ and $N = 260$.
To this end, we use a reconstruction grid with $256 \times  256$ pixels and solve the generalized normal equation
\begin{equation*}
(A^TA + \gamma \, B) \, c = A^T y
\end{equation*}
by applying the conjugated gradient method.
The reconstructions can be found in Figure~\ref{fig:phantom_reconstruction_tikhonov}.

\begin{figure}[t]
\centering
\subfigure[Shepp-Logan phantom]{\includegraphics[viewport=69 17 383 332, height=0.25\textwidth]{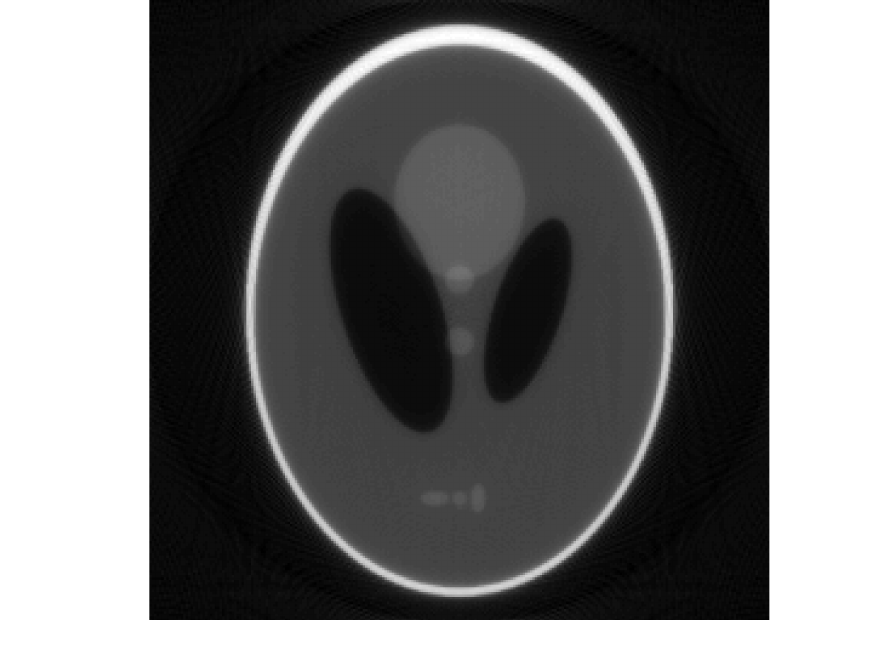}}
\hfil
\subfigure[Thorax phantom]{\includegraphics[viewport=69 17 383 332, height=0.25\textwidth]{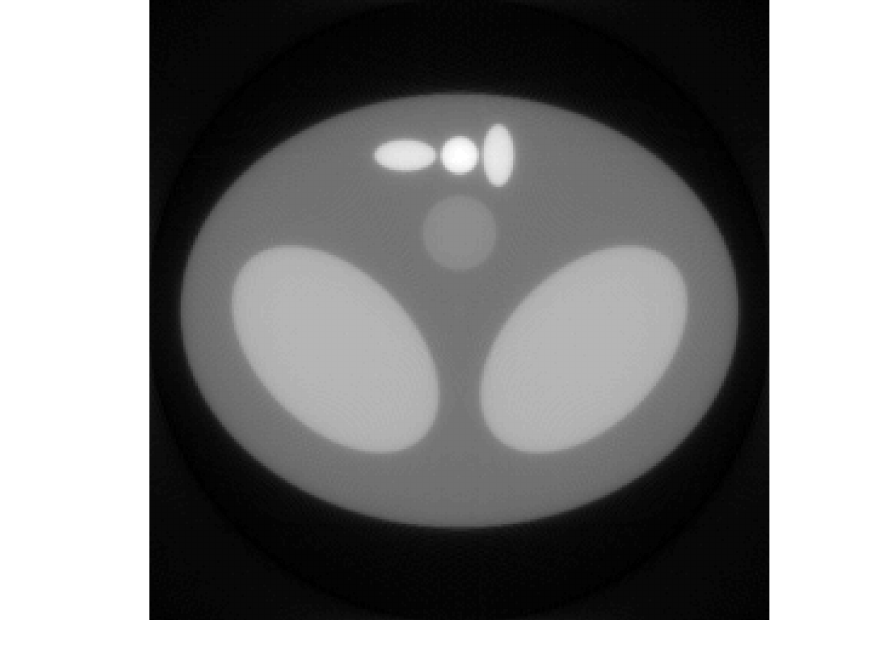}}
\caption{Tikhonov reconstructions of two phantoms with $B = \Id$ and $\gamma = 0.05$.}
\label{fig:phantom_reconstruction_tikhonov}
\end{figure}

\section{Kaczmarz's method}

In the case of an underdetermined system, the equation $Ac = y$ is likely to have infinitely many solutions and we apply an iterative procedure called {\em Kaczmarz's method}.
Let $a_j \in \R^N$ denote the $j$th row of the Radon matrix $A \in \R^{M \times N}$ and $y_j \in \R$ be the $j$th component of the measurement vector $y \in \R^M$.
Then, the linear system $A c = y$ consists of the $M$ linear equations
\begin{equation*}
a_j^T c = y_j
\quad \mbox{ for } 1 \leq j \leq M
\end{equation*}
and the idea of Kaczmarz's method is to generate a sequence $c^{(0)}, c^{(1)}, \ldots$ of vectors $c^{(k)} \in \R^N$ satisfying {\em one} of the above equations.
To this end, the previous iterate $c^{(k-1)}$ is projected onto the {\em affine space} $S_{a_{j(k)},y_{j(k)}}$ generated by the $j(k)$th equation $a_{j(k)}^T c = y_{j(k)}$.

\begin{definition}[Affine space]
For $r \in \R^N$ and $p \in \R$, the {\em affine space} $S_{r,p}$ is defined as
\begin{equation*}
S_{r,p} = \Bigl\{ x \in \R^N \Bigm| r^T x = p \Bigr\} \subset \R^N.
\end{equation*}
\end{definition}

Note that $r$ is orthogonal to $S_{r,p}$ and, thus, the orthogonal projection $\Pi_{S_{r,p}}$ onto $S_{r,p}$ is given by
\begin{equation*}
\Pi_{S{r,p}} u = u - \frac{r^T u - p}{r^T r} \, r
\quad \mbox{ for } u \in \R^N.
\end{equation*}
For the sake of brevity, let $\Pi_j$ denote the orthogonal projection onto the affine space $S_{a_j,y_j}$ and define
\begin{equation*}
\Pi = \Pi_M \cdots \Pi_1.
\end{equation*}
With this, Kaczmarz's method can be formulated as follows.
Choose a starting vector $c^{(0)} \in \R^N$ and iterate
\begin{equation}
c^{(k)} = \Pi c^{(k-1)}
\quad \mbox{ for } k \in \N
\label{eq:Kaczmarz_method}
\end{equation}
until a stopping criterion is fulfilled, see also Algorithm~\ref{algo:Kaczmarz_method}.

\begin{algorithm}[t]
\caption{Kaczmarz's method}
\label{algo:Kaczmarz_method}
\begin{algorithmic}[1]
\Require System matrix $A \in \R^{M \times N}$, measurement vector $y \in \R^M$ and tolerance $\delta > 0$
\Statex
\State \textbf{choose} starting vector $c^{(0)} \in \R^N$
\Statex
\For{$k=1,2,\ldots$}
\State $c^{(k,0)} = c^{(k-1)}$
\For{$j = 1,\ldots,M$}
\State ${\displaystyle c^{(k,j)} = c^{(k,j-1)} - \frac{a_j^T c^{(k,j-1)} - y_j}{a_j^T a_j} \, a_j}$
\EndFor
\State $c^{(k)} = c^{(k,M)}$
\If{$\|c^{(k)} - c^{(k-1)}\|_{\R^N} \leq \delta$ or $\|Ac^{(k)} - y\|_{\R^M} \leq \delta \, \|y\|_{\R^M}$}
\State \textbf{return} $c^\ast = c^{(k)}$
\EndIf
\EndFor
\Statex
\Ensure Approximate solution $c^\ast$ to $Ac = y$
\end{algorithmic}
\end{algorithm}

\bigbreak

We will show that Kaczmarz's method~\eqref{eq:Kaczmarz_method} converges to a solution of $Ac = y$ for $k \to \infty$.
The proof is based on the following result on the iteration of {\em linear} orthogonal projections.

\begin{lemma}
\label{lem:orthogonal_projection_limit}
For $j = 1,\ldots,M$, let $P_j$ be the orthogonal projection onto a linear space $V_j \subset \R^N$ and define
\begin{equation*}
P = P_M \cdots P_1.
\end{equation*} 
Then, we have
\begin{equation*}
P^k x \xrightarrow{k \to \infty} Tx
\quad \forall \, x \in \R^N,
\end{equation*}
where $T$ is the orthogonal projection onto $\ker(\Id - P)$.
\end{lemma}

\begin{proof}
Let $x \in \R^N$.
Since $P_j$ is an orthogonal projection, its operator norm is bounded by $1$, i.e.,
\begin{equation*}
\|P_j\| \leq 1,
\end{equation*}
and, thus, we also have
\begin{equation*}
\|P\| \leq \|P_M\| \cdot \ldots \cdot \|P_1\| \leq 1.
\end{equation*}
Consequently, the sequence $(\|P^k x\|_{\R^N})_{k \in \N}$ is monotonically decreasing and bounded from below so that its limit $p \in \R_{\geq 0}$ exists.
If $p = 0$, this implies that
\begin{equation*}
P^k x \xrightarrow{k \to \infty} 0.
\end{equation*}
To prove that we then also have $Tx = 0$, recall the orthogonal decomposition
\begin{equation*}
\R^N = \ker(\Id - P) \oplus \ker(\Id - P)^\perp = \ker(\Id - P) \oplus \im(\Id - P).
\end{equation*}
Hence, $\Id - T$ is the orthogonal projection onto $\im(\Id - P)$ so that
\begin{equation*}
(\Id - T)(\Id - P) = \Id - P
\quad \implies \quad
T = TP
\end{equation*}
and
\begin{equation*}
(\Id - P) T = 0
\quad \implies \quad
T = PT.
\end{equation*}
In particular, we indeed have
\begin{equation*}
Tx = P^k Tx = T P^k x \xrightarrow{k \to \infty} 0
\end{equation*}
so that
\begin{equation*}
P^k x \xrightarrow{k \to \infty} Tx.
\end{equation*}

\bigbreak

To deal with the case $p > 0$, we first prove by induction on the number $M$ of factors in $P$ that for any sequence $(x_k)_{k \in \N}$ with $\|x_k\|_{\R^N} \leq 1$ and $\|Px_k\|_{\R^N} \xrightarrow{k \to \infty} 1$ we have
\begin{equation*}
\lim_{k \to \infty} (\Id - P)x_k = 0.
\end{equation*}
For $M = 1$ holds that
\begin{align*}
\|(\Id - P)x_k\|_{\R^N}^2 & = \|(\Id - P_1)x_k\|_{\R^N}^2 = \|x_k\|_{\R^N}^2 - 2 x_k^T P_1 x_k + \|P_1x_k\|_{\R^N}^2 \\
& = \|x_k\|_{\R^N}^2 - 2 (x_k-P_1x_k)^T P_1 x_k - \|P_1x_k\|_{\R^N}^2 = \|x_k\|_{\R^N}^2 - \|P_1x_k\|_{\R^N}^2.
\end{align*}
Since $\|x_k\|_{\R^N} \leq 1$ and $\|P_1x_k\|_{\R^N} \xrightarrow{k \to \infty} 1$ by assumption, we have $\|(\Id - P)x_k\|_{\R^N} \xrightarrow{k \to \infty} 0$ and the induction seed holds.
Now, assume that the induction hypothesis holds for $M-1$ factors.
We put
\begin{equation*}
P = P_M S
\quad \mbox{ for } \quad
S = P_{M-1} \cdots P_1
\end{equation*}
and obtain
\begin{equation*}
(\Id - P) x_k = (\Id - S)x_k + (S - P)x_k = (\Id - S)x_k + (\Id - P_M)Sx_k.
\end{equation*}
The first term goes to $0$ by the induction hypothesis and the second term goes to $0$ by the induction seed, since
\begin{equation*}
\|Sx_k\|_{\R^N} \leq \|P_{M-1}\| \cdot \ldots \cdot \|P_1\| \cdot \|x_k\|_{\R^N} \leq 1
\end{equation*}
and
\begin{equation*}
\|P_MSx_k\|_{\R^N} = \|Px_k\|_{\R^N} \xrightarrow{k \to \infty} 1.
\end{equation*}

We now set
\begin{equation*}
x_k = \|P^kx\|_{\R^N}^{-1} \, P^k x \in \R^N.
\end{equation*}
Then,
\begin{equation*}
\|x_k\|_{\R^N} = 1
\quad \land \quad
\|Px_k\|_{\R^N} = \|P^kx\|_{\R^N}^{-1} \, \|P^{k+1}x\|_{\R^N} \xrightarrow{k \to \infty} 1
\end{equation*}
and, consequently,
\begin{equation*}
\lim_{k \to \infty} (\Id - P) x_k = 0
\quad \implies \quad
\lim_{k \to \infty} (\Id - P) P^k x = 0 = \lim_{k \to \infty} P^k (\Id - P) x.
\end{equation*}
This implies that $P^k z$ goes to $0$ for all $z \in \im(\Id - P)$.
On the other hand, for $z \in \ker(\Id - P)$ we have
\begin{equation*}
P^k z = z
\quad \forall \, k \in \N.
\end{equation*}
Hence, with the orthogonal decomposition
\begin{equation*}
\ker(\Id - P) \oplus \im(\Id - P) = \R^N
\end{equation*}
follows that
\begin{equation*}
P^k x = P^k Tx + P^k (\Id - T)x = Tx + P^k (\Id - T)x \xrightarrow{k \to \infty} Tx
\end{equation*}
and the proof is complete.
\end{proof}

With Lemma~\ref{lem:orthogonal_projection_limit} we are now prepared to prove the convergence of Kaczmarz's method~\eqref{eq:Kaczmarz_method}.

\begin{theorem}
\label{theo:Kaczmarz_method_convergence}
Let $A \in \R^{M \times N}$ and $y \in \R^M$ be given.
If the linear system $Ac = y$ has at least one solution, then Kaczmarz's method~\eqref{eq:Kaczmarz_method} converges to a solution $c^\ast$ of the system.
Moreover, if $c^{(0)} \in \im(A^T)$, e.g. $c^{(0)} = 0$, then Kaczmarz's method converges to a minimal norm solution, i.e.,
\begin{equation*}
c^\ast = \argmin_{Ac = y} \|c\|_{\R^N}.
\end{equation*}
\end{theorem}

\begin{proof}
Let $\Q_j$, for $j = 1,\ldots,M$, be the orthogonal projection onto $\ker(a_j)$, i.e.,
\begin{equation*}
\Q_j u = u - \frac{a_j^T u}{a_j^T a_j} \, a_j
\quad \mbox{ for } u \in \R^N,
\end{equation*}
and define
\begin{equation*}
\Q = \Q_M \cdots \Q_1.
\end{equation*}
Moreover, let $c \in \R^N$ be a solution to $Ac = y$ and let $c^{(0)} \in \R^N$.
Then, for any $x \in \R^N$ we have
\begin{equation*}
\Pi_j x = x - \frac{a_j^T x - y_j}{a_j^T aj} \, a_j = c + (x-c) - \frac{a_j^T (x-c)}{a_j^T aj} \, a_j = c + \Q_j(x - c)
\end{equation*}
so that
\begin{equation*}
\Pi x = c + \Q(x - c)
\end{equation*}
and, consequently,
\begin{equation*}
c^{(k)} = \Pi^k c^{(0)} = c + \Q^k(c^{(0)} - c)
\quad \forall \, k \in \N.
\end{equation*}
Let $\T$ be the orthogonal projection onto the nullspace $\ker(A)$ of $A$, which satisfies
\begin{equation*}
\ker(A) = \bigcap_{j = 1}^M \ker(a_j) = \im(\Q) = \ker(\Id - \Q).
\end{equation*}
Then, with Lemma~\ref{lem:orthogonal_projection_limit} we have
\begin{equation*}
c^{(k)} = c + \Q^k(c^{(0)} - c) \xrightarrow{k \to \infty} c + \T(c^{(0)} - c) = (\Id - \T) \, c + \T c^{(0)} = c^\ast,
\end{equation*}
where
\begin{equation*}
Ac^\ast = Ac - A \T c + A \T c^{(0)} = Ac = y,
\end{equation*}
i.e., $c^\ast \in \R^N$ is a solution to $Ac = y$.
If $c^{(0)} \in \im(A^T) = \ker(A)^\perp$, we have $\T c^{(0)} = 0$ and, hence,
\begin{equation*}
c^\ast = (\Id - \T) c = A^+ y
\end{equation*}
with the pseudoinverse $A^+$ of $A$, since the orthogonal projection $\T$ onto $\ker(A)$ can be expressed as
\begin{equation*}
\T = \Id - A^+ A.
\end{equation*}
Consequently, $c^\ast$ is the minimal norm solution to $Ac = y$.
\end{proof}

Note that Kaczmarz's method can be modified by introducing a relaxation parameter $\omega > 0$.
To this end, we replace the affine-linear projections $\Pi_j$ by $\Pi_j^\omega = (1-\omega) \, \Id + \omega \, \Pi_j$, i.e.,
\begin{equation*}
\Pi_j^\omega u = u - \omega \, \frac{a_j^T u - y_j}{a_j^T a_j} \, a_j
\quad \mbox{ for } u \in \R^N.
\end{equation*}
One can show that Theorem~\ref{theo:Kaczmarz_method_convergence} also holds for Kaczmarz's method with relaxation if $\omega \in (0,2)$.

\begin{remark}
We collect the following remarks on the convergence of Kaczmarz's method~\eqref{eq:Kaczmarz_method}.
\begin{itemize}
\item[(a)] If the affine spaces $S_{a_j,y_j}$ are almost orthogonal, Kaczmarz's method usually produces a good approximate solution after only a few iterations.
\item[(b)] If the affine spaces $S_{a_j,y_j}$ are almost parallel, which is usually the case in tomography, then Kaczmarz's method may converge very slowly.
The convergence can sometimes be accelerated by randomly permuting the rows of the system matrix $A \in \R^{M \times N}$.
\item[(c)] One still has to address the ill-posedness of the CT reconstruction problem and apply a regularization strategy.
One option is to stop the iteration early enough before convergence.
\end{itemize}
\end{remark}

\bigbreak

The original {\em algebraic reconstruction technique} (ART) involves a non-negativity constraint.
This can be incorporated by using positive basis functions in~\eqref{eq:algebraic_ansatz} and ensuring $c \geq 0$ by applying the projection
\begin{equation*}
\Pi^+ = \Pi_M^+ \cdots \Pi_1^+,
\end{equation*}
where
\begin{equation*}
\Pi_j^+ u = \max\biggl(0,u - \frac{a_j^T u - y_j}{a_j^T a_j} \, a_j\biggr)
\quad \mbox{ for } u \in \R^N.
\end{equation*}
With this, Kaczmarz's method with non-negativity constraint can be formulated as follows.
Choose a starting vector $c^{(0)} \in \R^N$ and iterate
\begin{equation*}
c^{(k)} = \Pi^+ c^{(k-1)}
\quad \mbox{ for } k \in \N
\end{equation*}
until a stopping criterion is fulfilled.
This classical version of ART is summarized in Algorithm~\ref{algo:Kaczmarz_method_non_negative}.

\begin{algorithm}[t]
\caption{Kaczmarz's method with non-negativity constraint}
\label{algo:Kaczmarz_method_non_negative}
\begin{algorithmic}[1]
\Require System matrix $A \in \R^{M \times N}$, measurement vector $y \in \R^M$ and tolerance $\delta > 0$
\Statex
\State \textbf{choose} starting vector $c^{(0)} \in \R^N$
\Statex
\For{$k=1,2,\ldots$}
\State $c^{(k,0)} = c^{(k-1)}$
\For{$j = 1,\ldots,M$}
\State ${\displaystyle c^{(k,j)} = \max\biggl(0,c^{(k,j-1)} - \frac{a_j^T c^{(k,j-1)} - y_j}{a_j^T a_j} \, a_j\biggr)}$
\EndFor
\State $c^{(k)} = c^{(k,M)}$
\If{$\|c^{(k)} - c^{(k-1)}\|_{\R^N} \leq \delta$ or $\|Ac^{(k)} - y\|_{\R^M} \leq \delta \, \|y\|_{\R^M}$}
\State \textbf{return} $c^\ast = c^{(k)}$
\EndIf
\EndFor
\Statex
\Ensure Approximate solution $c^\ast$ to $Ac = y$ with side condition $c \geq 0$
\end{algorithmic}
\end{algorithm}

To illustrate the reconstruction with ART, we use Algorithm~\ref{algo:Kaczmarz_method_non_negative} to recover the Shepp-Logan phantom (see Figure~\ref{fig:shepp-logan_phantom}) and the thorax phantom (see Figure~\ref{fig:thorax_phantom}) on a reconstruction grid with $256 \times  256$ pixels from Radon data in parallel beam geometry with $M = 120$ and $N = 240$.
The ART reconstructions of both phantoms are displayed in Figure~\ref{fig:phantom_reconstruction_art}. 

\begin{figure}[b]
\centering
\subfigure[Shepp-Logan phantom]{\includegraphics[viewport=69 17 383 332, height=0.25\textwidth]{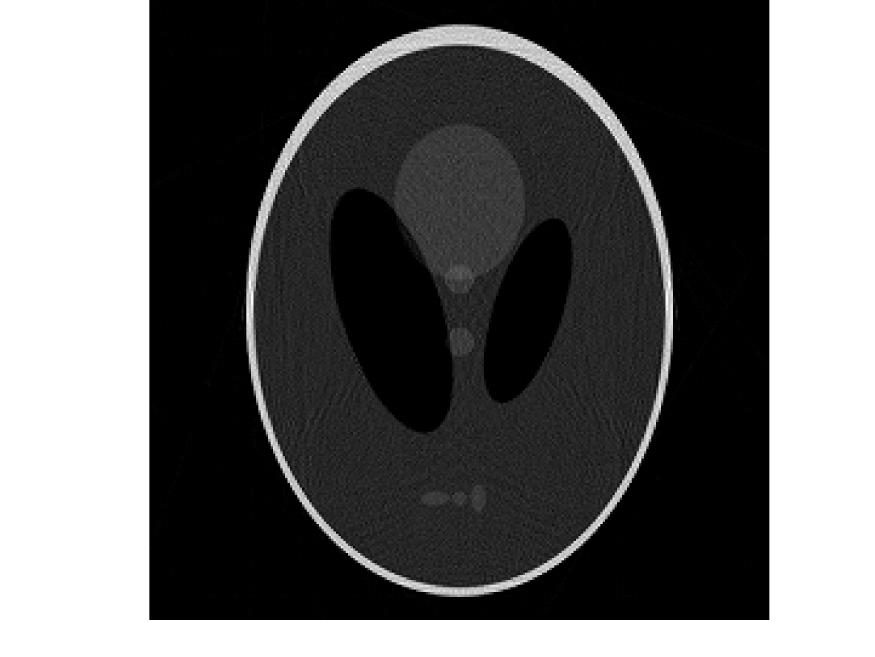}}
\hfil
\subfigure[Thorax phantom]{\includegraphics[viewport=69 17 383 332, height=0.25\textwidth]{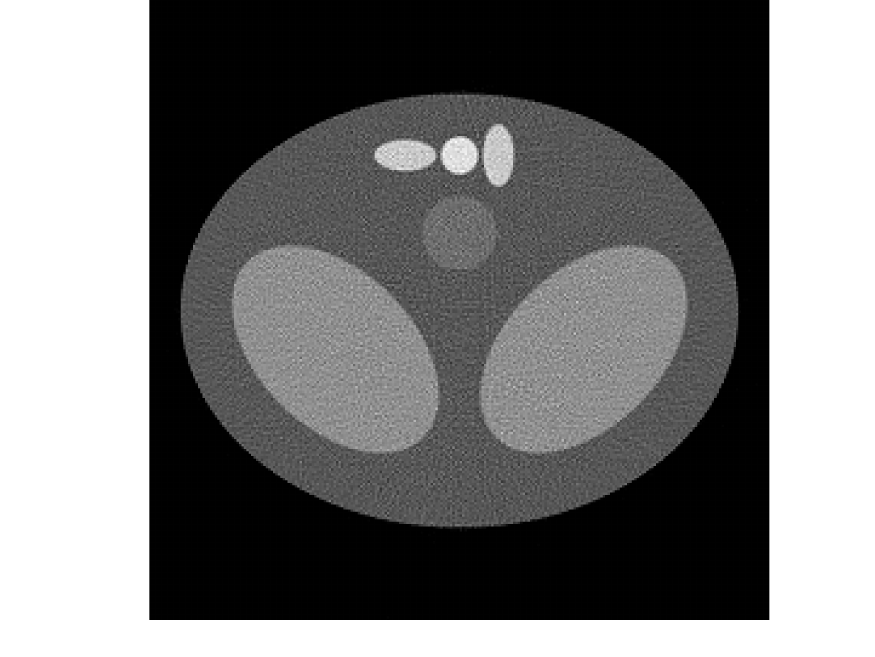}}
\caption{ART reconstructions of two phantoms.}
\label{fig:phantom_reconstruction_art}
\end{figure}

\chapter{Three-dimensional reconstruction}

In this chapter we introduce two different multivariate versions of the Radon transform, namely the X-ray transform~$\Xray$ and the multivariate Radon transform~$\Radon$, and explain exact as well as approximate reconstruction techniques in the case of three dimensions.
For the sake of simplicity, we restrict the discussion to smooth functions $f \in \Schwartz(\R^n)$ and skip some technical proofs.

\section{Multivariate integral transforms}

To define a multivariate version of the Radon transform for functions on $\R^n$, let us first recall that the bivariate Radon transform integrates functions over straight lines in the plane, which can either be seen as affine subspaces of dimension $1$ or as affine subspaces of codimension $1$.
The latter subspaces are called hyperplanes and in two dimensions the difference between integrals over hyperplanes and integrals over lines is only notational.
In higher dimensions, however, these two concepts lead to different integral transforms, which are introduced in this section.
To this end, let $\Sphere^{n-1}$ denote the unit sphere in $\R^n$, i.e., $\Sphere^{n-1} = \bigl\{x \in \R^n \bigm| \|x\|_{\R^n} = 1\bigr\}$.

\subsection*{The multivariate Radon transform}

The ($n$-dimensional) Radon transform $\Radon$ integrates a function $f$ on $\R^n$ over hyperplanes, i.e., affine subspaces of dimension $n-1$, which are parametrized as follows.

\begin{definition}[Hyperplane]
For any pair $(\theta,t) \in \Sphere^{n-1} \times \R$ of parameters, we define
\begin{equation*}
H_{\theta,t} = \bigl\{x \in \R^n \bigm| x^T \theta = t\bigr\} \subset \R^n
\end{equation*}
to be the unique {\em hyperplane} that is perpendicular to $\theta$ and has (signed) distance $t$ to the origin.
\end{definition}

Note that any hyperplane in $\R^n$ can be characterized as an $H_{\theta,t}$ for suitable $(\theta,t) \in \Sphere^{n-1} \times \R$ and we have
\begin{equation}
H_{-\theta,-t} = H_{\theta,t}.
\label{eq:hyperplane_symmetry}
\end{equation}
For the sake of brevity, we define the unit cylinder $Z^n = \Sphere^{n-1} \times \R$.

\begin{definition}[Radon transform]
Let $f \in \Schwartz(\R^n)$.
Then, the {\em Radon transform} $\Radon f$ of $f$ is defined as
\begin{equation*}
\Radon f(\theta,t) = \int_{H_{\theta,t}} f(x) \: \d x
\quad \mbox{ for } (\theta,t) \in Z^n.
\end{equation*}
\end{definition}

Note that $\Radon$ defines a continuous linear operator from $\Schwartz(\R^n)$ into the space $\Cont_b(Z^n)$ of continuous and uniformly bounded functions on $Z^n$.
Moreover, for any $f \in \Schwartz(\R^n)$ we have
\begin{equation*}
f \geq 0
\quad \implies \quad
\Radon f \geq 0
\end{equation*}
and, due to~\eqref{eq:hyperplane_symmetry},
\begin{equation}
\Radon f(-\theta,-t) = \Radon f(\theta,t).
\label{eq:Radon_eveness_n}
\end{equation}

\bigbreak

For fixed direction $\theta \in \Sphere^{n-1}$ we define the {\em Radon projection} $\Radon_\theta f$ of $f \in \Schwartz(\R^n)$ as
\begin{equation*}
\Radon_\theta f(t) = \Radon f(\theta,t)
\quad \mbox{ for } t \in \R.
\end{equation*}
Then, $\Radon_\theta$ defines a continuous linear operator from $\Schwartz(\R^n)$ into $\Schwartz(\R)$ and we have the following relation between the derivative of $\Radon_\theta f$ and the directional derivative of $f$.

\begin{proposition}
\label{prop:Radon_derivative}
Let $f \in \Schwartz(\R^n)$ and $\theta \in \Sphere^{n-1}$.
Then, we have
\begin{equation*}
\frac{\d}{\d t} \Radon_\theta f(t) = \Radon_\theta (\nabla_\theta f)(t)
\quad \forall \, t \in \R,
\end{equation*}
where $\nabla_\theta f$ is the directional derivative of $f$ in direction $\theta$.
\end{proposition}

\begin{proof}
By the definition of $\Radon_\theta f$, we have
\begin{equation*}
\Radon_\theta f(t) = \int_{H_{\theta,t}} f(x) \: \d x = \int_{\theta^\perp} f(y + t\theta) \: \d y,
\end{equation*}
since the hyperplane $H_{\theta,t}$ can be rewritten as
\begin{equation*}
H_{\theta,t} = \bigl\{x \in \R^n \bigm| x^T \theta = t\bigr\} = \bigl\{y + t\theta \bigm| y \in \theta^\perp\bigr\}.
\end{equation*}
Thus, differentiation with respect to $t$ yields
\begin{align*}
\frac{\d}{\d t} \Radon_\theta f(t) = \int_{\theta^\perp} \frac{\d}{\d t} f(y + t\theta) \: \d t = \int_{\theta^\perp} \nabla f(y + t\theta)^T \theta \: \d t = \Radon_\theta(\nabla_\theta f)(t),
\end{align*}
where the order of differentiation and integration can be changed because $f \in \Schwartz(\R^n)$.
\end{proof}

We now prove the $n$-dimensional version of the Fourier slice theorem, Theorem~\ref{theo:Fourier_slice}.
To this end, recall that the Fourier transform $\Fourier f$ of $f \equiv f(x) \in \L^1(\R^n)$ is given by
\begin{equation*}
\Fourier f(\omega) = \int_{\R^n} f(x) \, \e^{-\i x^T \omega} \: \d x
\quad \mbox{ for } \omega \in \R^n. 
\end{equation*}
For a function $h \equiv h(\theta,t)$ on $Z^n$ satisfying $h(\theta,\cdot) \in \L^1(\R)$ for all $\theta \in \Sphere^{n-1}$ we define its Fourier transform $\Fourier h$ as the univariate Fourier transform acting on the second variable $t$, i.e.,
\begin{equation*}
\Fourier h(\theta,s) = \int_\R h(\theta,t) \, \e^{-\i s t} \: \d t
\quad \mbox{ for } (\theta,s) \in Z^n.
\end{equation*}

With this, the $n$-dimensional Fourier slice theorem reads as follows.

\begin{theorem}[Fourier slice theorem]
\label{theo:Fourier_slice_n}
For any $f \in \Schwartz(\R^n)$ we have
\begin{equation*}
\Fourier(\Radon f)(\theta,s) = \Fourier f(s\theta)
\quad \forall \, (\theta,s) \in Z^n.
\end{equation*}
\end{theorem}

\begin{proof}
For $(\theta,s) \in Z^n$, the definition of the $n$-dimensional Fourier transform yields
\begin{align*}
\Fourier f(s\theta) & = \int_{\R^n} f(x) \, \e^{-\i s x^T \theta} \: \d x = \int_\R \int_{x^T \theta = t} f(x) \, \e^{-\i s x^T \theta} \: \d x \, \d t \\
& = \int_\R \left(\int_{x^T \theta = t} f(x) \: \d x\right) \e^{-\i s t} \: \d t = \Fourier(\Radon f)(\theta,s)
\end{align*}
by Fubini's theorem and the definition of the $n$-dimensional Radon transform.
\end{proof}

A direct consequence of the Fourier slice theorem is the injectivity of $\Radon$.

\begin{corollary}[Injectivity of the Radon transform]
For $f \in \Schwartz(\R^n)$ we have
\begin{equation*}
\Radon f = 0
\quad \implies \quad
f = 0,
\end{equation*}
i.e., the Radon transform $\Radon$ is injective on $\Schwartz(\R^n)$.
\end{corollary}

The inversion of $\Radon$ involves a rescaled version of the $\L^2$-adjoint of $\Radon$, which is given by the $n$-dimensional back projection $\Back$.

\begin{definition}[Back projection]
Let $g \in \Cont_b(Z^n)$.
Then, the {\em back projection} $\Back g$ of $g$ is defined as
\begin{equation*}
\Back g(x) = \frac{1}{(2\pi)^{n-1}} \int_{\Sphere^{n-1}} g(\theta,x^T \theta) \: \d \theta
\quad \mbox{ for } x \in \R^n.
\end{equation*}
\end{definition}

Based on the Fourier slice Theorem~\ref{theo:Fourier_slice_n} we are now prepared to derive the n-dimensional version of the filtered back projection (FBP) formula~\eqref{eq:FBP_formula} for inverting $\Radon$.

\begin{theorem}[FBP formula]
For $f \in \Schwartz(\R^n)$ the {\em filtered back projection formula}
\begin{equation}
f(x) = \frac{1}{2} \, \Back\big(\Fourier^{-1}[|S|^{n-1} \Fourier(\Radon f)(\theta,S)]\big)(x)
\label{eq:FBP_formula_n}
\end{equation}
holds for all $x \in \R^n$.
\end{theorem}

\begin{proof}
Let $x \in \R^n$ be fixed.
Applying the $n$-dimensional Fourier inversion formula to $f$ yields the identity
\begin{equation*}
f(x) = \Fourier^{-1}(\Fourier f)(x) = \frac{1}{(2\pi)^n} \int_{\R^n} \Fourier f(\omega) \, \e^{\i x^T \omega} \: \d \omega = \frac{1}{(2\pi)^n} \int_{\Sphere^{n-1}} \int_0^\infty \Fourier f(s\theta) \, \e^{\i s x^T \theta} \, s^{n-1} \: \d s \, \d \theta.
\end{equation*}
With the eveness condition~\eqref{eq:Radon_eveness_n} and the Fourier slice Theorem~\ref{theo:Fourier_slice_n} follows that
\begin{align*}
f(x) & = \frac{1}{2(2\pi)^n} \int_{\Sphere^{n-1}} \int_\R \Fourier (\Radon f)(\theta,S) \, \e^{\i S x^T \theta} \, |S|^{n-1} \: \d S \, \d \theta \\
& = \frac{1}{2(2\pi)^{n-1}} \int_{\Sphere^{n-1}} \Fourier^{-1}[|S|^{n-1} \Fourier(\Radon f)(\theta,S)](\theta,x^T \theta) \: \d \theta \\
& = \frac{1}{2} \, \Back\big(\Fourier^{-1}[|S|^{n-1} \Fourier(\Radon f)(\theta,S)]\big)(x)
\end{align*}
due to the definition of the back projection.
\end{proof}

As in the two-dimensional case, the reconstruction problem from Radon data
\begin{equation*}
\bigl\{\Radon f(\theta,t) \bigm| (\theta,t) \in Z^n\bigr\}
\end{equation*}
is ill-posed.
The degree of ill-posedness can again be determined by studying the smoothing effect of $\Radon$ in the Sobolev scale.
To this end, recall that the Sobolev space $\H^\alpha(\R^n)$ of fractional order $\alpha \in \R$ is given by
\begin{equation*}
\H^\alpha(\R^n) = \bigl\{f \in \Schwartz'(\R^n) \bigm| \|f\|_{\H^\alpha(\R^n)} < \infty\bigr\},
\end{equation*}
where
\begin{equation*}
\|f\|_{\H^\alpha(\R^n)}^2 = \int_{\R^n} \bigl(1 + \|x\|_{\R^n}^2\bigr)^\alpha \, |\Fourier f(x)|^2 \: \d x.
\end{equation*}
Furthermore, for an open subset $\Omega \subset \R^n$ the space $\H^\alpha_0(\Omega)$ consists of those Sobolev functions whose support is contained in $\overline{\Omega}$, i.e.,
\begin{equation*}
\H^\alpha_0(\Omega) = \bigl\{f \in \H^\alpha(\R^n) \bigm| \supp(f) \subset \overline{\Omega}\bigr\}.
\end{equation*}

For functions on $Z^n$ we define the Sobolev space $\H^\alpha(Z^n)$, for $\alpha \in \R$, as the space of all functions $g \equiv g(\theta,t)$ with $g(\theta,\cdot) \in \H^\alpha(\R)$ for almost all $\theta \in \Sphere^{n-1}$ and
\begin{equation*}
\|g\|_{\H^\alpha(Z^n)} = \left(\int_{\Sphere^{n-1}} \int_\R (1+S^2)^\alpha \, |\Fourier g(\theta,S)|^2 \: \d S \, \d \theta\right)^{\nicefrac{1}{2}} < \infty.
\end{equation*}
Then, one can prove that for $f \in \L^1(\R^n) \cap \H^\alpha_0(\Omega)$ with $\alpha \in \R$ and an open, bounded set $\Omega \subset \R^n$ we have $\Radon f \in \H^{\alpha+\nicefrac{(n-1)}{2}}(Z^n)$ and the Sobolev estimate
\begin{equation*}
c_{\alpha,n} \, \|f\|_{\H^\alpha(\R^n)} \leq \|\Radon f\|_{\H^{\alpha+\nicefrac{(n-1)}{2}}(Z^n)} \leq C_{\alpha,n} \, \|f\|_{\H^\alpha(\R^n)}
\end{equation*}
holds with positive constants $c_{\alpha,n}, C_{\alpha,n} > 0$ depending only on smoothness $\alpha$ and dimension $n$.
Consequently, the $n$-dimensional Radon reconstruction problem is ill-posed of order $\frac{n-1}{2}$.

\subsection*{The X-ray transform}

The ($n$-dimensional) X-ray transform $\Xray$ integrates a function $f$ on $\R^n$ over straight lines, i.e., affine subspaces of dimension $1$, which are parametrized as follows.

\begin{definition}[Straight line]
For any pair $(\theta,x) \in \Sphere^{n-1} \times \R^n$ of parameters, we define
\begin{equation*}
\Ell_{\theta,x} = \bigl\{x + t \theta \bigm| t \in \R\bigr\} \subset \R^n
\end{equation*}
to be the unique {\em straight line} in $\R^n$ with direction $\theta$ that passes through the reference point $x$.
\end{definition}

Note that any straight line in $\R^n$ can be characterized as an $\Ell_{\theta,x}$ for suitable $(\theta,x) \in \Sphere^{n-1} \times \R^n$ and we have
\begin{equation*}
\Ell_{\theta,x+y} = \Ell_{\theta,x}
\quad \forall \, y \in \langle\theta\rangle = \bigl\{t\theta \bigm| t \in \R\bigr\}.
\end{equation*}
Thus, it suffices to consider the tangent bundle $T^n$ of $\Sphere^{n-1}$ given by
\begin{equation*}
T^n = \bigl\{(\theta,x) \in \R^{2n} \bigm| \theta \in \Sphere^{n-1}, ~ x \in \theta^\perp\bigr\}.
\end{equation*}

\begin{definition}[X-ray transform]
Let $f \in \Schwartz(\R^n)$.
Then, the {\em X-ray transform} $\Xray f$ of $f$ is defined as
\begin{equation*}
\Xray f(\theta,x) = \int_{\Ell_{\theta,x}} f(y) \: \d y = \int_\R f(x + t \theta) \: \d t
\quad \mbox{ for } (\theta,x) \in T^n.
\end{equation*}
\end{definition}

Note that $\Xray$ defines a continuous linear operator from $\Schwartz(\R^n)$ into the space $\Cont_b(T^n)$ of continuous and uniformly bounded functions on $T^n$.
Moreover, for any $f \in \Schwartz(\R^n)$ we have
\begin{equation*}
f \geq 0
\quad \implies \quad
\Xray f \geq 0
\end{equation*}
and
\begin{equation*}
\Xray f(-\theta,x) = \Xray f(\theta,x).
\end{equation*}

\bigbreak

We now prove that the Radon transform $\Radon f$ of $f$ can be expressed as an integral over $\Xray f$.

\begin{proposition}
Let $f \in \Schwartz(\R^n)$ and $(\omega,s) \in Z^n$.
Then, for any $\theta \in \Sphere^{n-1}$ with $\omega^T \theta = 0$ we have
\begin{equation*}
\Radon f(\omega,s) = \int_{\substack{x \in \theta^\perp \\ x^T \omega = s}} \Xray f(\theta,x) \: \d x.
\end{equation*}
\end{proposition}

\begin{proof}
As $\omega,\theta \in \Sphere^{n-1}$ with $\omega^T \theta = 0$, there is an orthogonal matrix $A \in \R^{n \times n}$ with
\begin{equation*}
A \omega = e_1
\quad \mbox{ and } \quad
A \theta = e_2.
\end{equation*}
This implies that
\begin{equation*}
\bigl\{x \in \R^n \bigm| x \in \theta^\perp, ~ x^T \omega = s\bigr\} = \bigl\{A^T(s,0,y) \bigm| y \in \R^{n-2}\bigr\}
\end{equation*}
and, consequently,
\begin{align*}
\int_{\substack{x \in \theta^\perp \\ x^T \omega = s}} \Xray f(\theta,x) \: \d x & = \int_{\R^{n-2}} \Xray f\bigl(\theta,A^T(s,0,y)\bigr) \: \d y = \int_{\R^{n-2}} \int_\R f\bigl(A^T((s,0,y) + t \underbrace{A\theta}_{=e_2})\bigr) \: \d t \, \d y \\
& = \int_{\R^{n-1}} f\bigl(A^T(s,z)\bigr) \: \d z = \int_{\omega^\perp} f(y + s \omega) \: \d y = \Radon f(\omega,s),
\end{align*}
as stated.
\end{proof}

For a function $g \equiv g(\theta,x)$ on $T^n$ satisfying $g(\theta,\cdot) \in \L^1(\theta^\perp)$ for all $\theta \in \Sphere^{n-1}$ we define its Fourier transform $\Fourier g$ as the $(n-1)$-dimensional Fourier transform on $\theta^\perp$ acting on the second variable $x$, i.e.,
\begin{equation*}
\Fourier g(\theta,y) = \int_{\theta^\perp} g(\theta,x) \, \e^{-\i x^T y} \: \d x
\quad \mbox{ for } (\theta,y) \in T^n.
\end{equation*}

With this, the Fourier slice theorem for the X-ray transform $\Xray$ reads as follows.

\begin{theorem}[X-ray Fourier slice theorem]
\label{theo:Fourier_slice_Xray}
For any $f \in \Schwartz(\R^n)$ we have
\begin{equation*}
\Fourier(\Xray f)(\theta,y) = \Fourier f(y)
\quad \forall \, (\theta,y) \in T^n.
\end{equation*}
\end{theorem}

\begin{proof}
For $(\theta,y) \in T^n$, the definition of the $n$-dimensional Fourier transform yields
\begin{equation*}
\Fourier f(y) = \int_{\R^n} f(x) \, \e^{-\i x^T y} \: \d x = \int_{\theta^\perp} \left(\int_\R f(x + t \theta) \: \d t\right) \e^{-\i x^T y} \: \d x = \Fourier(\Xray f)(\theta,y)
\end{equation*}
by Fubini's theorem and the definition of the X-ray transform.
\end{proof}

A direct consequence of the X-ray Fourier slice theorem is the injectivity of $\Xray$.

\begin{corollary}[Injectivity of the X-ray transform]
For $f \in \Schwartz(\R^n)$ we have
\begin{equation*}
\Xray f = 0
\quad \implies \quad
f = 0,
\end{equation*}
i.e., the X-ray transform $\Xray$ is injective on $\Schwartz(\R^n)$.
\end{corollary}

The inversion of $\Xray$ involves a rescaled version of the $\L^2$-adjoint of $\Xray$, which is given by the X-ray back projection $\Xback$.

\begin{definition}[X-ray back projection]
Let $g \in \Cont_b(T^n)$.
Then, the {\em X-ray back projection} $\Xback g$ of $g$ is defined as
\begin{equation*}
\Xback g(x) = \frac{1}{|\Sphere^{n-2}|} \int_{\Sphere^{n-1}} g(\theta,E_\theta x) \: \d \theta
\quad \mbox{ for } x \in \R^n,
\end{equation*}
where $E_\theta$ is the orthogonal projection onto $\theta^\perp$, i.e.,
\begin{equation*}
E_\theta x = x - (x^T \theta) \, \theta
\quad \mbox{ for } x \in \R^n.
\end{equation*}
\end{definition}

Based on the X-ray Fourier slice Theorem~\ref{theo:Fourier_slice_Xray} and the general integral formula
\begin{equation}
\int_{\R^n} h(x) \: \d x = \frac{1}{|\Sphere^{n-2}|} \int_{\Sphere^{n-1}} \int_{\theta^\perp} \|y\|_{\R^n} \, h(y) \: \d y \, \d \theta
\label{eq:integral_formula}
\end{equation}
we are now prepared to derive the X-ray version of the FBP formula~\eqref{eq:FBP_formula_n} for inverting $\Xray$.

\begin{theorem}[X-ray FBP formula]
\label{theo:filtered_back_projection_Xray}
For $f \in \Schwartz(\R^n)$ the {\em X-ray filtered back projection formula}
\begin{equation}
f(x) = \frac{1}{2\pi} \, \Xback\big(\Fourier^{-1}[\|y\|_{\R^n} \Fourier(\Xray f)(\theta,y)]\big)(x)
\label{eq:FBP_formula_Xray}
\end{equation}
holds for all $x \in \R^n$.
\end{theorem}

\begin{proof}
Let $x \in \R^n$ be fixed.
Applying the $n$-dimensional Fourier inversion formula to $f$ yields the identity
\begin{equation*}
f(x) = \Fourier^{-1}(\Fourier f)(x) = \frac{1}{(2\pi)^n} \int_{\R^n} \Fourier f(\omega) \, \e^{\i x^T \omega} \: \d \omega.
\end{equation*}
With the integral formula~\eqref{eq:integral_formula} and the X-ray Fourier slice Theorem~\ref{theo:Fourier_slice_Xray} follows that
\begin{align*}
f(x) & = \frac{1}{(2\pi)^n \, |\Sphere^{n-2}|} \int_{\Sphere^{n-1}} \int_{\theta^\perp} \|y\|_{\R^n} \, \Fourier f(y) \, \e^{\i x^T y} \: \d y \, \d \theta \\
& = \frac{1}{(2\pi)^n \, |\Sphere^{n-2}|} \int_{\Sphere^{n-1}} \int_{\theta^\perp} \|y\|_{\R^n} \, \Fourier(\Xray f)(\theta,y) \, \e^{\i x^T y} \: \d y \, \d \theta \\
& = \frac{1}{2\pi \, |\Sphere^{n-2}|} \int_{\Sphere^{n-1}} \Fourier^{-1}[\|y\|_{\R^n} \Fourier(\Xray f)(\theta,y)](\theta,E_\theta x) \: \d \theta \\
& = \frac{1}{2\pi} \, \Xback\big(\Fourier^{-1}[\|y\|_{\R^n} \Fourier(\Xray f)(\theta,y)]\big)(x)
\end{align*}
due to the definition of the X-ray back projection.
\end{proof}

We close this section with studying the degree of ill-posedness of the reconstruction problem from X-ray data
\begin{equation*}
\bigl\{\Xray f(\theta,x) \bigm| (\theta,x) \in T^n\bigr\}.
\end{equation*}
To this end, we define the Sobolev space $\H^\alpha(T^n)$ of fractional order $\alpha \in \R$ on $T^n$ as the space of all functions $g \equiv g(\theta,x)$ with $g(\theta,\cdot) \in \H^\alpha(\theta^\perp)$ for almost all $\theta \in \Sphere^{n-1}$ and
\begin{equation*}
\|g\|_{\H^\alpha(T^n)} = \left(\int_{\Sphere^{n-1}} \int_{\theta^\perp} \bigl(1+\|y\|_{\R^n}^2\bigr)^\alpha \, |\Fourier g(\theta,y)|^{\nicefrac{1}{2}} \: \d y \, \d \theta\right)^{\nicefrac{1}{2}} < \infty.
\end{equation*}
Then, one can prove that for $f \in \L^1(\R^n) \cap \H^\alpha_0(\Omega)$ with $\alpha \in \R$ and an open, bounded set $\Omega \subset \R^n$ we have $\Xray f \in \H^{\alpha+\nicefrac{1}{2}}(T^n)$ and the Sobolev estimate
\begin{equation*}
c_{\alpha,n} \, \|f\|_{\H^\alpha(\R^n)} \leq \|\Xray f\|_{\H^{\alpha+\nicefrac{1}{2}}(T^n)} \leq C_{\alpha,n} \, \|f\|_{\H^\alpha(\R^n)}
\end{equation*}
holds with positive constants $c_{\alpha,n}, C_{\alpha,n} > 0$ depending only on smoothness $\alpha$ and dimension $n$.
Consequently, the X-ray reconstruction problem is ill-posed of order $\frac{1}{2}$, independent of~$n$.

\section{Analytic inversion of the X-ray transform on \texorpdfstring{$\R^3$}{R\^{}3}}

We now study the X-ray transform~$\Xray$ in three dimensions, which describes the three-dimensional X-ray model and, thus, is the most important case for applications.
In this case, the filtered back projection formula~\eqref{eq:FBP_formula_Xray} for inverting $\Xray$ reads
\begin{equation*}
f(x) = \frac{1}{2\pi} \, \Xback\big(\Fourier^{-1}[\|y\|_{\R^3} \Fourier(\Xray f)(\theta,y)]\big)(x) = \frac{1}{(2\pi)^4} \int_{\Sphere^2} \int_{\theta^\perp} \|y\|_{\R^3} \, \Fourier(\Xray f)(\theta,y) \, \e^{\i x^T y} \: \d y \, \d \theta.
\end{equation*}
Thus, one needs $\Xray f(\theta,y)$ for all $\theta \in \Sphere^2$ and $y \in \theta^\perp$ in order to find the function $f$.
In practice, however, $\Xray f$ is usually not known on all of $\Sphere^2$. 
This shows that the FBP formula for inverting~$\Xray$ in Theorem~\ref{theo:filtered_back_projection_Xray} is not as useful as the FBP formula for inverting $\Radon$ in Theorem~\ref{theo:filtered_back_projection}.
Instead, we now present two different inversion formulas that are more suitable for applications.

\subsection*{Orlov's formula}

In the mid-1970s, Orlov derived an inversion formula for the X-ray transform $\Xray$ using parallel X-ray projections only for a subset $\Sphere_0^2 \subset \Sphere^2$ of directions.
Here, $\Sphere_0^2$ is the spherical zone around the equator given by
\begin{equation}
\Sphere_0^2 = \bigl\{\theta(\varphi,\vartheta) \bigm| \vartheta_-(\varphi) \leq \vartheta \leq \vartheta_+(\varphi), ~ 0 \leq \varphi < 2\pi\bigr\}
\label{eq:spherical_zone}
\end{equation}
with spherical coordinates
\begin{equation*}
\theta(\varphi,\vartheta) = \begin{pmatrix}
\cos(\varphi) \cos(\vartheta) \\
\sin(\varphi) \cos(\vartheta) \\
\sin(\vartheta)
\end{pmatrix}
\quad \mbox{ for } 0 \leq \varphi < 2\pi, ~ |\vartheta| \leq \frac{\pi}{2}
\end{equation*}
and functions $\vartheta_\pm: [0,2\pi) \to \R$ such that
\begin{equation*}
-\frac{\pi}{2} < \vartheta_-(\varphi) < 0 < \vartheta_+(\varphi) < \frac{\pi}{2}
\quad \forall \, 0 \leq \varphi < 2\pi.
\end{equation*}
For example, if $\vartheta_\pm = \pm \vartheta_0$ are constant functions with $\vartheta_0 \in (0,1)$, then $\Sphere_0^2$ is the spherical zone between the horizontal planes $x_3 = \pm \sin(\vartheta_0)$, see Figure~\ref{fig:spherical_zone}.

\begin{figure}[t]
\centering
\includegraphics[height=5.75cm,keepaspectratio]{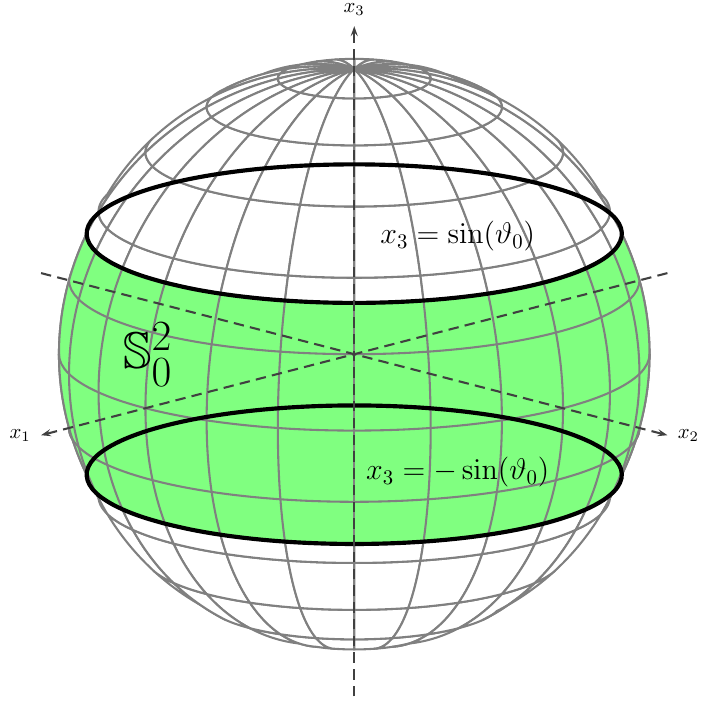}
\caption{Illustration of the spherical zone $\Sphere_0^2$ in~\eqref{eq:spherical_zone} for constant functions $\vartheta_\pm = \pm \vartheta_0$.}
\label{fig:spherical_zone}
\end{figure}

\medskip

To state Orlov's formula, let $\ell(x,y)$ denotes the length of the intersection of $\Sphere_0^2$ with the subspace spanned by the vectors $x,y \in \R^3$.
Due to the assumptions on $\vartheta_\pm$, we have $\ell(x,y) > 0$ if $x$ and $y$ are linearly independent.
In this case, $\ell(x,y)$ is the length of the intersection of $\Sphere_0^2$ with the plane spanned by $0,x,y \in \R^3$.

\begin{theorem}[Orlov's formula]
Let $f \in \Schwartz(\R^3)$.
Then, {\em Orlov's inversion formula}
\begin{equation}
f(x) = \Delta \int_{\Sphere_0^2} h(\theta,E_\theta x) \: \d \theta
\label{eq:Orlov_formula_part1}
\end{equation}
holds for all $x \in \R^3$, where $\Delta$ is the Laplace operator acting on $x$, i.e., $\Delta = \frac{\partial^2}{\partial x_1^2} + \frac{\partial^2}{\partial x_2^2} + \frac{\partial^2}{\partial x_3^2}$, and $h$ is obtained from $\Xray f$ by
\begin{equation}
h(\theta,x) = \frac{1}{4\pi^2} \int_{\theta^\perp} \frac{\Xray f(\theta,x - y)}{\ell(\theta,y) \, \|y\|_{\R^3}} \: \d y
\quad \mbox{ for } (\theta,x) \in T^3.
\label{eq:Orlov_formula_part2}
\end{equation}
\end{theorem}

\begin{proof}
See, for example,~\cite[Theorem 2.16]{Natterer2001a}.
\end{proof}

Note that the integral in~\eqref{eq:Orlov_formula_part1} is a back projection over $\Sphere_0^2$ and the integral in~\eqref{eq:Orlov_formula_part2} is a convolution over $\theta^\perp$.
Thus, a standard discretization of Orlov's formula leads to an algorithm of filtered back projection type for the reconstruction from X-ray data in {\em parallel beam geometry}.

\subsection*{Grangeat's formula}

In contrast to Orlov's inversion formula~\eqref{eq:Orlov_formula_part1}, Grangeat's formula does not require parallel projections for all directions in the spherical zone $\Sphere_0^2$.
Instead, it assumes that an X-ray source emits a cone of X-ray beams and travels along a curve $\Gamma \subset \R^3$ around the object together with a two-dimensional detector array.
More precisely, we assume that for every point on the curve~$\Gamma$ we are given the line integrals of a function $f$ along all lines starting at that point and travelling through $\supp(f)$.
This motivates us to define the following variant of the X-ray transform $\Xray$.

\begin{definition}[Cone beam transform]
Let $f \in \Schwartz(\R^3)$.
Then, the {\em cone beam transform} $\Fan f$ of $f$ is defined as
\begin{equation*}
\Fan f(a,\theta) = \int_0^\infty f(a + t \theta) \: \d t
\quad \mbox{ for } (a,\theta) \in \R^3 \times \Sphere^2.
\end{equation*}
\end{definition}

Note that $\Fan$ defines a continuous linear operator from $\Schwartz(\R^3)$ into $\Cont_b(\R^3 \times \Sphere^2)$ and we have
\begin{equation*}
\Fan f(a,\theta) + \Fan f(a,-\theta) = \Xray f(\theta,a).
\end{equation*}
We think of $a$ as the source of an X-ray with direction $\theta$.
Moreover, we extend $\Fan f$ to $\R^3 \times \R^3\setminus\{0\}$ by
\begin{equation*}
\Fan f(a,x) = \int_0^\infty f(a + t x) \: \d t = \|x\|_{\R^3}^{-1} \, \Fan f\bigl(a,\|x\|_{\R^3}^{-1} \, x\bigr)
\quad \mbox{ for } (a,x) \in \R^3 \times \R^3\setminus\{0\},
\end{equation*}
which makes $\Fan f$ a homogeneous function of degree $-1$ in the second argument.

\medskip

Grangeat's method for inverting $\Fan$ requires the curve $\Gamma$ to satisfy the following condition of Tuy-Kirillov.
A practically important example for such a curve is shown in Figure~\ref{fig:Tuy_example}.

\begin{definition}[Tuy's condition]
Let $\Omega \subset \R^3$ be a subset of $\R^3$ and let $\Gamma \subset \R^3$ be a curve with parametrization
\begin{equation*}
\gamma: [0,1] \to \R^3.
\end{equation*}
Then, $\Gamma$ is said to satisfy {\em Tuy's condition} with respect to $\Omega$ if every hyperplane $H \subset \R^3$ that intersects $\Omega$ also intersects $\Gamma$, i.e., for every $(\theta,s) \in Z^3$ with $H_{\theta,s} \cap \Omega \neq \emptyset$ there is a $t_{\theta,s} \in [0,1]$ such that $\gamma(t_{\theta,s}) \in \H_{\theta,s}$, that is
\begin{equation*}
\gamma(t_{\theta,s})^T \theta = s.
\end{equation*}
\end{definition}

\begin{figure}[t]
\centering
\includegraphics[height=5.75cm,keepaspectratio]{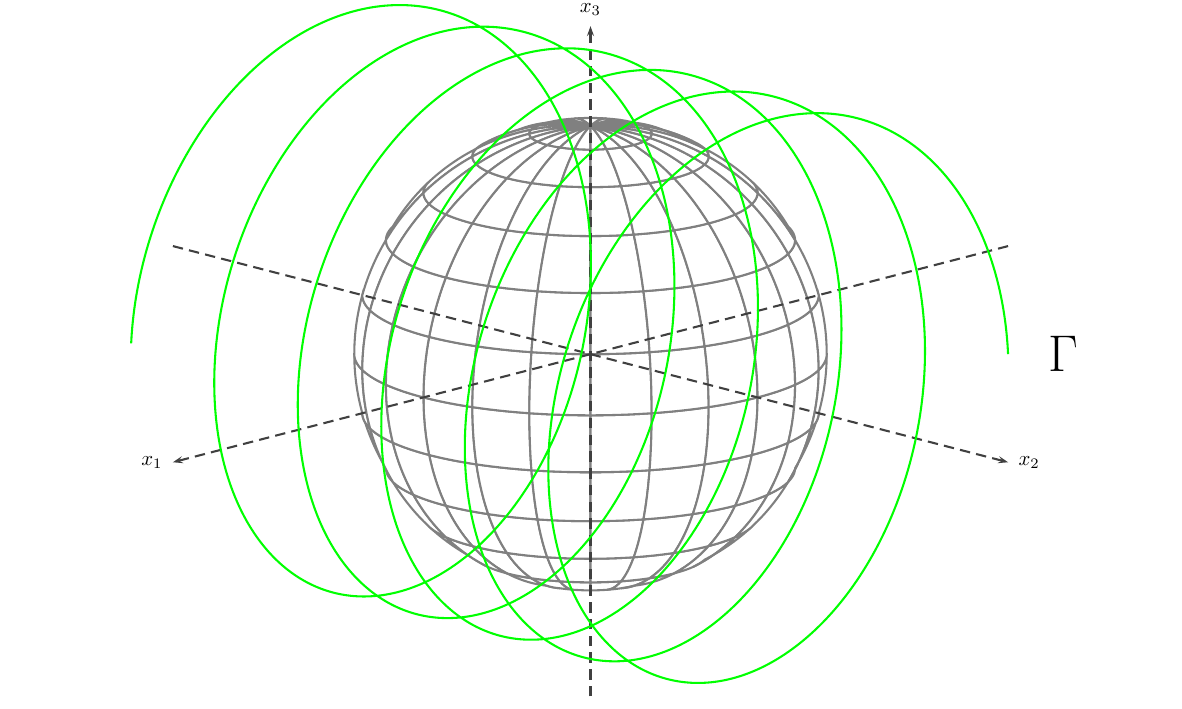}
\caption{Example for a curve $\Gamma \subset \R^3$ satisfying Tuy's condition with respect to $B_1(0)$.}
\label{fig:Tuy_example}
\end{figure}

The gist of Grangeat's method is the following relation between the cone beam transform $\Fan$ and the three-dimensional Radon transform $\Radon$, which is known as {\em Grangeat's formula}.

\begin{theorem}[Grangeat's formula]
\label{theo:Grangeat_formula}
Let $f \in \Schwartz(\R^3)$.
Then, for $(a,\theta) \in \R^3 \times \Sphere^2$ we have
\begin{equation*}
\frac{\partial}{\partial t} \Radon f(\theta,a^T \theta) = \int_{\Sphere^2 \cap \theta^\perp} \nabla_\theta \Fan f(a,\omega) \: \d \omega,
\end{equation*}
where $\nabla_\theta \Fan f$ is the directional derivative of $\Fan f$ in direction $\theta$, acting on the second variable.
\end{theorem}

\begin{proof}
According to Proposition~\ref{prop:Radon_derivative} we have
\begin{equation*}
\frac{\partial}{\partial t} \Radon f(\theta,t) = \Radon (\nabla_\theta f)(\theta,t)
\quad \forall \, t \in \R.
\end{equation*}
Thus, with $t = a^T \theta$ follows that
\begin{equation*}
\frac{\partial}{\partial t} \Radon f(\theta,a^T \theta) = \int_{H_{\theta,a^T \theta}} \nabla_\theta f(x) \: \d x = \int_{\theta^\perp} \nabla_\theta f(a + y) \: \d y = \int_{\Sphere^2 \cap \theta^\perp} \int_0^\infty \nabla_\theta f(a+s\omega) \, s \: \d s \, \d \omega.
\end{equation*}
Now observe that
\begin{equation*}
\nabla_\theta f(a+s\omega) = \sum_{j=1}^3 \theta_j \, \frac{\partial}{\partial x_j} f(a+s\omega)
\end{equation*}
and, by interchanging differentiation and integration,
\begin{equation*}
\nabla_\theta \Fan f(a,\omega) = \sum_{j=1}^3 \theta_j \int_0^\infty t \, \frac{\partial}{\partial x_j} f(a+t\omega) \: \d t.
\end{equation*}
Consequently,
\begin{equation*}
\frac{\partial}{\partial t} \Radon f(\theta,a^T \theta) = \int_{\Sphere^2 \cap \theta^\perp} \sum_{j=1}^3 \theta_j \int_0^\infty s \, \frac{\partial}{\partial x_j} f(a+s\omega) \: \d s \, \d \omega = \int_{\Sphere^2 \cap \theta^\perp} \nabla_\theta \Fan f(a,\omega) \: \d \omega,
\end{equation*}
as stated.
\end{proof}

The other ingredient of Grangeat's method is the FBP formula~\eqref{eq:FBP_formula_n} for $n = 3$, i.e.,
\begin{equation*}
f(x) = \frac{1}{2} \, \Back\big(\Fourier^{-1}[|S|^2 \Fourier(\Radon f)(\theta,S)]\big)(x).
\end{equation*}
Using Proposition~\ref{prop:Fourier_properties}, this can be rewritten as
\begin{equation}
f(x) = -\frac{1}{2} \, \Back\biggl(\frac{\partial^2}{\partial t^2} \Radon f(\theta, t)\biggr)(x) = -\frac{1}{8\pi^2} \int_{\Sphere^2} \frac{\partial^2}{\partial t^2} \Radon f(\theta,x^T \theta) \: \d \theta.
\label{eq:FBP_formula_3}
\end{equation}

\begin{corollary}[Grangeat's method]
Let $f \in \Schwartz(\R^3)$ with support $\supp(f)$ and let $\Gamma \subset \R^3$ be a curve parametrized by $\gamma: [0,1] \to \R^3$ that satisfies Tuy's condition with respect to $\supp(f)$.
Then, the inversion formula
\begin{equation*}
f(x) = -\frac{1}{8\pi^2} \int_{\Sphere^2} \frac{\partial}{\partial t} \int_{\Sphere^2 \cap \theta^\perp} \nabla_\theta \Fan f(\gamma(t_{\theta,x}),\omega) \: \d \omega \, \d \theta
\end{equation*}
holds for all $x \in \R^3$, where $t_{\theta,x} \in [0,1]$ satisfies $x^T \theta = \gamma(t_{\theta,x})^T \theta$.
\end{corollary}

\begin{proof}
Since $\Gamma$ satisfies Tuy's condition with respect to $\supp(f)$, for any $x \in \supp(f)$ and $\theta \in \Sphere^2$ there is a $t_{\theta,x} \in [0,1]$ such that
\begin{equation*}
x^T \theta = \gamma(t_{\theta,x})^T \theta.
\end{equation*}
Now, using Grangeat's formula, Theorem~\ref{theo:Grangeat_formula}, with $a = \gamma(t_{\theta,x})$ yields
\begin{equation*}
\frac{\partial}{\partial t} \Radon f(\theta,x^T \theta) = \int_{\Sphere^2 \cap \theta^\perp} \nabla_\theta \Fan f(\gamma(t_{\theta,x}),\omega) \: \d \omega.
\end{equation*}
From this formula, $\frac{\partial^2}{\partial t^2} \Radon f(\theta,t)$ can be computed for $t = x^T \theta$ for each $x \in \R^3$ with $f(x) \neq 0$.
Using this in the three-dimensional FBP formula~\eqref{eq:FBP_formula_3} gives the result.
\end{proof}

Grangeat's method permits exact reconstruction of $f$ if each plane hitting $\supp(f)$ contains at least one source, i.e., point on the curve $\Gamma$.
For the reconstruction of $f$ at a point $x \in \supp(f)$ it uses only those values of $\Fan f(\gamma(t_{\theta,x}),\omega)$ for which $\omega$ is almost perpendicular to $\theta$.
Thus, it requires cone beam data only along X-rays that run in a small cone whose axis joins $x$ with an arbitrary source position on $\Gamma$. 
Consequently, a discretization of Grangeat's method leads to a reconstruction algorithm for X-ray data in so called {\em cone beam geometry}.
In practice, however, the source-detector pair often travels along a circle, which does not satisfy Tuy's condition.

\section{Approximate inversion of the X-ray transform on \texorpdfstring{$\R^3$}{R\^{}3}}

We close this chapter on three-dimensional reconstruction techniques by describing the so called {\em Feldkamp-Davis-Kress (FDK) algorithm}, which is presently the most widely used approximate reconstruction method for cone beam scanning with a circle as source curve, see Figure~\ref{fig:cone_beam_circle}.
Note that none of the exact inversion formulas we described before applies to this scanning geometry.
Therefore, we do not aim at exact inversion but rather derive the ingenious approximate formula which was given by Feldkamp, Davis and Kress in 1984.

\begin{figure}[t]
\centering
\includegraphics[height=5.75cm,keepaspectratio]{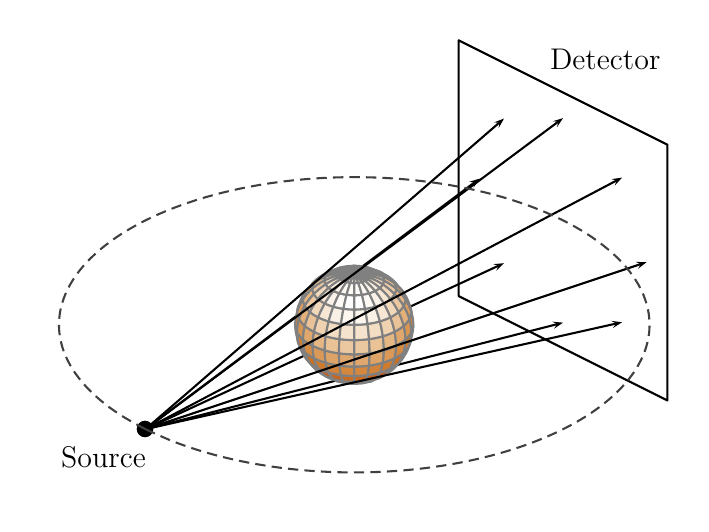}
\caption{Cone beam scanning geometry with source on a circle.}
\label{fig:cone_beam_circle}
\end{figure}

We start with fixing notations.
We assume that the X-ray source travels on a circle of radius $\rho > 0$ in the horizontal $x_1-x_2$-plane and specify its position by $\bfx_{\text{S}} = \rho\theta$ with $\theta \in \Sphere^1_0 = \Sphere^1 \times \{0\}$, i.e.,
\begin{equation*}
\bfx_{\text{S}} = \begin{pmatrix}
\rho\cos(\varphi) \\ \rho\sin(\varphi) \\ 0
\end{pmatrix}
\quad \mbox{ for } 0 \leq \varphi < 2\pi.
\end{equation*}
Moreover, we assume that the target function $f$ is supported in $B_r(0)$ with $0 < r < \rho$ and denote by $g(\theta,y)$ the line integral of $f$ along the line joining $\bfx_{\text{S}} = \rho\theta$ with $y \in \theta^\perp$, i.e.,
\begin{equation*}
g(\theta,y) = \Fan f\left(\rho\theta,\frac{y - \rho\theta}{\|y - \rho\theta\|_{\R^3}}\right).
\end{equation*}

The idea of the FDK algorithm is as follows.
Consider the plane $\Pi(x,\theta)$ through $\rho\theta$ and $x$ that intersects $\theta^\perp$ in a horizontal line.
The reduction of the X-ray cone coming from the source to this plane results in two-dimensional X-ray data in fan beam geometry with a straight detector line.
Thus, we use a linear version of the approximate fan beam reconstruction formula~\eqref{eq:FBP_approximate_fan} to compute the contribution of those beams to the reconstruction of $f$.
Finally, we integrate all those contributions over $\theta$, disregarding the fact that they come from different planes, which form a sheaf with vertex in $x$.
The result is the FDK reconstruction of $f$ at point $x \in B_r(0)$.

\subsection*{Linear fan beam geometry}

Before we come to the derivation of the FDK algorithm, we first explain the linear fan beam geometry in two dimensions and how the FBP formula~\eqref{eq:FBP_approximate_fan} can be adopted to this setting.

\begin{figure}[t]
\centering
\includegraphics[height=5.75cm,keepaspectratio]{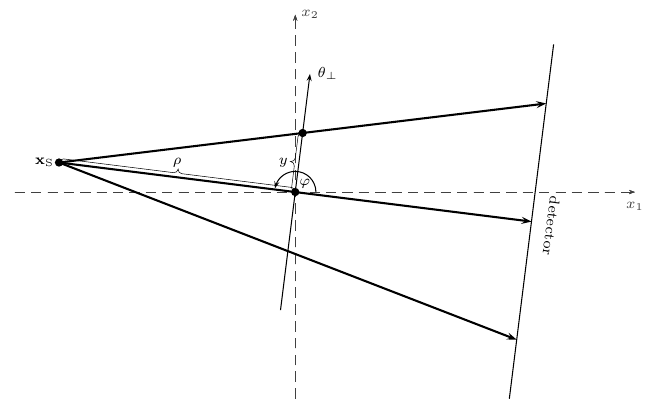}
\caption{Linear fan beam geometry.}
\label{fig:fan_beam_linear}
\end{figure}

In linear fan beam geometry, a source travels on a circle with radius $\rho > 0$ around the object under investigation and emits a fan of X-ray beams that are detected by a straight line detector.
The detector is located opposite the source and the detector positions are uniformly distributed on the detector line orthogonal to the central ray of the fan, see Figure~\ref{fig:fan_beam_linear}.
We express the source position $\bfx_{\text{S}} \in \R^2$ as
\begin{equation*}
\bfx_{\text{S}} = \begin{pmatrix}
\rho \cos(\varphi) \\ \rho \sin(\varphi)
\end{pmatrix}
\quad \mbox{ for } 0 \leq \varphi < 2\pi
\end{equation*}
and denote by $g(\varphi,y)$ the line integral value of $f$ along the line joining $\rho\theta$ and $y\theta_\perp$, where
\begin{equation*}
\theta \equiv \theta(\varphi) = \begin{pmatrix}
\cos(\varphi) \\ \sin(\varphi)
\end{pmatrix}
\quad \mbox{ and } \quad
\theta_\perp \equiv \theta_\perp(\varphi) = \begin{pmatrix}
\sin(\varphi) \\ -\cos(\varphi)
\end{pmatrix}.
\end{equation*}
With the two-dimensional Radon transform $\Radon$ the linear fan beam data $g(\varphi,y)$ can be expressed as
\begin{equation*}
g(\varphi,y) = \Radon f(t,\psi)
\end{equation*}
with
\begin{equation}
t = \frac{\rho y}{\sqrt{\rho^2 + y^2}}
\quad \mbox{ and } \quad
\psi = \varphi + \arctan\Bigl(\frac{y}{\rho}\Bigr)-\frac{\pi}{2}.
\label{eq:coordinates_t_vartheta}
\end{equation}

Recall that, for $\supp(f) \subseteq B_r(0)$ with $0 < r < \rho$, the approximate FBP formula~\eqref{eq:FBP_approximate_form} can be written as
\begin{equation*}
f_L(x) = \frac{1}{4\pi} \int_0^{2\pi} \int_{-\rho}^\rho (\Fourier^{-1} F_L)(x_1 \cos(\psi) + x_2 \sin(\psi)-t) \, \Radon f(t,\psi) \: \d t \, \d \psi,
\end{equation*}
with a low-pass filter $F_L$ of bandwidth $L > 0$.
Introducing the coordinates $(\varphi,y)$, we obtain
\begin{equation*}
f_L(x) = \frac{1}{4\pi} \int_0^{2\pi} \int_{-\rho}^\rho (\Fourier^{-1} F_L)(x_1 \cos(\psi) + x_2 \sin(\psi)-t) \, g(\varphi,y) \frac{\rho^3}{(\rho^2 + y^2)^{\nicefrac{3}{2}}} \: \d y \, \d \varphi,
\end{equation*}
where~\eqref{eq:coordinates_t_vartheta} has to be inserted for $t,\psi$.
After some manipulations similar to those in Section~\ref{sec:standard_fan_beam}, that are justified if $r \ll \rho$, we arrive at the approximate FBP formula for linear fan beam data
\begin{equation}
f_L(x) \approx \frac{1}{4\pi} \int_{\Sphere^1} \frac{\rho^2}{(\rho - x^T \theta)^2} \int_{-\rho}^\rho (\Fourier^{-1} F_L)\biggl(\frac{\rho x^T \theta_\perp}{\rho - x^T \theta} - y\biggr) \, g(\theta,y) \frac{\rho}{(\rho^2 + y^2)^{\nicefrac{1}{2}}} \: \d y \, \d \theta,
\label{eq:FBP_approximate_fan_linear}
\end{equation}
which can be implemented as the approximate FBP formula~\eqref{eq:FBP_approximate_fan} for standard fan beam data.

\subsection*{Feldkamp-Davis-Kress algorithm}

We are now prepared to continue with the derivation of the FDK algorithm.
To this end, we introduce a coordinate system for the hyperplane $\theta^\perp$.
For $\theta = (\cos(\varphi),\sin(\varphi),0)$ with $\varphi \in [0,2\pi)$ we set $\theta_\perp = (\sin(\varphi),-\cos(\varphi),0)$.
Then, $\theta_\perp, e_3$ form an orthonormal basis of~$\theta^\perp$ and the line through $\rho\theta$ and $x$ hits $\theta^\perp$ at $y = y_2 \theta_\perp + y_3 e_3$ with
\begin{equation}
y_2 = \frac{\rho}{\rho - x^T \theta} x^T \theta_\perp
\quad \mbox{ and } \quad
y_3 = \frac{\rho}{\rho - x^T \theta} x_3.
\label{eq:FDK_local_coordinates}
\end{equation}
The plane $\Pi(x,\theta)$ intersects $\theta^\perp$ in the horizontal line
\begin{equation*}
\Ell_{x,\theta} = \bigl\{t \theta_\perp + y_3 e_3 \bigm| t \in \R\bigr\}
\end{equation*}
and we define the point $y_3 e_3$ on $\Ell_{x,\theta}$ to be the origin in $\Pi(x,\theta)$.
Then, the coordinates $x^\prime$ of $x$ in $\Pi(x,\theta)$ are $x^\prime = x - y_3 e_3$ and the direction vector $\theta^\prime$ lying above (or beneath) $\theta$ in $\Pi(x,\theta)$ is given by
\begin{equation}
\theta^\prime = \frac{\rho\theta - y_3 e_3}{\rho^\prime}
\quad \mbox{ with } \quad
\rho^\prime = (\rho^2 + y_3^2)^{\nicefrac{1}{2}},
\label{eq:FDK_local_direction}
\end{equation}
where $\rho^\prime$ is the distance of $\rho\theta$ to the origin in $\Pi(x,\theta)$.
We now use the linear fan beam FBP formula~\eqref{eq:FBP_approximate_fan_linear} to compute the contribution of the direction $\theta^\prime$ to the approximate reconstruction of $f(x)$ as
\begin{equation*}
I(x,\theta) = \frac{{\rho^\prime}^2}{(\rho^\prime - {x^\prime}^T \theta^\prime)^2} \int_{-\rho}^\rho (\Fourier^{-1} F_L)\biggl(\frac{\rho^\prime {x^\prime}^T \theta_\perp}{\rho^\prime - {x^\prime}^T \theta^\prime} - y_2^\prime\biggr) \, g(\theta,y_2^\prime \theta_\perp + y_3 e_3) \frac{\rho^\prime}{({\rho^\prime}^2 + {y_2^\prime}^2)^{\nicefrac{1}{2}}} \: \d y_2^\prime,
\end{equation*}
where
\begin{equation*}
{x^\prime}^T \theta^\prime = \frac{\rho^\prime}{\rho} \, x^T \theta
\quad \mbox{ and } \quad
{x^\prime}^T \theta_\perp = x^T \theta_\perp.
\end{equation*}
Using this and~\eqref{eq:FDK_local_direction}, we obtain
\begin{equation*}
I(x,\theta) = \frac{{\rho}^2}{(\rho - x^T \theta)^2} \int_{-\rho}^\rho (\Fourier^{-1} F_L)(y_2 - y_2^\prime) \, g(\theta,y_2^\prime \theta_\perp + y_3 e_3) \frac{(\rho^2 + y_3^2)^{\nicefrac{1}{2}}}{(\rho^2 + {y_2^\prime}^2 + y_3^2)^{\nicefrac{1}{2}}} \: \d y_2^\prime,
\end{equation*}
where $y_2,y_3$ are given in~\eqref{eq:FDK_local_coordinates}.
This is the contribution of $\theta^\prime$ in $\Pi(x,\theta)$ to the approximate FBP reconstruction of $f$ at $x$.
According to~\eqref{eq:FBP_approximate_fan_linear} we would have to integrate over the corresponding contributions for all directions in $\Pi(x,\theta)$.
However, the line integral data is not given for these directions and the idea of FDK is to instead integrate over the sources we actually have, i.e., with $\Sphere^\prime = \bigl\{\frac{\rho\theta - y_3 e_3}{\rho^\prime} \bigm| \theta \in \Sphere^1_0\bigr\}$ we compute
\begin{equation*}
f_L(x) \approx \frac{1}{4\pi} \int_{\Sphere^\prime} I(x,\theta) \: \d \theta^\prime = \frac{1}{4\pi} \int_{\Sphere^1_0} I(x,\theta) \, \frac{\rho}{(\rho^2 + y_3^2)^{\nicefrac{1}{2}}} \: \d \theta.
\end{equation*}
This results in the {\em FDK approximate reconstruction formula}
\begin{equation*}
f_L(x) \approx \frac{1}{4\pi} \int_{\Sphere^1_0} \frac{{\rho}^2}{(\rho - x^T \theta)^2} \int_{-\rho}^\rho (\Fourier^{-1} F_L)(y_2 - y_2^\prime) \, g(\theta,y_2^\prime \theta_\perp + y_3 e_3) \frac{\rho}{(\rho^2 + {y_2^\prime}^2 + y_3^2)^{\nicefrac{1}{2}}} \: \d y_2^\prime \, \d \theta
\end{equation*}
with $y_2,y_3$ from~\eqref{eq:FDK_local_coordinates} and a low-pass filter $F_L$ of bandwidth $L > 0$ as in Section~\ref{sec:low-pass_filter}.
It can be implemented similar to the linear fan beam FBP reconstruction formula~\eqref{eq:FBP_approximate_fan_linear} and leads us to an approximate reconstruction algorithm of filtered back projection type for X-ray data in {\em cone beam geometry} with X-ray source on a circle.
This reconstruction algorithm is known as the {\em Feldkamp-Davis-Kress (FDK) algorithm}, which is still used in contemporary CT scanners.

\appendix

\chapter{Mathematical tools}

In this appendix, we state some general mathematical tools, which are used, but not proven during the course.
For a comprehensive treatment of the topics we refer the reader to~\cite{Bredies2018,Stein1971}.

\section{Fourier analysis}

The Fourier transform is a basic tool in the mathematics of computerized tomography and is used extensively in this course.
In this section we define the $n$-dimensional Fourier transform and collect some important properties.
Furthermore, we introduce the convolution product and describe its interplay with the Fourier transform.

\subsection*{The Fourier transform}

We start with the definition of the Fourier transform on the space $\L^1(\R^n)$ of integrable functions.

\begin{definition}[Fourier transform]
The {\em Fourier transform} $\Fourier f$ of a function $f \in \L^1(\R^n)$ is defined as
\begin{equation*}
\Fourier f(\omega) = \int_{\R^n} f(x) \, \e^{-\i x^T \omega} \: \d x
\quad \mbox{ for } \omega \in \R^n.
\end{equation*}
\end{definition}

We remark that the Fourier transform $\Fourier f$ of a function $f \in \L^1(\R^n)$ is well-defined on $\R^n$.
The first important observation is that in this case the Fourier transform $\Fourier f$ is even continuous and, in particular, its point evaluation makes sense.

\begin{lemma}[Riemann-Lebesgue]
For $f \in \L^1(\R^n)$, its Fourier transform $\Fourier f$ is uniformly continuous on $\R^n$ and satisfies
\begin{equation*}
|\Fourier f(\omega)| \tendsto 0
\quad \mbox{ for } \quad
\|\omega\|_{\R^n} \to \infty.
\end{equation*}
\end{lemma}

\begin{proof}
See, for example, \cite[Theorem I.1.2]{Stein1971}.
\end{proof}

Let $\Cont_0(\R^n)$ denote the space of continuous functions vanishing at infinity, i.e.,
\begin{equation*}
\Cont_0(\R^n) = \bigl\{ f \in \Cont(\R^n) \bigm| f(x) \tendsto 0 \mbox{ for } \|x\|_{\R^n} \to \infty \bigr\},
\end{equation*}
which is equipped with the norm $\|\cdot\|_\infty$ given by
\begin{equation*}
\|f\|_\infty = \sup_{x \in \R^n} |f(x)|
\quad \mbox{ for } f \in \Cont_0(\R^n).
\end{equation*}
Then, we have the following continuity result for the Fourier transform.

\begin{theorem}
The Fourier transform $\Fourier: \L^1(\R^n) \to \Cont_0(\R^n)$ is a continuous linear operator with norm $\|\Fourier\| \leq 1$, i.e.,
\begin{equation*}
\|\Fourier f\|_\infty \leq \|f\|_{\L^1(\R^n)}
\quad \forall \, f \in \L^1(\R^n).
\end{equation*}
\end{theorem}

\begin{proof}
See, for example, \cite[Theorem I.1.1]{Stein1971}.
\end{proof}

\bigbreak

We now define the inverse Fourier transform on $\L^1(\R^n)$.

\begin{definition}[Inverse Fourier transform]
For $f \in \L^1(\R^n)$, the {\em inverse Fourier transform} $\Fourier^{-1} f$ is defined as
\begin{equation*}
\Fourier^{-1} f(x) = (2\pi)^{-n} \int_{\R^n} f(\omega) \, \e^{\i x^T \omega} \: \d \omega
\quad \mbox{ for } x \in \R^n.
\end{equation*}
\end{definition}

Although, by the Riemann-Lebesgue lemma, the Fourier transform $\Fourier f$ of  $f \in \L^1(\R^n)$ vanishes at infinity, this does not necessarily imply that $\Fourier f \in \L^1(\R^n)$.
Thus, in order to apply the inverse Fourier transform to $\Fourier f$, we have to assume that $\Fourier f \in \L^1(\R^n)$ is satisfied.
In this case we indeed get the following inverse relationship.

\begin{theorem}[Fourier inversion]
Let $f \in \L^1(\R^n)$ with $\Fourier f \in \L^1(\R^n)$.
Then, the identity
\begin{equation*}
\Fourier^{-1}(\Fourier f)(x) = f(x) = \Fourier(\Fourier^{-1} f)(x)
\end{equation*}
holds for almost all $x \in \R^n$ with equality in every continuity point of $f$.
\end{theorem}

\begin{proof}
See, for example, \cite[Corollary I.1.21]{Stein1971}.
\end{proof}

As a corollary we get the injectivity of the Fourier transform on $\L^1(\R^n)$.

\begin{corollary}[Injectivity of $\Fourier$]
For $f \in \L^1(\R^n)$ we have
\begin{equation*}
\Fourier f = 0
\quad \implies \quad
f = 0,
\end{equation*}
i.e., the Fourier transform $\Fourier$ is injective on $\L^1(\R^n)$.
\end{corollary}

We remark that the inverse Fourier transform $\Fourier^{-1}$ can be expressed in terms of the Fourier transform $\Fourier$ and the parity operator $^\ast: \L^p(\R^n) \to \L^p(\R^n)$ for $1 \leq p \leq \infty$, which is defined as
\begin{equation*}
f^\ast(x) = f(-x)
\quad \mbox{ for } x \in \R^n.
\end{equation*}

\begin{remark}
For $f \in \L^1(\R^n)$ we have $\Fourier^{-1} f = (2\pi)^{-n} \, \Fourier f^\ast = (2\pi)^{-n} \, (\Fourier f)^\ast$.
\end{remark}

We now list some basic properties of the Fourier transform.

\begin{proposition}
\label{prop:Fourier_properties}
For $f \in \L^1(\R^n)$ the following properties hold true.
\begin{enumerate}
\item {\em Translation:} For $y \in \R^n$ we consider the function
\begin{equation*}
g(x) = f(x - y)
\quad \mbox{ for } x \in \R^n.
\end{equation*}
Then,
\begin{equation*}
\Fourier g(x) = \e^{-\i x^T y} \, \Fourier f(x)
\quad \forall \, x \in \R^n.
\end{equation*}

\item {\em Scaling:} For $a > 0$ we consider the function
\begin{equation*}
g(x) = f(ax)
\quad \mbox{ for } x \in \R^n.
\end{equation*}
Then,
\begin{equation*}
\Fourier g(x) = a^{-n} \, \Fourier f(a^{-1} x)
\quad \forall \, x \in \R^n.
\end{equation*}

\item {\em Modulation:} For $y \in \R^n$ we consider the function
\begin{equation*}
g(x) = \e^{\i x^T y} \, f(x)
\quad \mbox{ for } x \in \R^n.
\end{equation*}
Then,
\begin{equation*}
\Fourier g(x) = \Fourier f(x - y)
\quad \forall \, x \in \R^n.
\end{equation*}

\item Let $\alpha \in \N_0^n$ be a multi-index and $\D^\alpha = \frac{\partial^\alpha}{\partial x^\alpha}$.
If $\D^\alpha f$ exists and is in $\L^1(\R^n)$, then
\begin{equation*}
\Fourier (\D^\alpha f) = \i^{|\alpha|} \, x^\alpha \, \Fourier f,
\end{equation*}
whereas, if $x^\alpha f$ is integrable on $\R^n$, then
\begin{equation*}
\Fourier (x^\alpha f) = \i^{|\alpha|} \,  \D^\alpha (\Fourier f).
\end{equation*}
\end{enumerate}
\end{proposition}

\begin{proof}
See, for example, \cite[Theorem 7.8]{Folland1992}.
\end{proof}

Another important property of the Fourier transform is {\em Parseval's identity}.

\begin{theorem}[Parseval's identity]
For $f,g \in \L^1(\R^n)$ we have
\begin{equation*}
\int_{\R^n} \Fourier f(x) \, g(x) \: \d x = \int_{\R^n} f(x) \, \Fourier g(x) \: \d x.
\end{equation*}
\end{theorem}

\begin{proof}
See, for example, \cite[Theorem I.1.15]{Stein1971}.
\end{proof}

The next theorem is the classical Rayleigh-Plancherel theorem, which shows that the Fourier transform preserves the $\L^2$-norm up to a multiplicative constant.

\begin{theorem}[Rayleigh-Plancherel]
Let $f \in \L^1(\R^n)$ and $f \in \L^2(\R^n)$ or $\Fourier f \in \L^2(\R^n)$.
Then, we have $\Fourier f \in \L^2(\R^n)$ or $f \in \L^2(\R^n)$, respectively, and
\begin{equation*}
\|f\|_{\L^2(\R^n)} = (2\pi)^{-\nicefrac{n}{2}} \, \|\Fourier f\|_{\L^2(\R^n)}.
\end{equation*}
More generally, for $f,g \in \L^1(\R^n) \cap \L^2(\R^n)$ we have
\begin{equation*}
(f,g)_{\L^2(\R^n)} = (2\pi)^{-n} \, (\Fourier f,\Fourier g)_{\L^2(\R^n)}.
\end{equation*}
\end{theorem}

\begin{proof}
See, for example, \cite[Theorem I.2.1]{Stein1971}.
\end{proof}

Since $\L^1(\R^n) \cap \L^2(\R^n) \subset \L^2(\R^n)$ is dense, the Rayleigh-Plancherel theorem shows that the Fourier transform can be continuously extended to an operator
\begin{equation*}
\Fourier: \L^2(\R^n) \to \L^2(\R^n),
\end{equation*}
which is an isometry up to a multiplicative constant.
Further, the extended operator $\Fourier$ is bijective on $\L^2(\R^n)$ and its inverse $\Fourier^{-1}$ is the continuous extension of the inverse Fourier transform.
In this course, however, we will not distinguish between the regular Fourier transform and its extension.

Consequently, the Fourier transform and its inverse are now defined on the whole of $\L^2(\R^n)$.
But for $f \in \L^2(\R^n)$, the point evaluation of $\Fourier f$ makes sense only almost everywhere and the Fourier inversion formula holds in $\L^2$-sense.

\begin{corollary}[Fourier inversion in $\L^2(\R^n)$]
For $f \in \L^2(\R^n)$ the Fourier inversion formula
\begin{equation*}
\Fourier^{-1} (\Fourier f) = f = \Fourier (\Fourier^{-1} f)
\end{equation*}
holds in $\L^2$-sense and, in particular, almost everywhere on $\R^n$.
\end{corollary}

We close this paragraph on the Fourier transform with a variant of the classical {\em Paley-Wiener theorem}, which characterizes the Fourier transform of compactly supported functions.

\begin{theorem}[Paley-Wiener]
Let $f \in \L^1(\R^n) \setminus \{0\}$ be compactly supported.
Then, its Fourier transform $\Fourier f$ is analytic and cannot have compact support.
\end{theorem}

\begin{proof}
See, for example, \cite[Theorem 7.23]{Rudin1991} and the identity theorem for analytic functions.
\end{proof}

\subsection*{The convolution product}

We now define the convolution product of functions in $\L^1(\R^n)$ and investigate its interaction with the Fourier transform.

\begin{definition}[Convolution]
The {\em convolution product} $f * g$ of two functions $f,g \in \L^1(\R^n)$ is defined as
\begin{equation*}
(f * g)(x) = \int_{\R^n} f(x-y) \, g(y) \: \d y
\quad \mbox{ for } x \in \R^n.
\end{equation*}
\end{definition}

We remark that the convolution product of $f,g \in \L^1(\R^n)$ exists and is again in $\L^1(\R^n)$.
More generally, we have the following result.

\begin{theorem}[Young's inequality]
Let $f \in \L^p(\R^n)$ and $g \in \L^q(\R^n)$ with $1 \leq p,q \leq \infty$. 
Then, we have $f * g \in \L^r(\R^n)$ with $1 \leq r \leq \infty$ satisfying
\begin{equation*}
\frac{1}{p} + \frac{1}{q} = \frac{1}{r} + 1
\end{equation*}
and {\em Young's inequality}
\begin{equation*}
\|f*g\|_{\L^r(\R^n)} \leq \|f\|_{\L^p(\R^n)} \, \|g\|_{\L^q(\R^n)}
\end{equation*}
holds with equality if $f$ and $g$ are non-negative almost everywhere on $\R^n$.
\end{theorem}

\begin{proof}
See, for example, \cite[Theorem 3.13]{Bredies2018}.
\end{proof}

A special situation occurs if $f \in \L^p(\R^n)$ and $g \in \L^q(\R^n)$ with {\em dual} exponents $1 \leq p,q \leq \infty$, i.e.,
\begin{equation*}
\frac{1}{p} + \frac{1}{q} = 1.
\end{equation*}

\begin{theorem}
Let $f \in \L^p(\R^n)$ and $g \in \L^q(\R^n)$ with $1 \leq p,q \leq \infty$ satisfying
\begin{equation*}
\frac{1}{p} + \frac{1}{q} = 1.
\end{equation*}
Then, $f * g$ is bounded and continuous on $\R^n$, i.e., $f * g \in \Cont_b(\R^n)$, where
\begin{equation*}
\Cont_b(\R^n) = \bigl\{ f \in \Cont(\R^n) \bigm| \|f\|_\infty < \infty \bigr\}.
\end{equation*}
If we further have $1 < p,q < \infty$, then $f * g$ vanishes at infinity, i.e., $f * g \in \Cont_0(\R^n)$.
\end{theorem}

\begin{proof}
See, for example, \cite[Theorem 3.14]{Bredies2018}.
\end{proof}

We now list some basic properties of the convolution product.

\begin{proposition}
The convolution product satisfies the following properties.
\begin{enumerate}
\item {\em Commutativity:}
\begin{equation*}
f * g = g * f
\quad \forall \, f,g \in \L^1(\R^n)
\end{equation*}

\item {\em Linearity:}
\begin{equation*}
f * (\alpha \, g + \beta \, h) = \alpha \, (f * g) + \beta \, (f * h)
\quad \forall \, \alpha,\beta \in \R, \, f,g,h \in \L^1(\R^n)
\end{equation*}

\item {\em Integration:}
\begin{equation*}
\int_{\R^n} (f * g)(x) \: \d x = \bigg(\int_{\R^n} f(x) \: \d x\bigg) \, \bigg(\int_{\R^n} g(x) \: \d x\bigg)
\quad \forall \, f,g \in \L^1(\R^n)
\end{equation*}

\item {\em Translation:} For $f \in \L^1(\R^n)$ and $a \in \R^n$ we consider the function
\begin{equation*}
f_a(x) = f(x-a)
\quad \mbox{ for } x \in \R^n.
\end{equation*}
Then,
\begin{equation*}
f_a * g = (f * g)_a
\quad \forall \, g \in \L^1(\R^n).
\end{equation*}

\item Let $f \in \L^1(\R^n)$ and let $g \in \Cont^k(\R^n)$, $k \in \N$, be bounded such that its derivatives $\D^\alpha g$ are also bounded for all multi-indices $\alpha \in \N_0^n$ with $|\alpha| \leq k$.
Then, we have $f * g \in \Cont^k(\R^n)$ and
\begin{equation*}
\D^\alpha (f * g) = f * \D^\alpha g.
\end{equation*}
\end{enumerate}
\end{proposition}

\begin{proof}
See, for example, \cite[Chapter 3.3]{Bredies2018}.
\end{proof}

We finish this section by stating the most important property of the convolution product, which is given by the classical {\em Fourier convolution theorem} and describes the interaction between the convolution product and the Fourier transform.

\begin{theorem}[Fourier convolution theorem]
Let $f,g \in \L^1(\R^n)$ be given functions.
Then, we have
\begin{equation*}
\Fourier (f * g) = \Fourier f \cdot \Fourier g
\end{equation*}
and
\begin{equation*}
\Fourier^{-1} (f * g) = (2\pi)^n \, \Fourier^{-1} f \cdot \Fourier^{-1} g.
\end{equation*}
Additionally, if $\Fourier f, \Fourier g \in \L^1(\R^n)$, then
\begin{equation*}
\Fourier (f \cdot g) = (2\pi)^{-n} \, (\Fourier f) * (\Fourier g).
\end{equation*}
\end{theorem}

\begin{proof}
See, for example, \cite[Theorem I.1.4]{Stein1971}.
\end{proof}

\subsection*{Distributions}

Distributions or, to be more precise, {\em tempered} distributions play an important role in the definition of Sobolev spaces of fractional order.
Thus, in this section we introduce distributions as generalized functions and extend the Fourier transform to the space of tempered distributions.

\smallbreak

The space of distributions is given by the topological dual of the space of test functions, which is defined as follows.

\begin{definition}[Space of test functions]
Let $\Omega \seq \R^n$ be a domain in $\R^n$.
Then, the space of {\em test functions} on $\Omega$ is defined as
\begin{equation*}
\Test(\Omega) = \bigl\{ f \in \Cont^\infty(\Omega) \bigm| \supp(f) \seq \Omega \mbox{ compact} \bigr\}.
\end{equation*}
\end{definition}

The following lemma explains the expression 'test function'.

\begin{lemma}[Fundamental lemma of variational calculus]
Let $\Omega \seq \R^n$ be a domain in $\R^n$ and $f \in \L_\loc^1(\Omega)$ be locally integrable.
Then, we have
\begin{equation*}
f \equiv 0
\quad \mbox{ a.e.\ on } \Omega
\qquad \iff \qquad
\int_\Omega f(x) \, \phi(x) \: \d x = 0
\quad \forall \, \phi \in \Test(\Omega).
\end{equation*}
\end{lemma}

\begin{proof}
See, for example, \cite[Lemma 2.75]{Bredies2018}.
\end{proof}

Calculating the integral $\int_\Omega f(x) \, \phi(x) \: \d x$ is also called {\em testing} the function $f \in \L_\loc^1(\Omega)$ with $\phi \in \Test(\Omega)$.
Thus, the fundamental lemma of variational calculus states that $f \in \L_\loc^1(\Omega)$ is almost everywhere uniquely determined by testing with all functions $\phi \in \Test(\Omega)$.

\begin{example}
The function $f: \R^n \to \R$ with
\begin{equation*}
f(x) = \begin{cases}
\exp\Big(-\frac{1}{1 - \|x\|_{\R^n}^2}\Big) & \text{for } \|x\|_{\R^n} < 1 \\
0 & \text{for } \|x\|_{\R^n} \geq 1
\end{cases}
\end{equation*}
is a test function on $\R^n$, i.e., it satisfies $f \in \Test(\R^n)$.
\end{example}

We now introduce the notion of a distribution.

\begin{definition}[Distribution]
Let $\Omega \seq \R^n$ be a domain in $\R^n$.
The topological dual space of $\Test(\Omega)$ with respect to the natural topology, denoted by $\Test'(\Omega)$, is called the space of {\em distributions}.
\end{definition}

We note that, for a domain $\Omega \seq \R^n$, each function $f \in \L_\loc^1(\Omega)$ induces a distribution $T_f \in \Test'(\Omega)$ via
\begin{equation*}
T_f(\phi) = \int_\Omega f(x) \, \phi(x) \: \d x
\quad \mbox{ for } \phi \in \Test(\Omega).
\end{equation*}
In this sense we have $\L_\loc^1(\Omega) \subset \Test'(\Omega)$ and distributions of the form $T_f$ are called {\em regular}.
Because of the relation between $f$ and $T_f$, distributions are also called {\em generalized functions}.

However, there also exist distributions that are not regular.
One example is the well-known {\em Dirac distribution} $\delta \in \Test'(\R^n)$ with
\begin{equation*}
\delta(f) = f(0)
\quad \mbox{ for } f \in \Test(\R^n)
\end{equation*}
or, for fixed $x_0 \in \R^n$, the {\em shifted Dirac distribution} $\delta_{x_0} \in \Test'(\R^n)$ with
\begin{equation*}
\delta_{x_0}(f) = f(x_0)
\quad \mbox{ for } f \in \Test(\R^n).
\end{equation*}

In what follows, we denote the action of a distribution $T \in \Test'(\Omega)$ on a test function $\phi \in \Test(\Omega)$ by the {\em duality pairing}
\begin{equation*}
\langle T,\phi \rangle = T(\phi).
\end{equation*}
With this we define the derivative of a distribution as follows.

\begin{definition}[Derivative of a distribution]
Let $\Omega \seq \R^n$ be a domain and $T \in \Test'(\Omega)$.
For $\alpha \in \N_0^n$, we define the {\em derivative} $\D^\alpha T \in \Test'(\Omega)$ of $T$ via
\begin{equation*}
\langle \D^\alpha T,\phi \rangle = (-1)^{|\alpha|} \, \langle T,\D^\alpha \phi \rangle
\quad \mbox{ for } \phi \in \Test(\Omega).
\end{equation*}
\end{definition}

We use the same technique to define the multiplication of a distribution with a $\Cont^\infty$-function.
To this end, we note that the product $g \cdot \phi$ of $g \in \Cont^\infty(\Omega)$ and $\phi \in \Test(\Omega)$ is again in $\Test(\Omega)$.

\begin{definition}[Multiplication with a $\Cont^\infty$-function]
Let $\Omega \seq \R^n$ be a domain.
Further, let $T \in \Test'(\Omega)$ be a distribution and $g \in \Cont^\infty(\Omega)$.
Then, the distribution $g \cdot T \in \Test'(\Omega)$ is defined as
\begin{equation*}
\langle g \cdot T,\phi \rangle = \langle T,g \cdot \phi \rangle
\quad \mbox{ for } \phi \in \Test(\Omega).
\end{equation*}
\end{definition}

We would like to use this technique to also define the Fourier transform of a distribution.
However, the Fourier transform of a non-trivial test function is not a test function, since it cannot have compact support due to the Paley-Wiener theorem.
To resolve this problem, we have to restrict ourselves to a smaller subspace of distributions, the so called {\em tempered distributions}.
This space is given by the topological dual of the so called {\em Schwartz space} of rapidly decreasing functions, which is defined as follows.

\begin{definition}[Schwartz space]
The {\em Schwartz space} $\Schwartz(\R^n)$ of rapidly decaying functions is defined as
\begin{equation*}
\Schwartz(\R^n) = \bigl\{ f \in \Cont^\infty(\R^n) \bigm| \forall \, \alpha,\beta \in \N_0^n: \; |f|_{\alpha,\beta} < \infty \bigr\},
\end{equation*}
where
\begin{equation*}
|f|_{\alpha,\beta} = \sup_{x \in \R^n} |x^\alpha \, \D^\beta f(x)|
\quad \mbox{ for } \alpha,\beta \in \N_0^n.
\end{equation*}
\end{definition}

The space of Schwartz functions plays a central role in the theory of Fourier transforms.

\begin{theorem}
The Fourier transform $\Fourier: \Schwartz(\R^n) \to \Schwartz(\R^n)$ is an automorphism of $\Schwartz(\R^n)$.
In particular,
\begin{equation*}
\Fourier^{-1} (\Fourier f) = f = \Fourier (\Fourier^{-1} f)
\quad \forall \, f \in \Schwartz(\R^n).
\end{equation*}
\end{theorem}

\begin{proof}
See, for example, \cite[Theorem 4.15]{Bredies2018}.
\end{proof}

Since $\Test(\R^n) \subset \Schwartz(\R^n)$, we have $\Schwartz'(\R^n) \subset \Test'(\R^n)$ and the dual space of $\Schwartz(\R^n)$ consists of a special class of distributions.
These are called {\em tempered distributions}.

\begin{definition}[Tempered distribution]
The topological dual space of $\Schwartz(\R^n)$ with respect to the natural topology, denoted by $\Schwartz'(\R^n)$, is called the space of {\em tempered distributions}.
\end{definition}

Since the Fourier transform of a Schwartz function is again a Schwartz function, we can now define the Fourier transform of tempered distributions.

\begin{definition}[Fourier transform of tempered distributions]
Let $T \in \Schwartz'(\R^n)$ be tempered.
Then, its Fourier transform $\Fourier T \in \Schwartz'(\R^n)$ is defined via
\begin{equation*}
\langle \Fourier T,f \rangle = \langle T,\Fourier f \rangle
\quad \mbox{ for } f \in \Schwartz(\R^n).
\end{equation*}
Analogously, its inverse Fourier transform $\Fourier^{-1} T \in \Schwartz'(\R^n)$ is given by
\begin{equation*}
\langle \Fourier^{-1} T,f \rangle = \langle T,\Fourier^{-1} f \rangle
\quad \mbox{ for } f \in \Schwartz(\R^n).
\end{equation*}
\end{definition}

We remark that, if $T \in \Schwartz'(\R^n)$ is regular and given by $T = T_f$ for some function $f \in \L^1(\R^n)$, we have
\begin{equation*}
\Fourier T_f = T_{\Fourier f}
\end{equation*}
due to Parseval's identity.
Hence, the definition of the classical and the distributional Fourier transform coincide on $\L^1(\R^n)$.

\begin{remark}
We have $\L^p(\R^n) \subset \Schwartz'(\R^n)$ for all $1 \leq p \leq \infty$ in the sense that the functional $T_f: \Schwartz(\R^n) \to \R$,
\begin{equation*}
\langle T_f,\phi \rangle = \int_{\R^n} f(x) \, \phi(x) \: \d x
\quad \mbox{ for } \phi \in \Schwartz(\R^n),
\end{equation*}
is a tempered distribution.
This observation implies that the Fourier transform is now defined for all $\L^p$-spaces.
However, the Fourier transform of $f \in \L^p(\R^n)$ with $p > 2$ is in general not a function, but only a distribution, in contrast to the case of $p \leq 2$.
\end{remark}

Like the Schwartz space $\Schwartz(\R^n)$, also the space of tempered distributions $\Schwartz'(\R^n)$ plays a central role in Fourier analysis.

\begin{theorem}
The Fourier transform $\Fourier: \Schwartz'(\R^n) \to \Schwartz'(\R^n)$ is an automorphism of $\Schwartz'(\R^n)$ with respect to the weak topology.
In particular, we have
\begin{equation*}
\Fourier^{-1} (\Fourier f) = f = \Fourier (\Fourier^{-1} f)
\quad \forall \, f \in \Schwartz'(\R^n).
\end{equation*}
\end{theorem}

\begin{proof}
See, for example, \cite[Theorem 4.25]{Bredies2018}.
\end{proof}

Many properties of the regular Fourier transform carry over to the distributional Fourier transform.
As an example, we restate that $\Fourier$ translates differentiation into multiplication.
To this end, we first remark that for $T \in \Schwartz'(\R^n)$ and $\alpha \in \N_0^n$ the distributional derivative $\D^\alpha T \in \Schwartz'(\R^n)$ is again tempered.
Further, the multiplication with a function $f \in \Cont^\infty(\R^n)$ of at most polynomial growth is well-defined via
\begin{equation*}
\langle f \cdot T,\phi \rangle = \langle T,f \cdot \phi \rangle
\quad \mbox{ for } \phi \in \Schwartz(\R^n).
\end{equation*}

\begin{proposition}
For $T \in \Schwartz'(\R^n)$ and $\alpha \in \N_0^n$, we have
\begin{equation*}
\Fourier (\D^\alpha T) = \i^{|\alpha|} \, x^\alpha \, \Fourier T.
\end{equation*}
\end{proposition}

\begin{proof}
See, for example, \cite[Theorem 7.15]{Rudin1991}.
\end{proof}

\section{Sobolev spaces}

Sobolev spaces play an important role in the understanding of the ill-posedness of the CT reconstruction problem.
For this reason, we now define Sobolev spaces of fractional order and list some basic properties.
This is based on a characterization of regular Sobolev spaces by means of the Fourier transform.

\smallbreak

We begin with the standard definition of Sobolev spaces.
To this end, let us recall that for a domain $\Omega \seq \R^n$, a multi-index $\alpha \in \N_0^n$ and a distribution $T \in \Test'(\Omega)$ the derivative $\D^\alpha f \in \Test'(\Omega)$ is defined via
\begin{equation*}
\langle \D^\alpha T,\phi \rangle = (-1)^{|\alpha|} \, \langle T,\D^\alpha \phi \rangle
\quad \mbox{ for } \phi \in \Test(\Omega).
\end{equation*}
In general, the distributional derivative $\D^\alpha f$ of a function $f \in \L_\loc^1(\Omega) \subset \Test'(\Omega)$ is not a function itself.
But in case it is, $\D^\alpha f$ is called {\em weak derivative} of $f$.

\begin{definition}[Weak derivative]
Let $\Omega \seq \R^n$ be a domain, $f \in \L_\loc^1(\Omega)$ be locally integrable and $\alpha \in \N_0^n$.
If there exists a function $g \in \L_\loc^1(\Omega)$ with
\begin{equation*}
\int_\Omega g(x) \, \phi(x) \: \d x = (-1)^{|\alpha|} \int_\Omega f(x) \, \D^\alpha \phi(x) \: \d x
\quad \forall \, \phi \in \Test(\Omega),
\end{equation*}
then $f$ is called {\em weakly differentiable} on $\Omega$ with weak derivative $\D^\alpha f = g$.
If the weak derivatives $\D^\alpha f \in \L_\loc^1(\Omega)$ exist for all $|\alpha| \leq k$ with $k \in \N$, then $f$ is called {\em $k$-times weakly differentiable}.
\end{definition}

\begin{remark}
Weak derivatives are uniquely determined almost everywhere on $\Omega$ according to the fundamental lemma of variational calculus.
\end{remark}

Now, the common Sobolev spaces are defined as spaces of functions whose weak derivatives belong to certain $\L^p$-spaces.

\begin{definition}[Sobolev space]
Let $\Omega \seq \R^n$ be a domain, $1 \leq p \leq \infty$ and $k \in \N_0$.
Then, the {\em Sobolev space} $\H^{k,p}(\Omega)$ is defined as
\begin{equation*}
\H^{k,p}(\Omega) = \bigl\{ f \in \L^p(\Omega) \bigm| \forall \, |\alpha| \leq k: \; \D^\alpha f \in \L^p(\Omega) \bigr\}
\end{equation*}
and equipped with the {\em Sobolev norm}
\begin{equation*}
\|f\|_{\H^{k,p}(\Omega)} = \begin{cases}
\Big(\sum_{|\alpha| \leq k} \|\D^\alpha f\|_{\L^p(\Omega)}^p\Big)^{\nicefrac{1}{p}} & \text{for } p < \infty \\
\max_{|\alpha| \leq k} \|\D^\alpha f\|_{\L^\infty(\Omega)} & \text{for } p = \infty.
\end{cases}
\end{equation*}
\end{definition}

\begin{remark}
For $p = 2$ we simply write $\H^k(\Omega) \equiv \H^{k,2}(\Omega)$ and these spaces are Hilbert spaces with the inner product
\begin{equation*}
(f,g)_{\H^k(\Omega)} = \sum_{|\alpha| \leq k} (\D^\alpha f,\D^\alpha g)_{\L^2(\Omega)}
\quad \mbox{ for } f,g \in \H^k(\Omega).
\end{equation*}
\end{remark}

For $p = 2$, the Fourier transform $\Fourier$ translates weak differentiation into multiplication and vice versa.

\begin{lemma}
Let $f \in \L^2(\R^n)$ and $\alpha \in \N_0^n$ so that the weak derivative $\D^\alpha f$ is also in $\L^2(\R^n)$.
Then, we have
\begin{equation*}
\Fourier (\D^\alpha f) = \i^{|\alpha|} \, x^\alpha \, \Fourier f
\end{equation*}
and, if $x^\alpha f \in \L^2(\R^n)$,
\begin{equation*}
\Fourier (x^\alpha f) = \i^{|\alpha|} \, \D^\alpha \Fourier f.
\end{equation*}
\end{lemma}

\begin{proof}
See, for example, \cite[Lemma 4.28]{Bredies2018}.
\end{proof}

This lemma yields the following characterization of regular Sobolev spaces by means of the Fourier transform.

\begin{theorem}[Characterization of $\H^k(\R^n)$]
For $k \in \N$ we have
\begin{equation*}
f \in \H^k(\R^n)
\quad \iff \quad
\int_{\R^n} (1+\|\omega\|_{\R^n}^2)^k \, |\Fourier f(\omega)|^2 \: \d \omega < \infty.
\end{equation*}
\end{theorem}

\begin{proof}
See, for example, \cite[Theorem 4.29]{Bredies2018}.
\end{proof}

Observe that the above theorem relates the weak differentiability of a function $f \in \L^2(\R^n)$ to the decay properties of its Fourier transform $\Fourier f$.
This characterization does not only give a useful tool to investigate the smoothness of a function but also offers a possibility to generalize the definition of Sobolev spaces $\H^k(\R^n)$ of integer order to spaces $\H^\alpha(\R^n)$ of arbitrary smoothness order $\alpha \in \R$.
However, if $\alpha < 0$, we have to enlarge the basic set from $\L^2(\R^n)$ to the space of tempered distributions $\Schwartz'(\R^n)$.

\begin{definition}[Sobolev space of fractional order]
The {\em Sobolev space $\H^\alpha(\R^n)$ of fractional order $\alpha \in \R$} is defined as
\begin{equation*}
\H^\alpha(\R^n) = \bigl\{ f \in \Schwartz'(\R^n) \bigm| \|f\|_{\H^\alpha(\R^n)} < \infty \bigr\},
\end{equation*}
where the Sobolev norm $\|\cdot\|_{\H^\alpha(\R^n)}$ is given by
\begin{equation*}
\|f\|_{\H^\alpha(\R^n)}^2 = \int_{\R^n} (1+\|\omega\|_{\R^n}^2)^\alpha \, |\Fourier f(\omega)|^2 \: \d \omega.
\end{equation*}
Further, for an open subset $\Omega \seq \R^n$, we define the Sobolev space $\H_0^\alpha(\Omega)$ by
\begin{equation*}
\H_0^\alpha(\Omega) = \bigl\{ f \in \H^\alpha(\R^n) \bigm| \supp(f) \seq \overline{\Omega} \bigr\},
\end{equation*}
where the support $\supp(f)$ of a tempered distribution $f \in \Schwartz'(\R^n)$ is defined as the complement of the largest open set $U \subset \R^n$ for which $\langle f,\phi \rangle = 0$ for all $\phi \in \Schwartz(\R^n)$ with $\supp(\phi) \subset U$.
\end{definition}

For $\alpha \in \N_0$, the above theorem shows that the space $\H^\alpha(\R^n)$ consists of those functions whose (distributional) derivatives up to order $\alpha$ are square-integrable.
Therefore, the definition of fractional Sobolev spaces is compatible with the definition of classical Sobolev spaces.
In particular, for $\alpha = 0$ we simply have
\begin{equation*}
\H^0(\R^n) = \L^2(\R^n).
\end{equation*}
By defining the equivalent Sobolev norms $\|\cdot\|_\alpha$ on $\H^\alpha(\R^n)$ for $\alpha \in \R$ via
\begin{equation*}
\|f\|_\alpha^2 = (2\pi)^{-n} \int_{\R^n} (1+\|\omega\|_{\R^n}^2)^\alpha \, |\Fourier f(\omega)|^2 \: \d \omega
\quad \mbox{ for } f \in \H^\alpha(\R^n),
\end{equation*}
we further obtain
\begin{equation*}
\|\cdot\|_0 = \|\cdot\|_{\L^2(\R^n)}
\end{equation*}
according to the Rayleigh-Plancherel theorem.

\bigskip

We close this section with some final remarks on Sobolev spaces.
\begin{enumerate}
\item The Sobolev space $\H^\alpha(\R^n)$ is a Hilbert space with the inner product
\begin{equation*}
(f,g)_{\H^\alpha(\R^n)} = \int_{\R^n} (1+\|\omega\|_{\R^n}^2)^\alpha \, \Fourier f(\omega) \, \overline{\Fourier g}(\omega) \: \d \omega
\quad \mbox{ for } f,g \in \H^\alpha(\R^n).
\end{equation*}

\item For $\alpha < \beta$ we have $\H^\beta(\R^n) \subset \H^\alpha(\R^n)$ and, in particular,
\begin{equation*}
\H^\alpha(\R^n) \subset \L^2(\R^n)
\quad \forall \, \alpha > 0.
\end{equation*}
Thus, for $\alpha > 0$, the Sobolev space $\H^\alpha(\R^n)$ can equivalently be defined as
\begin{equation*}
\H^\alpha(\R^n) = \bigl\{ f \in \L^2(\R^n) \bigm| \|f\|_\alpha < \infty \bigr\}.
\end{equation*}

\item The dual space of $\H^\alpha(\R^n)$ is topologically isomorphic to $\H^{-\alpha}(\R^n)$.
\end{enumerate}

\section{Elements of measure theory}

We finally list some results from measure theory that are used throughout the course.
To this end, recall that a measure space is a triple $(X,\A,\mu)$, where $X$ is a set, $\A$ is a $\sigma$-algebra on $X$ collecting the measurable subsets of $X$ and $\mu$ is a measure on $(X,\A)$.

\smallskip

We start with the following variant of Fubini's theorem.

\begin{theorem}[Fubini]
Let $(X,\A,\mu)$ and $(Y,\E,\nu)$ be $\sigma$-finite measure spaces.
Furthermore, let $f: X \times Y \to \C$ be $(\mu \times \nu)$-measurable.
Define the functions
\begin{equation*}
\varphi(x) = \int_{Y} f(x,\cdot) \: \d \nu
\enspace \mbox{ for } x \in X
\quad \mbox{ and } \quad
\psi(y) = \int_{X} f(\cdot,y) \: \d \mu
\enspace \mbox{ for } y \in Y
\end{equation*}
as well as
\begin{equation*}
\Phi(x) = \int_{Y} |f(x,\cdot)| \: \d \nu
\enspace \mbox{ for } x \in X
\quad \mbox{ and } \quad
\Psi(y) = \int_{X} |f(\cdot,y)| \: \d \mu
\enspace \mbox{ for } y \in Y.
\end{equation*}
\begin{itemize}
\item[(i)] If $0 \leq f \leq \infty$, then $\varphi$ is $\mu$-measurable, $\psi$ is $\nu$-measurable and
\begin{equation*}
\int_{X} \varphi \: \d \mu = \int_{X \times Y} f  \: \d (\mu \times \nu) = \int_{Y} \psi \: \d \nu.
\end{equation*}
\item[(ii)] If $f$ is real-valued and if
\begin{equation*}
\int_{X} \Phi \: \d \mu < \infty
\quad \mbox{ or } \quad
\int_{Y} \Psi \: \d \nu < \infty,
\end{equation*}
then $f \in \L^1(\mu \times \nu)$.
\item[(iii)] If $f \in \L^1(\mu \times \nu)$, then we have $f(x,\cdot) \in \L^1(\nu)$ for almost all $x \in X$, $f(\cdot,y) \in \L^1(\mu)$ for almost all $y \in Y$, $\varphi \in \L^1(\mu)$, $\psi \in \L^1(\nu)$, and
\begin{equation*}
\int_{X} \varphi \: \d \mu = \int_{X \times Y} f  \: \d (\mu \times \nu) = \int_{Y} \psi \: \d \nu.
\end{equation*}
\end{itemize}
\end{theorem}

\begin{proof}
See, for example, \cite[Theorem 8.8]{Rudin1987}.
\end{proof}

We continue with the change of variables formula for the Lebesgue measure $\lambda$.

\begin{theorem}[Change of variables]
Let $\Omega, \Sigma \subset \R^d$ be non-empty, open and $\varphi: \Sigma \to \Omega$ be a diffeomorphism, i.e., $\varphi$ is invertible and $\varphi$ as well as $\varphi^{-1}$ are continuously differentiable.
Furthermore, let $f: \Omega \to \C$ be Lebesgue-measurable.
If $0 \leq f \leq \infty$ or $f \in \L^1(\Omega)$, then
\begin{equation*}
\int_{\Omega} f \: \d \lambda = \int_{\Sigma} |\det(J\varphi))| \, (f \circ \varphi) \: \d \lambda,
\end{equation*}
where $J\varphi: \Sigma \to \R^{d \times d}$ denotes the Jacobian of $\varphi$.
\end{theorem}

\begin{proof}
See, for example, \cite[Theorem 2.69]{Bredies2018}.
\end{proof}

We close this section with Lebesgue's theorem on dominated convergence.

\begin{theorem}[Dominated convergence]
Let $(X,\A,\mu)$ be a measure space and suppose that $(f_n)_{n\in\N}$ is a sequence of complex-valued measurable functions such that
\begin{equation*}
f(x) = \lim_{n \to \infty} f_n(x)
\end{equation*}
exists for almost all $x \in X$. If there is a function $g \in \L^1(\mu)$ such that
\begin{equation*}
|f_n(x)| \leq g(x)
\quad \, \forall n \in \N, ~ x \in X,
\end{equation*}
then $f \in \L^1(\mu)$ and
\begin{equation*}
\lim_{n \to \infty} \int_X |f_n - f| \: \d \mu = 0
\quad \mbox{ and } \quad
\lim_{n \to \infty} \int_X f_n \: \d \mu = \int_X f \: \d \mu.
\end{equation*}
\end{theorem}

\begin{proof}
See, for example, \cite[Theorem 1.34]{Rudin1987}.
\end{proof}

\backmatter



\defbibnote{bibhead}{\markboth{Bibliography}{Bibliography}}
\printbibliography[heading=bibintoc,prenote=bibhead]

\end{document}